\documentclass{amsart}
\title[Maximal surfaces with arbitrary genus]{%
  New maximal surfaces in Minkowski 3-space
\\
  with arbitrary genus
\\
  and their cousins in de Sitter 3-space
}

\date{October 15, 2009}
\usepackage{graphicx}
\usepackage{enumerate}
\usepackage{amssymb}
\usepackage{amsthm}
\theoremstyle{plain}
 \newtheorem{claim}{Claim}
 \newtheorem{theorem}{Theorem}[section]
 \newtheorem*{theorem*}{Theorem}
 \newtheorem{introtheorem}{Theorem}
   
 \newtheorem{problem}{Problem}
 \newtheorem*{lemma*}{Lemma}
 \newtheorem{proposition}[theorem]{Proposition}
 \newtheorem{fact}[theorem]{Fact}
 \newtheorem*{fact*}{Fact}

 \newtheorem{lemma}[theorem]{Lemma}
 \newtheorem{corollary}[theorem]{Corollary}
 \theoremstyle{remark}
 \newtheorem{introdefinition}{Definition}
   
 \newtheorem{definition}[theorem]{Definition}
 \newtheorem{remark}[theorem]{Remark}
 \newtheorem*{remark*}{Remark}
 \newtheorem*{problem*}{Problem}
 \newtheorem*{acknowledgements}{Acknowledgements}
 
 \newtheorem{introexample}{Example}
\numberwithin{equation}{section}

\newcommand{\Z}{\boldsymbol{Z}}

\newcommand{\R}{\boldsymbol{R}}
\newcommand{\C}{\boldsymbol{C}}

\newcommand{\Herm}{\operatorname{Herm}}

\newcommand{\SL}{\operatorname{SL}}

\renewcommand{\sl}{\operatorname{\mathfrak{sl}}}
\newcommand{\PSL}{\operatorname{PSL}}
\newcommand{\SU}{\operatorname{SU}}

\newcommand{\trace}{\operatorname{trace}}
\newcommand{\Ord}{\operatornamewithlimits{Ord}}
\newcommand{\Res}{\operatornamewithlimits{Res}}
\newcommand{\id}{e_0}
\newcommand{\identity}{\operatorname{id}}

\newcommand{\inner}[2]{\left\langle{#1},{#2}\right\rangle}

\renewcommand{\Re}{\operatorname{Re}}
\renewcommand{\Im}{\operatorname{Im}}

\renewcommand{\phi}{\varphi}
\renewcommand{\epsilon}{\varepsilon}

\newcommand{\cmcone}{\mbox{CMC-$1$}}
\def\transpose#1{{#1}^T}
\newcommand{\secondff}{\operatorname{I\!I}}

\author[S. Fujimori]{Shoichi Fujimori}
\address{%
   Department of Mathematics\\
   Fukuoka University of Education\\
   Munakata, Fukuoka 811-4192\\
   Japan
}
\email{fujimori@fukuoka-edu.ac.jp}
\author[W. Rossman]{Wayne Rossman}
\address{%
   Department of Mathematics\\
   Faculty of Science\\
   Kobe University\\
   Rokko, Kobe 657-8501\\
   Japan
}
\email{wayne@math.kobe-u.ac.jp}

\author[M. Umehara]{Masaaki Umehara}
\address{%
   Department of Mathematics\\
   Graduate School of Science\\
   Osaka University\\
   Toyonaka, Osaka 560-0043\\
   Japan
}
\email{umehara@math.sci.osaka-u.ac.jp}

\author[K. Yamada]{Kotaro Yamada}
\address{%
   Department of Mathematics\\
   Tokyo Institute of Technology\\
   O-okayama, Meguro, Tokyo 152-8551\\
   Japan
}
\email{kotaro@math.titech.ac.jp}

\author[S.-D. Yang]{Seong-Deog Yang}
\address{%
   Department of Mathematics\\
   Korea University\\
   Seoul 136-701\\
   Korea
}
\email{sdyang@korea.ac.kr}
\dedicatory{Dedicated to the memory of Professor Katsumi Nomizu}

\begin{document}
\begin{abstract}
 Until now, the only known maximal surfaces in Minkowski
 $3$-space of finite topology with compact singular set and without
 branch points were either genus zero or genus one,
 or came from a correspondence with minimal surfaces
 in Euclidean $3$-space 
 given by the third and fourth authors in a previous paper.
 In this paper, we discuss singularities and several global properties
 of maximal surfaces, and give explicit examples of such surfaces of
 arbitrary genus.
 When the genus is one, our examples are embedded outside a compact set.
 Moreover, we deform such examples to \cmcone{} faces 
 (mean curvature one surfaces with admissible singularities
 in de Sitter $3$-space) and obtain ``cousins'' of 
 those maximal surfaces.
 
 Cone-like singular points on maximal surfaces are very important,
 although they are not stable under perturbations of maximal surfaces.
 It is interesting to ask if cone-like  singular points can appear on 
 a maximal surface having other kinds of singularities.
 Until now, no such examples were known.
 We also construct a family of complete maximal surfaces with two
 complete ends and with both cone-like singular points and cuspidal
 edges.
\end{abstract}
\maketitle

\section*{Introduction}
\begingroup
\renewcommand{\theequation}{\arabic{equation}}
Maximal surfaces in the Minkowski $3$-space $\R^3_1$ arise as
solutions of the variational problem of locally maximizing the area
among spacelike surfaces. 
By definition, they have everywhere vanishing mean curvature. 
Like the case of minimal surfaces in Euclidean $3$-space, maximal
surfaces possess a Weierstrass-type representation formula \cite{K}.

The most significant difference between minimal and maximal surfaces is the
fact that the only complete spacelike maximal surfaces are planes
\cite{C,CY},
which is probably the main reason why people have not paid much attention
to maximal surfaces.
If we allow some sorts of singular points for maximal surfaces, 
however, the situation changes. 
Osamu Kobayshi \cite{K2} investigated cone-like singular points on
maximal surfaces. 
After that, many interesting examples with cone-like singular points have
been found and studied by F. J. L\'opez, R. L\'opez, and Souam 
\cite{LLS}, 
Fern\'andez and F. J. L\'opez \cite{FL},
and Fern\'andez, F. J. L\'opez and Souam \cite{FLS},
Fern\'andez \cite{Fer} and others.

On the other hand, for the study of more general singularities,
Estudillo and Romero \cite{ER} initially defined a class of maximal
surfaces with singular points of more general type, 
and investigated criteria for such surfaces to be planes.
Recently, Imaizumi  \cite{Imaizumi} studied the asymptotic behavior of
maximal surfaces, and Imaizumi-Kato \cite{IK} gave a classification of
maximal surfaces of genus zero with at most three embedded ends.
In \cite{UY}, the third and forth authors showed that if admissible
singular points are included, then there is an interesting class of
objects called {\it maxfaces}.
In fact, the three surface classes 
\begin{itemize}
 \item non-branched generalized maximal surfaces in the sense of \cite{ER},
 \item non-branched generalized maximal maps in the sense of
       \cite{IK}, and
 \item maxfaces in the sense of \cite{UY}
\end{itemize}
are all the same class of maximal surfaces.
So in this paper, we shall call such a class of surfaces {\it maxfaces}.
\begin{figure}
\begin{center}
\begin{tabular}{c@{\hspace{4em}}c}
 \includegraphics[width=.30\linewidth]{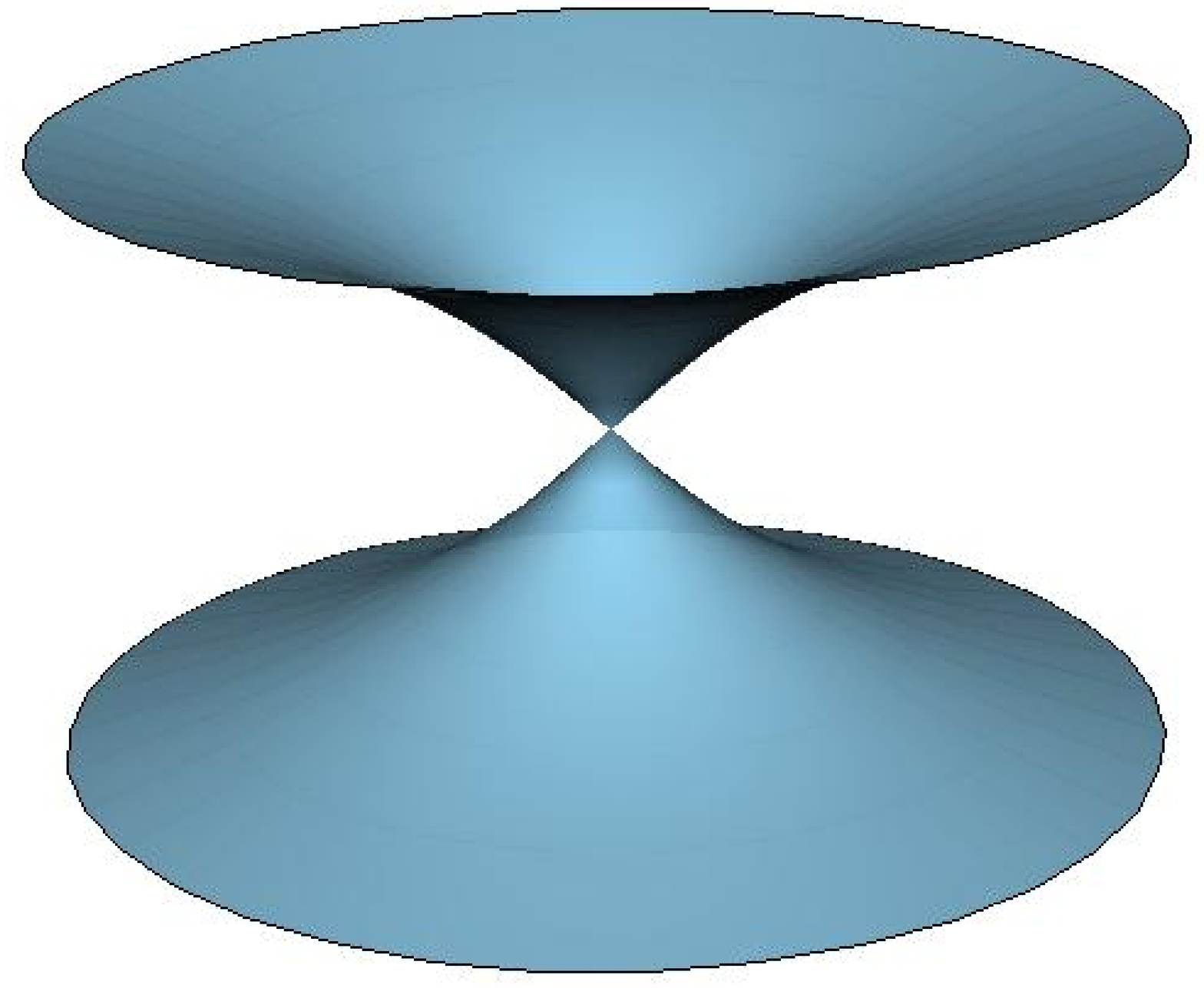} &
 \includegraphics[width=.20\linewidth]{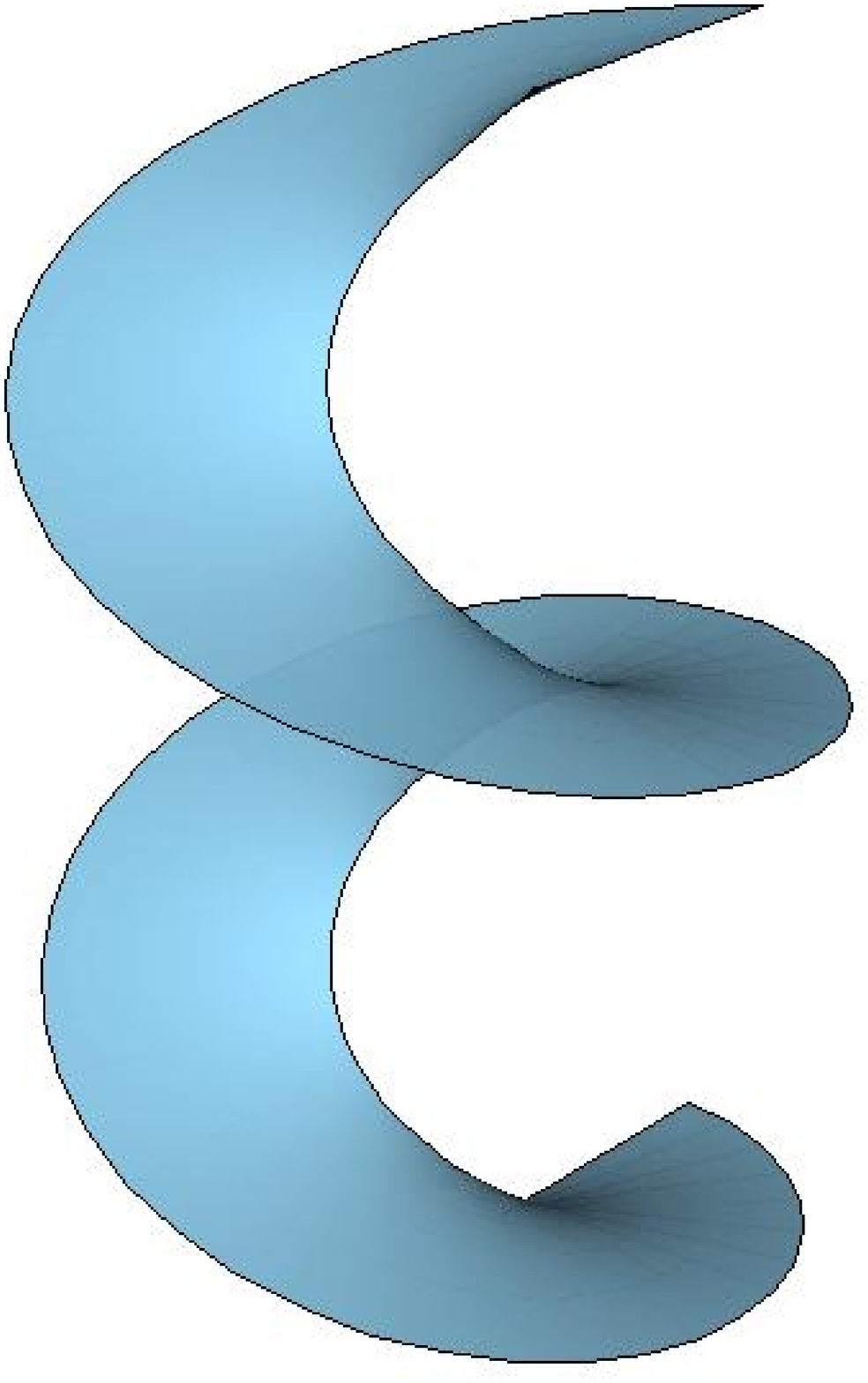} \\
 the Lorentzian catenoid, & 
 the Lorentzian helicoid, \\
 $(G,\eta)=(z,dz/z^2)$
 on $\C\setminus\{0\}$ &
 $(G,\eta)=(z,i\,dz/z^2)$\\
 &
 on the universal cover of $\C\setminus\{0\}$
\end{tabular}
\caption{%
  The duality between cone-like singular points
  and fold singular points.
  The pair $(G,\eta)$ denotes the Weierstrass 
  data, see Section~\ref{sec:prelim}.}
\label{fg:cat}
\end{center}
\end{figure}
\begin{figure}
\begin{center}
 \includegraphics[width=0.2\textwidth]{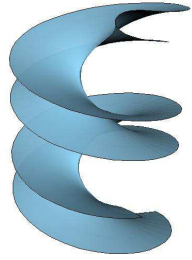}
\end{center}
\caption{%
 An associated surface of the Lorentzian helicoid,
 with Weierstrass data $(G,\eta)=(z,e^{\pi i/4}z^{-2}dz)$.}
\label{fig:ass-helicoid}
\end{figure}
Maxfaces are spacelike at their regular points, 
but the limiting tangent plane 
(that is, the Lorentzian orthogonal complement of the normal
vector)
at each singular point contains a lightlike direction.
For the global study of maximal surfaces, the following
terminology given in \cite{UY} is useful:
\begin{introdefinition}\label{def:complete}
 A maxface (or more generally, a generalized maximal surface) 
 $f\colon{}M$ $\to \R^3_1$ is called  {\it complete\/}
 if there exists a symmetric $2$-tensor $T$ 
 which vanishes outside a compact set in $M$, such that
 $ds^2+T$ is a complete Riemannian metric on $M$,
 where $ds^2$ is the induced metric by $f$.
 If $f$ is complete, the set of singular points is
 compact in $M$.
 On the other hand, a maxface is called {\it weakly complete\/}
 (in the sense of \cite{UY}), 
 if its null holomorphic lift into $\C^3$ (see Section~\ref{sec:prelim}) 
 has complete induced metric with respect
 to the canonical Hermitian metric on $\C^3$.
\end{introdefinition}

\begin{figure}
\begin{center}
\begin{tabular}{cc}
 \includegraphics[width=.30\linewidth]{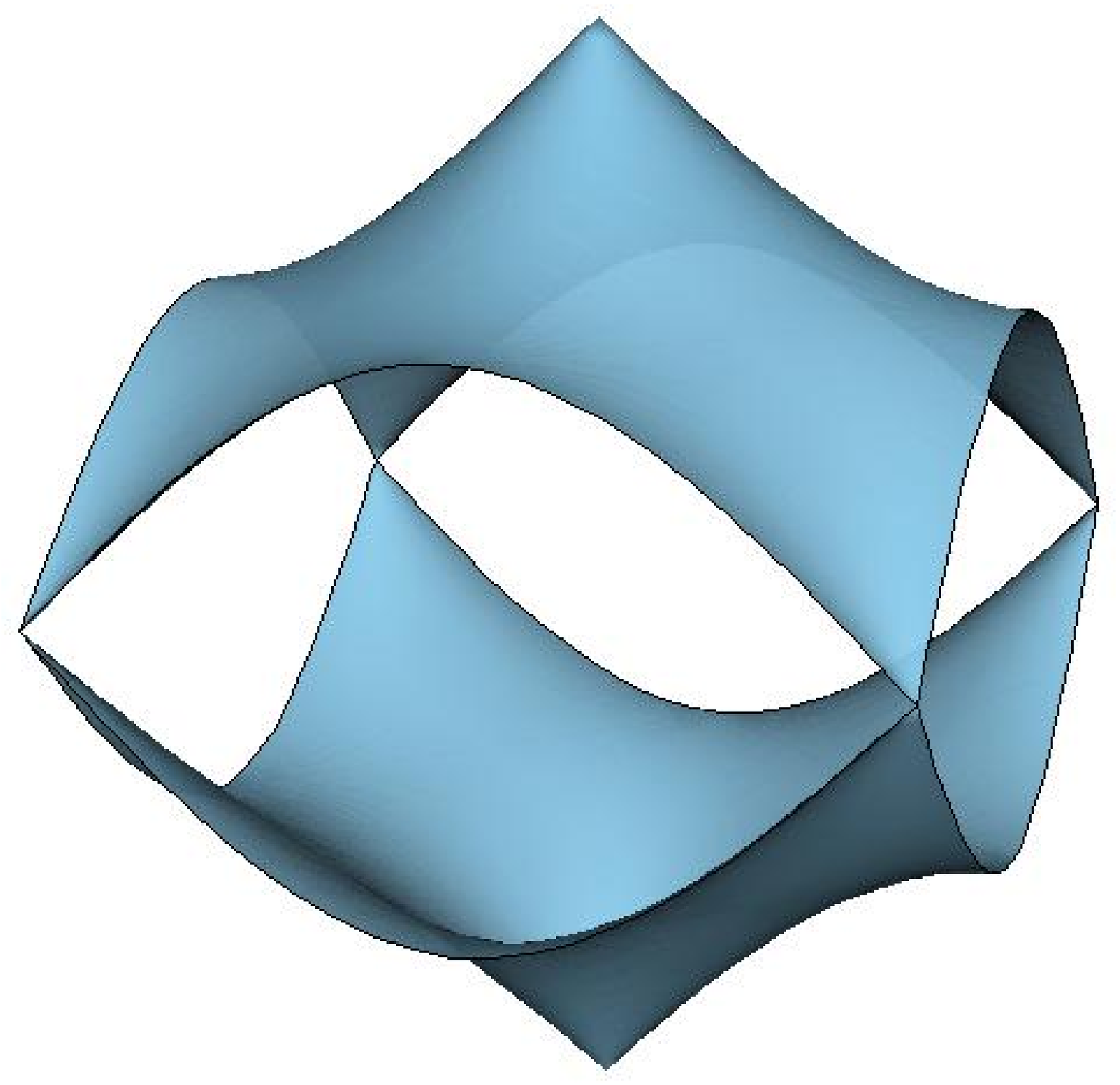} &
 \includegraphics[width=.55\linewidth]{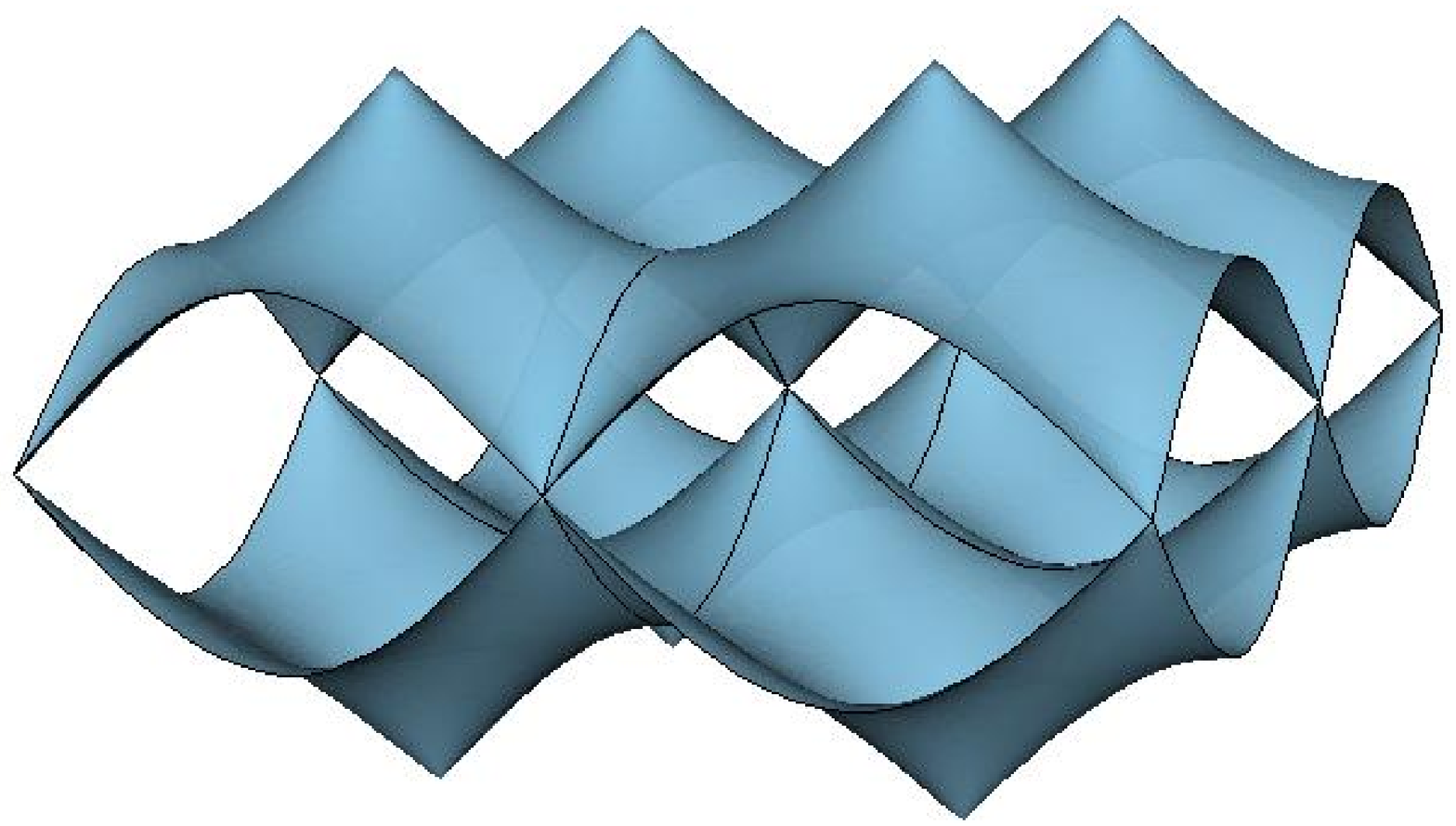} \\
\end{tabular}
\end{center}
\caption{%
   A weakly complete triply-periodic maxface with
   cone-like singular points corresponding to the Schwarz-P surface;
   see \cite{FRUYY3} for details.}
\label{fig:surface_P}
\end{figure}
\begin{figure}
\begin{center}
\begin{tabular}{cc}
 \includegraphics[width=.30\linewidth]{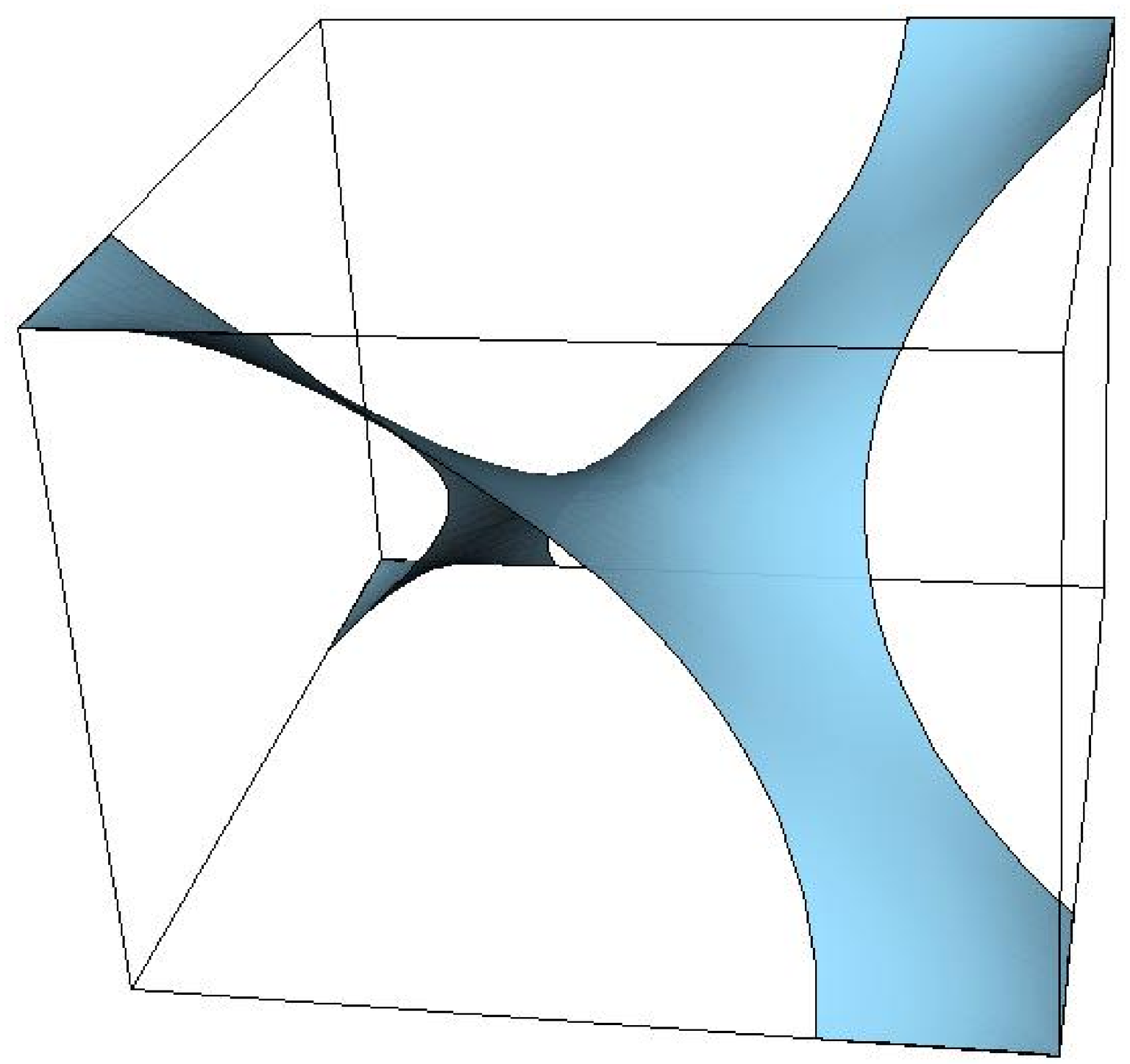} &
 \includegraphics[width=.55\linewidth]{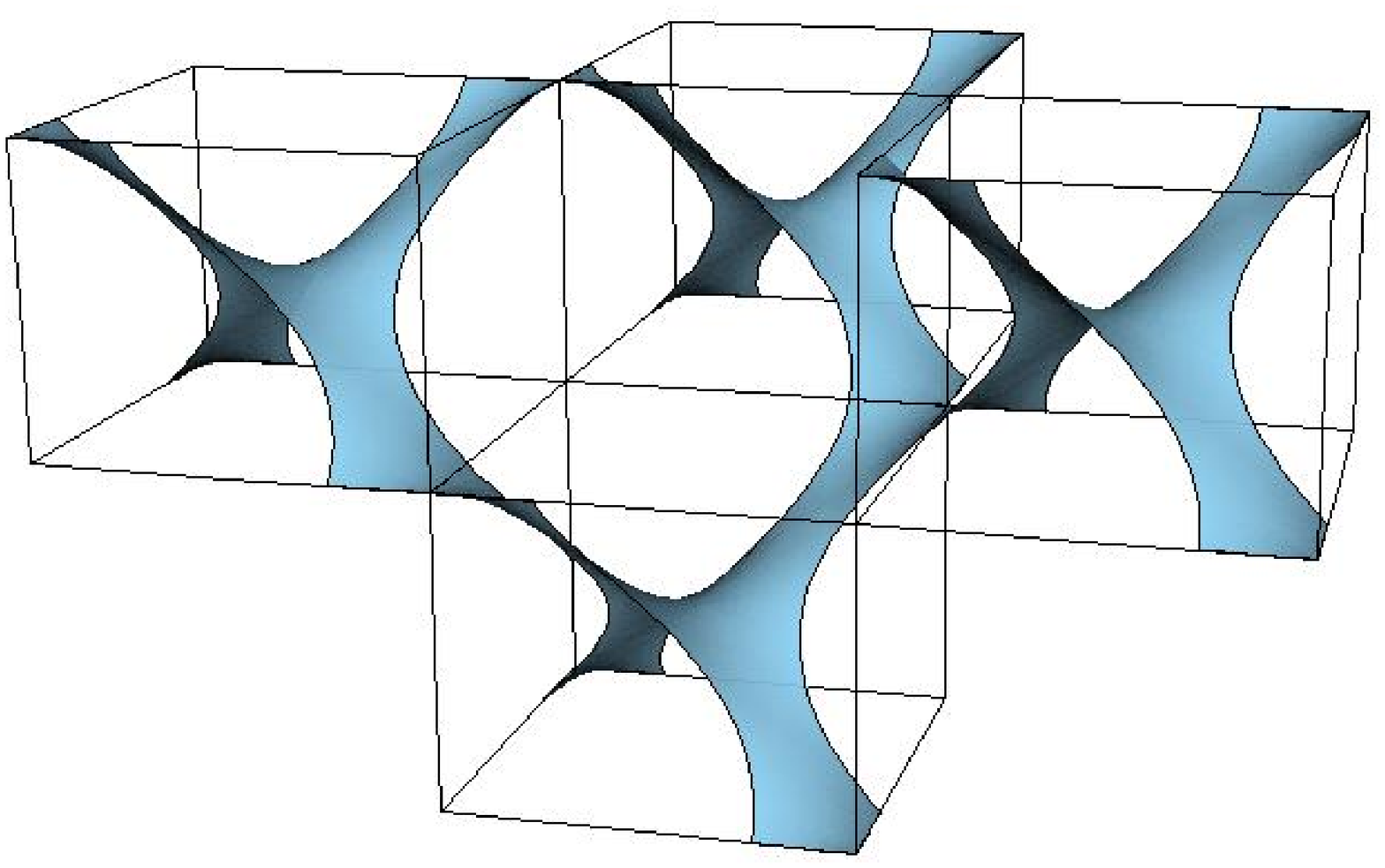} \\
\end{tabular}
\end{center}
\caption{%
   A weakly complete triply-periodic maxfaces with
   fold singular points corresponding to the Schwarz-D surface;
   see \cite{FRUYY3} for details.}
\label{fig:surface_D}
\end{figure}

As shown in \cite[Lemma 4.3]{UY},
completeness implies weak completeness.
Conversely, {\em a weakly complete maxface is complete if and only if the
singular set is compact and each end is conformally equivalent to 
a punctured disc} (see \cite{UY2}).
A typical well-known complete maxface is the Lorentzian catenoid
(see \cite{K}, see also \cite{ACM}
in which it is called the {\em Lorentzian elliptic catenoid\,};
Figure~\ref{fg:cat}, left),
which has a {\em cone-like singular point\/}
(see Section~\ref{sec:cone} for the definition).
Like minimal surfaces in Euclidean $3$-space $\R^3$,
maxfaces have conjugate surfaces.
A cone-like singular  point on a maxface corresponds to a fold singular
point on its conjugate maxface, in general.
For the proof of this and the definition of fold singular points, see
\cite{KY2}.
The Lorentzian helicoid 
(see \cite{K}, see also \cite{ACM}; Figure~\ref{fg:cat}, right) 
is weakly complete (but not complete)
and is the conjugate maxface of the Lorentzian catenoid, 
whose image consists of two surfaces with `boundary' 
in $\R^3_1$. 
The boundary (that is, the singular set) is a helix, and
each interior image point on the surface has two inverse images.
Namely, the Lorentzian helicoid can be regarded as a fold along a
helix. 

As shown in \cite{FSUY},
generic singular points of maxfaces are cuspidal edges, swallowtails 
and cuspidal cross caps.
Thus, cone-like singular points and folds are non-generic.
For example, one can consider an isometric deformation 
of Lorentzian helicoid corresponding to the family of Weierstrass data
$(z,e^{it}dz/z^2)$
($t\in [0,\pi/2]$).
Figure~\ref{fig:ass-helicoid} gives the maxface corresponding $t=\pi/4$,
whose singular points consist only of cuspidal edges.

However, they 
(i.e., cone-like singular points and folds)
are important singular points in the theory of maxfaces.
For example, fold singular points 
(i.e., the double surfaces in \cite{IK}) appear 
under a certain situation, see \cite[Proposition 7.7 and Page 581]{IK}.

It should be remarked that a complete maxface automatically has finite
total curvature outside of a compact set, and finite topology as well
(see \cite[Theorem 4.6 and Corollary 4.8]{UY}). 
This is a property that is crucially different from the case of
minimal surfaces in $\R^3$.  
So, interestingly, there are no complete periodic maxfaces although
there exist  compact maxfaces in a Lorentzian torus $\R^3_1/\Gamma$
for a suitable lattice 
(namely, triply-periodic weakly complete maxface in $\R^3_1$). 
In fact, the same Weierstrass data as for the Schwarz P-surface
and the Schwarz D-surface give  
such examples.
See Figures \ref{fig:surface_P} and \ref{fig:surface_D}.

In \cite[Theorem 4.11]{UY}, it was shown that an Osserman-type
inequality 
\begin{equation}\label{eq:oss}
    2\deg G \geq -\chi(M)+\text{(number of ends)}
\end{equation}
holds for the degree of the Gauss map of complete maxfaces
$f:M\to \R^3_1$, and equality holds if and only if all ends are properly
embedded. Here
\[
     G:M\longrightarrow S^2=\C\cup \{\infty\}
\]
is the Lorentzian Gauss map and $\deg G$ is its degree
as a map to 
the hyperbolic sphere $S^2$ considered as a compactification of the
hyperboloid in $\R^3_1$,
see \cite{UY} and \cite[Section 5]{KU}.
Since $G$ is meromorphic at each end of a complete maxface,
the left-hand side of \eqref{eq:oss} is finite
(see Fact~\ref{fact:complete}).

In \cite{FRUYY}, the authors showed that a similar 
Osserman-type inequality \eqref{eq:oss}
also holds for the hyperbolic Gauss maps $G$ of complete \cmcone{} faces
in de Sitter $3$-space
(for the definition of \cmcone{} faces, see
Section~\ref{sec:deform}).
In contrast to the case of maxfaces,
the hyperbolic Gauss map of a \cmcone{} face may have an essential
singularity at an end.

The Lorentzian catenoid satisfies equality in \eqref{eq:oss}. 
In \cite{UY}, several examples which attain equality in \eqref{eq:oss}
were given.
Recently, complete maxfaces with three embedded ends
were classified by Imaizumi and Kato \cite{IK}.
We mention here two new interesting phenomena on the shape of singular
points on maximal surfaces:
\begin{figure}
\begin{center}
\begin{tabular}{c@{}c}
 \includegraphics[width=.55\linewidth]{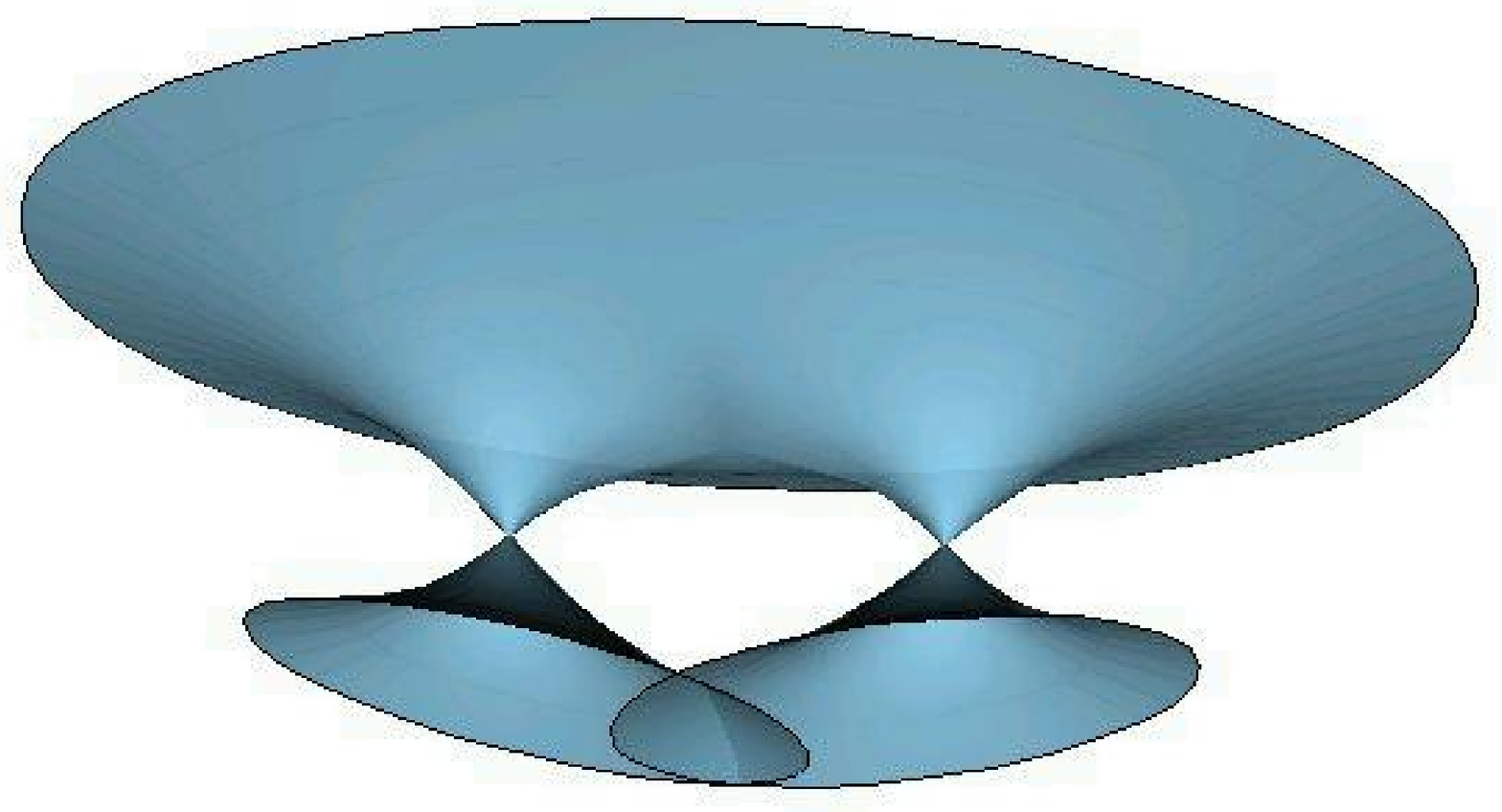} &
 \includegraphics[width=.40\linewidth]{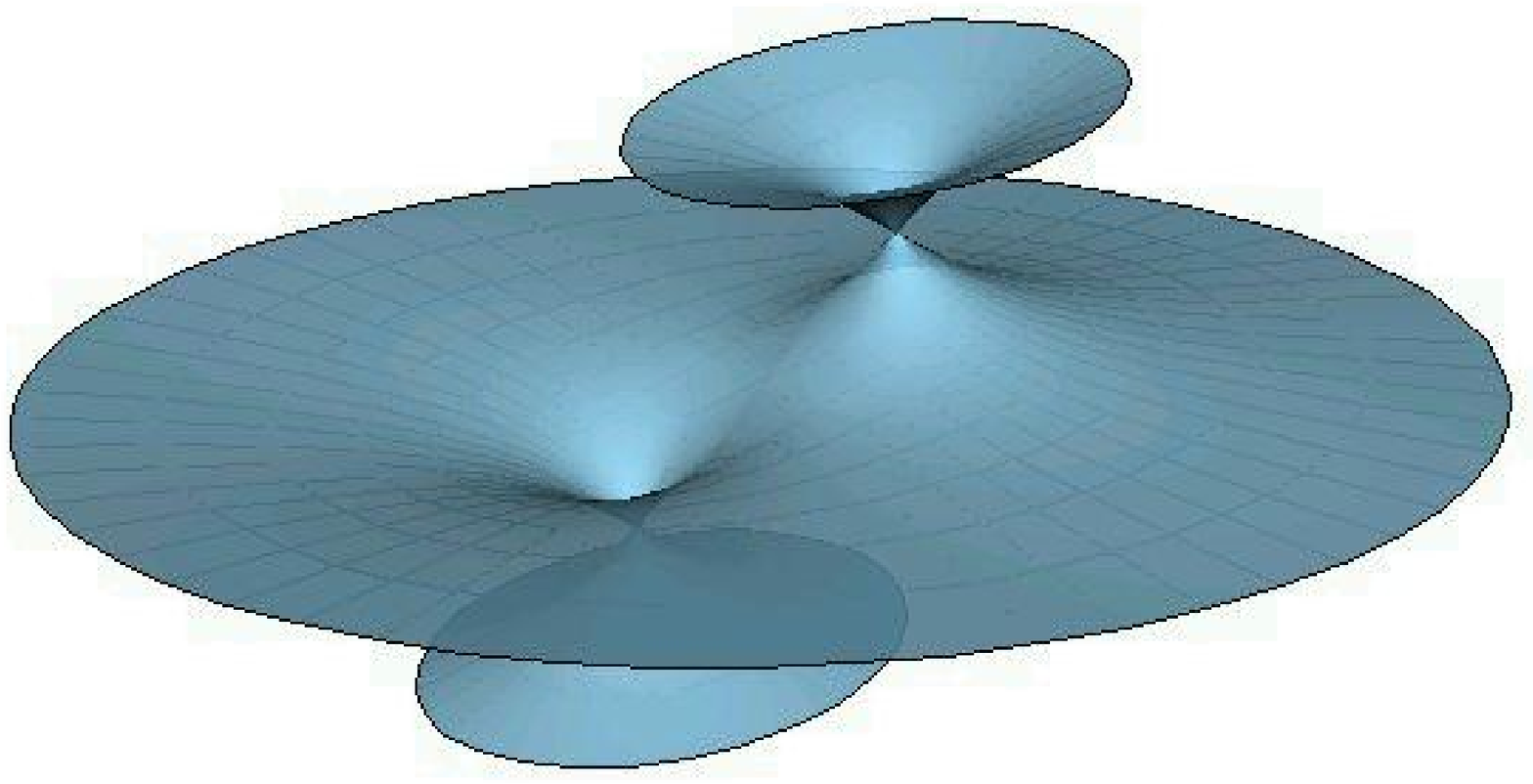} \\
 Trinoid in \eqref{eq:trinoid-1}  for $a=3.67$&
 Trinoid in \eqref{eq:trinoid-2}  for $c=0.1$\\[12pt]
 \includegraphics[width=.56\linewidth]{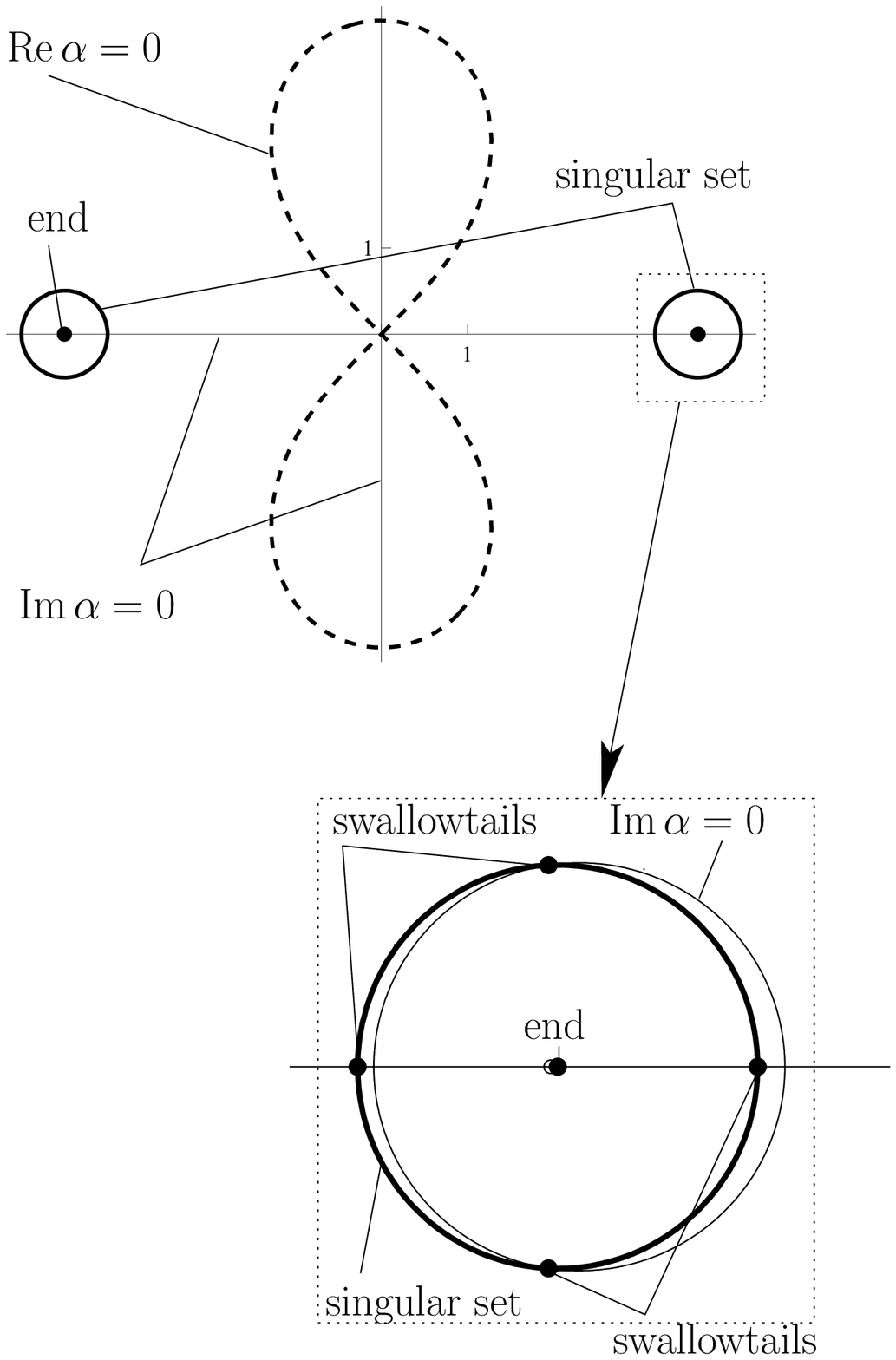}&
 \includegraphics[width=.41\linewidth]{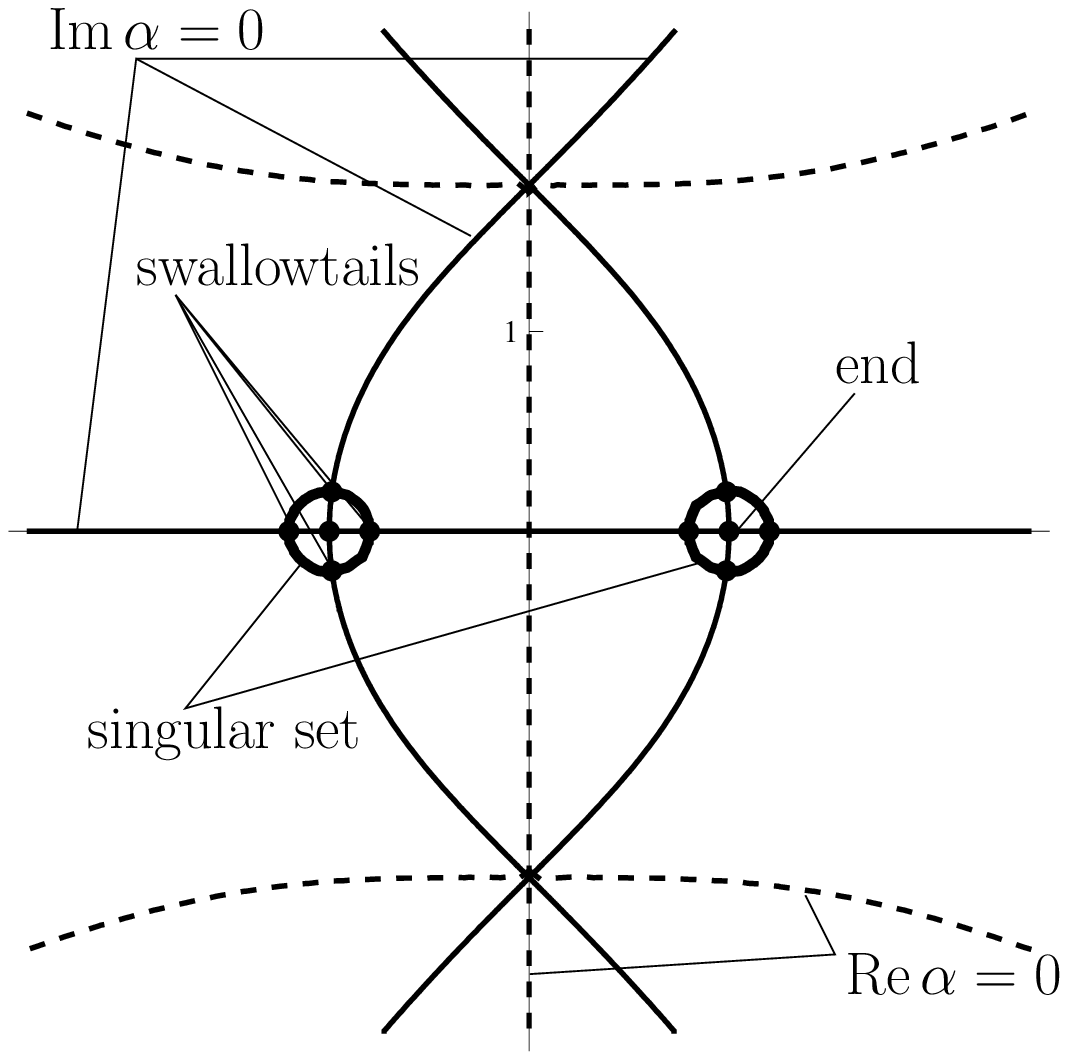}\\
 Singular set of \eqref{eq:trinoid-1} for $a=3.67$&
 Singular set of \eqref{eq:trinoid-2} for $c=0.1$\\
\end{tabular}\\[12pt]
 Both trinoids have eight swallowtails and cuspidal edges for
 singular points.
 No cuspidal cross caps appear.
 For the notations, see Section~\ref{sec:prelim}.
\end{center}
\caption{%
 The trinoids in Example \ref{ex:trinoid}.
}
\label{fig:trinoids}
\end{figure}
\begin{introexample}%
{\rm (Trinoids whose graphics seem to show cone-like
 singular points,
 although no cone-like singularities exist).}
\label{ex:trinoid}
 We set
 \[
    M=\C\setminus\{a,-a\} \qquad (a>1/2)
 \]
 and consider two Weierstrass data (see Section~\ref{sec:prelim})
 \begin{equation}\label{eq:trinoid-1}
   G:=\frac{b-z^2}z,\qquad \eta:=\frac{z^2dz}{(z^2-a^2)^2} \qquad
   \left(b:=-a^2+a\sqrt{4a^2-1}\right).
 \end{equation}
 Substituting these into \eqref{eq:maxface},
 we get a trinoid with three embedded ends at $z=\pm a,\infty$, 
 each of which is asymptotic to a Lorentzian catenoid.
 The left-hand figure in Figure~\ref{fig:trinoids} shows the image with
 $a=3.67$.

 On the other hand, set
 \begin{equation}\label{eq:trinoid-2}
   M=\C\setminus\{1,-1\},\qquad
   G=c\frac{z^2+3}{z^2-1},\quad \text{and}
   \quad
   \eta=\frac{dz}{c}\qquad (c>0, c\neq 1).
 \end{equation}
 Then we have another trinoid as in Figure~\ref{fig:trinoids}, right.
 The ends $1$, $-1$ are asymptotic to the Lorentzian catenoid,
 and the end $\infty$ to the plane.
 The figure shows the image for $c=0.1$.

 Since a maxface is symmetric with respect to a given cone-like
 singular point (cf.~\cite{IK} and \cite{KY2}), 
 neither of these two maxfaces admits any cone-like singular points.

 However, in Figure \ref{fig:trinoids}, 
 we can see two singular sets in each surface which look very much like
 cone-like singularities.
 The singular sets do, however, 
 consist of cuspidal edges and swallowtails (see
 Figure~\ref{fig:trinoids}).

 Similarly, there are four cuspidal cross caps on the  Enneper maxface
 (see \cite[Example 5.2]{UY}).
 However, it is difficult to recognize the crossing of two sheets
 on surfaces near these four singular points
 from computer graphics,
 since these two sheets are so close to each other.
 If a maxface admits a cone-like singular point
 (resp.\ a swallowtail),
 its conjugate surface admits fold singular points 
 (resp.\ a cuspidal cross cap)
 and vice versa (\cite{KY2} and \cite[Corollary 2.5]{FSUY};
 for the definition of cone-like singular points,
 see Definition~\ref{def:cone}).
 Thus, in computer graphics,
 the fact that a union of swallowtails sometimes look like cone-like
 singular points and the fact that cuspidal cross caps look like fold
 singular points seems to be a mutually dual phenomena. 
 The authors hope that one could establish a new theory for
 explaining this.
 In recent private conversations, 
 Shin Kato \cite{Ka} said that this phenomenon for trinoids
 as in Figure~\ref{fig:trinoids}
 seems to occur as a family of trinoids ``collapses'' to 
 Lorentzian catenoids (which have a cone-like singularity).
\end{introexample}
It is interesting to ask if there exist maxfaces having  
a cone-like singular point
and also having singular points which are not cone-like.
We give such an example:
\begin{introtheorem}\label{thm:cone}
 There exists a maxface $f\colon{}M\to\R^3_1$ 
 of genus $0$ with two complete ends
 whose singular set in $M$
 consists of cone-like singular points,
 cuspidal edges, swallowtails, and cuspidal cross caps.
 Here, the image of the cone-like singular points
 is a single point.
\end{introtheorem}
The proof of the first part of this theorem is given in
Section~\ref{sec:cone}.

In $\R^3$, examples of complete embedded minimal surfaces of finite
total curvature of course are known --- 
for example, the plane, the catenoid, the Costa surface, 
the Costa-Hoffman-Meeks surface, etc.
As a related class of maxfaces, the authors
give here the following:
\begin{introdefinition}\label{def:embedded}
  A complete
  maxface is called {\it embedded\/} ({\em in the wider sense})
  if it is embedded outside of some compact set of $\R^3_1$.  
\end{introdefinition}
From here on out, when we use the word ``embedded'' for a maxface, 
we always mean ``embedded in the wider sense'' as in the above 
definition. 
By definition, an embedded maxface attains equality in \eqref{eq:oss}.
In a joint work \cite{KY} with Kim, the fifth author constructed
maximal surfaces of genus $k=1,2,3,\dots$, 
which are complete generalized maximal surfaces
in the sense of \cite{ER}
(i.e., they admit branch points when $k\geq 2$), and when $k=1$
it is  a complete embedded  maxface, see Figure~\ref{fg:k=1}.
We shall call this example the {\em Kim-Yang toroidal maxface}.
We remark that there exist no embedded minimal surface in $\R^3$ 
with two ends except the catenoid (cf.\ \cite{S}).

Until now, the only known examples of embedded complete maxfaces were
\begin{itemize}
 \item the spacelike plane (which is the only example of a complete 
       maximal surface without singular points),
 \item the Lorentzian catenoid,
 \item the Kim-Yang toroidal maxface.
\end{itemize}
Also, until now the only known complete positive-genus maxfaces 
were the Lorentz\-ian Chen-Gackstatter surface 
(given in \cite[Example 5.5]{UY}) and the Kim-Yang to\-roidal maxface.
In this article, we construct complete maxfaces with two ends and
arbitrary genus, which are embedded if the genus is equal to $1$.
(Theorem~\ref{thm:main}).

On the other hand, surfaces of constant mean curvature one (\cmcone{}
surfaces) in de Sitter $3$-space $S^3_1$ have similar properties to
maximal surfaces in $\R^3_1$ (cf.\ \cite{F} and \cite{FRUYY}).
Analogous to maxfaces, a notion of \cmcone{} faces in $S^3_1$, 
which is \cmcone{} surfaces with certain kinds of singular points, was
introduced in \cite{F}.
Related to this, the second, the third and forth authors introduced 
in \cite{RUY} a method to deform minimal surfaces in $\R^3$ to \cmcone{}
surfaces in hyperbolic $3$-space.
In this paper, we demonstrate that this method works for 
maxfaces in $\R^3_1$ and \cmcone{} faces in $S^3_1$ as well
(see Section~\ref{sec:deform} and the appendix).

In this article, we shall prove:
\begin{introtheorem}\label{thm:main}
 There exists an family of complete maxfaces $f_k$ for
 $k=1,2,3,\dots$ with two ends, 
 and of genus $k$ if $k$ is odd and genus $k/2$ if $k$ is even. 
 Moreover, $f_1$ and $f_2$ are embedded {\rm (}in the wider sense{\rm)}. 
 Furthermore, the number of swallowtails and the number of cuspidal
 cross caps are both equal to
 $4(k+1)$ if $k$ is odd and $2(k+1)$ if $k$ is even.
 In particular, $f_1$ and $f_2$ are both of
 genus one, but are not congruent.
 {\rm(}Compare Figures~\ref{fg:k=1} and \ref{fg:k=2}.{\rm)}

 Moreover, such maxfaces $f_k$ can be deformed to complete 
 \cmcone{} faces in $S^3_1$,
 which are embedded {\rm(}in the wider sense{\rm)} if $k=1$ or $2$.
\end{introtheorem}
We note that $f_1$ is the Kim-Yang toroidal maxface, but the $f_k$
are new examples for all $k \geq 2$.  
The proof of this theorem is given in Sections~\ref{sec:main} and
\ref{sec:deform}.
The method of construction is somewhat similar to that of \cite{KY}, but
one salient feature of our new  family is that it is free of
branch points for all $k=1,2,3,\dots$, whereas
the surfaces in the corresponding family of \cite{KY} have 
two branch points whenever $k \geq 2$.
However, there is still a scarcity of examples of complete maxfaces, 
especially embedded examples.

In light of the relationships between maxfaces and \cmcone{} faces,
we give a pair of open problems about finding new surfaces:
\begin{problem}
 Are there complete maxfaces 
 {\rm (}resp.\  complete \cmcone{} faces{\rm)} with embedded ends
 of genus greater than one in $\R^3_1$ $($or $S^3_1)$?
 Furthermore, could such examples actually be 
 embedded {\rm(}in the wider sense{\rm)}?
\end{problem}
The first and second authors \cite{FR} provided numerical evidence for the
existence of \cmcone{} faces in $S^3_1$ of higher genus with two
embedded ends.
\begin{problem}
 Are there complete maxfaces 
 {\rm(}resp.\  complete \cmcone{} faces{\rm)} with more than
 two ends of positive genus in $\R^3_1$  {\rm(}or $S^3_1${\rm)}?
\end{problem}

The present state of this field is such that it would be very beneficial
to have a larger collection of examples, as described in these open 
problems, for example.  
A number of types of maximal surfaces can be produced from a canonical
correspondence with minimal surfaces in Euclidean $3$-space given in
Section 5 of \cite{UY} 
(for example, this method applies to the minimal
 surfaces of arbitrary genus found in \cite{Sa}),
although this construction needs to solve a period problem.
(Correspondence of minimal surfaces and maximal surfaces was
 first introduced in \cite{C} for graphs.)

Non-orientable  complete maxfaces were recently found in  a joint work
\cite{FuL} of the first author with L\'opez.
(Unfortunately, the ends of these examples are not embedded.)
On the other hand, \cmcone{} faces in $S^3_1$  are all orientable 
(this fact is not trivial, since the surface admits singular points,
see \cite{KU}).
Weakly complete (but not complete) bounded maxfaces 
(resp.\ \cmcone{} faces)
with arbitrary genus were constructed in \cite{MUY1} and \cite{MUY2}.
\begin{figure}
\begin{center}
\begin{tabular}{cc}
 \includegraphics[width=.40\linewidth]{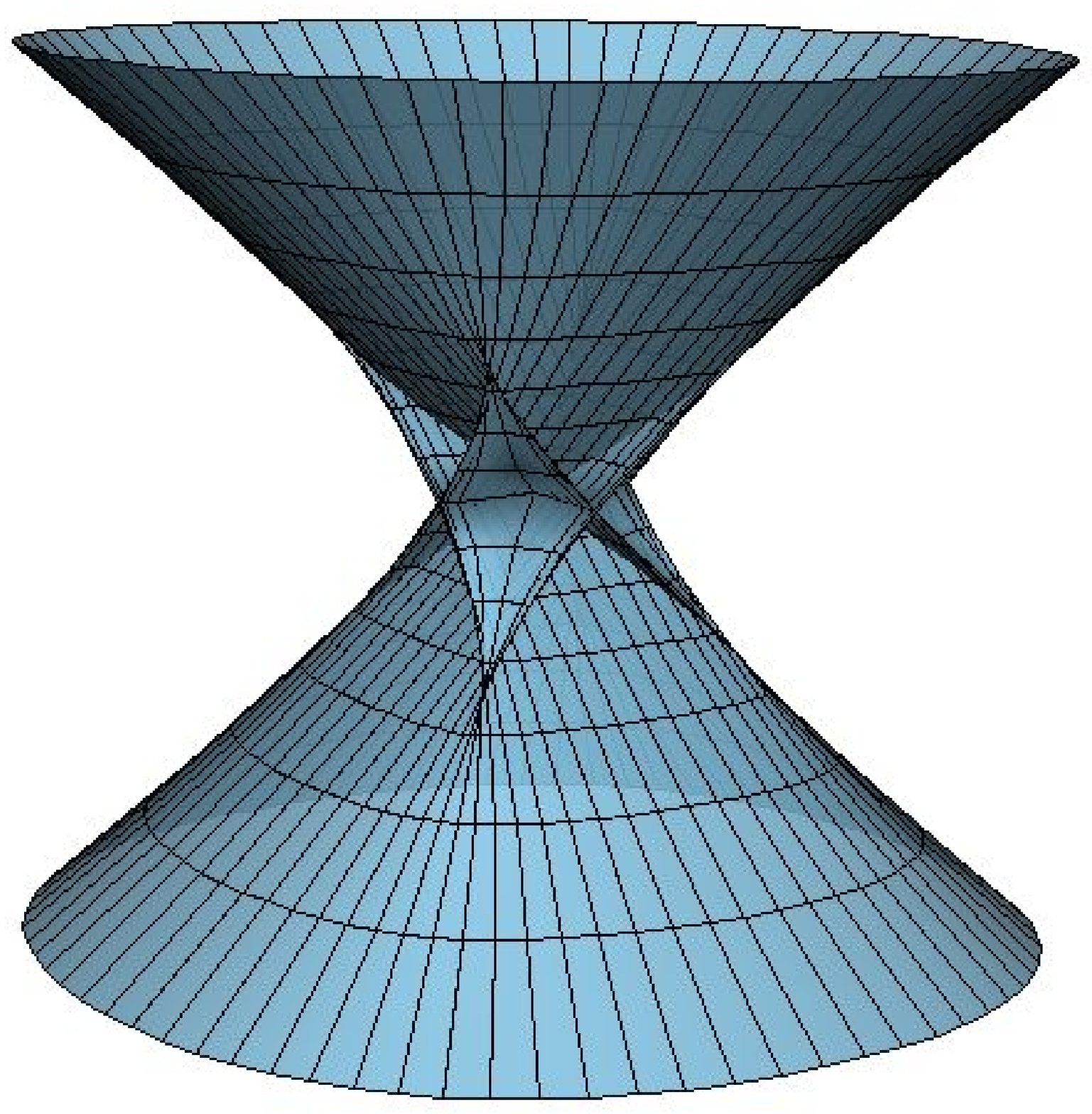} &
 \includegraphics[width=.40\linewidth]{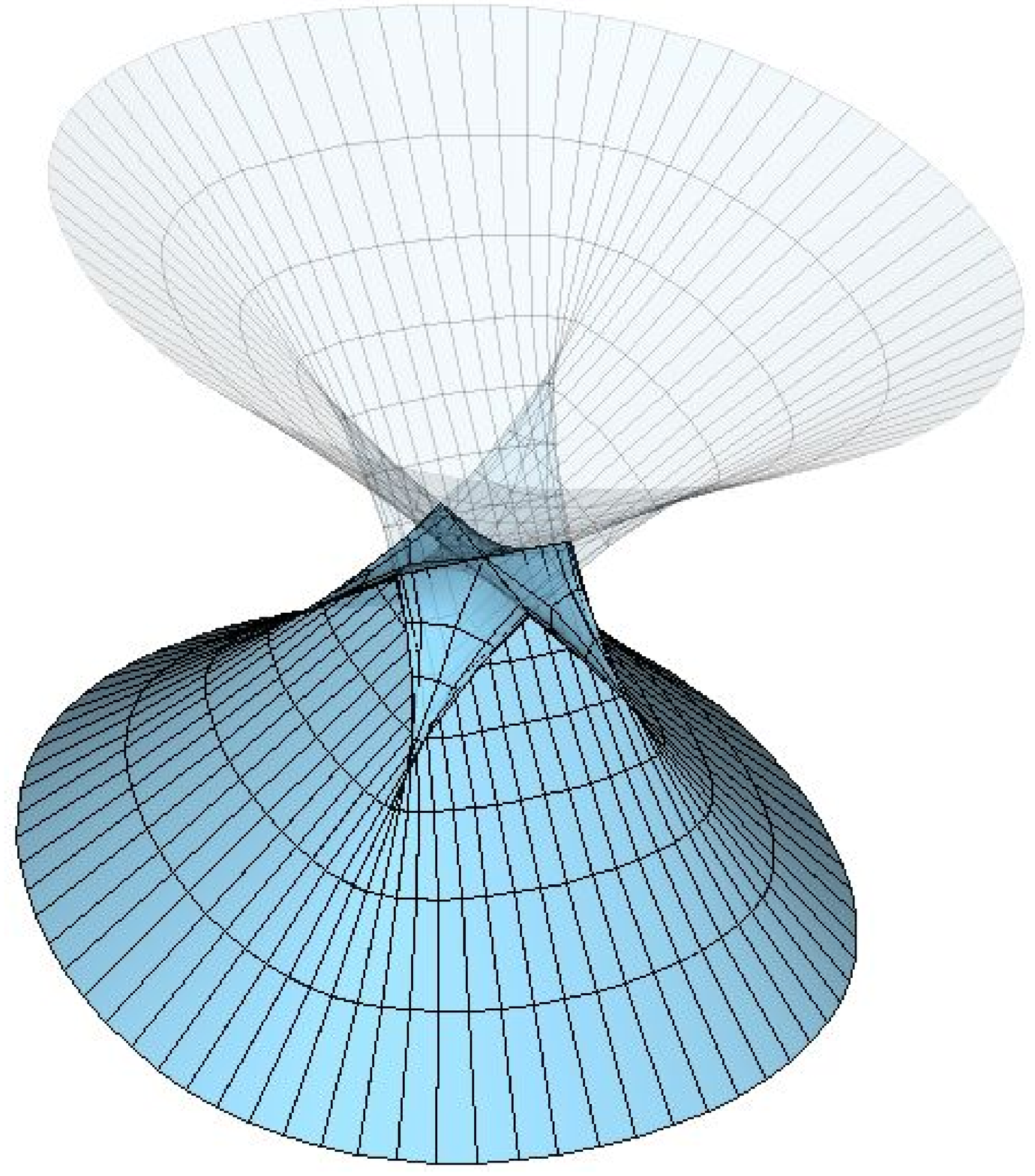} \\
\end{tabular}
\end{center}
\caption{%
 The example for $k=1$ (Kim-Yang toroidal maxface) and half of it.}
\label{fg:k=1}
\end{figure} 
\begin{acknowledgements}
     The authors thank Shin Kato for valuable conversations.
     The third and fourth authors also thank Francisco J.
     L\'opez for valuable conversations during their stay
     at Granada.
\end{acknowledgements}
\endgroup
\section{Preliminaries}\label{sec:prelim}
In this section, we review the Weierstrass-type representation 
formula for maxfaces (see \cite{K,UY}),
and criteria for singular points (see \cite{UY,FSUY}).

Throughout this paper, we denote by $\R^3_1$ the Minkowski
$3$-space with the inner product $\inner{~}{~}$ of
signature $(-,+,+)$.
\begin{fact}[{\cite[Theorem 2.6]{UY}}]\label{fact:weier}
 Let $M$ be a Riemann surface with a base point $o\in M$,
 and  $(G,\eta)$ a pair of a meromorphic function and
 a holomorphic $1$-form on $M$ such that
 \begin{equation}\label{eq:normalized}
   (1+|G|^2)^2|\eta|^2
 \end{equation}
 gives a {\rm(}positive definite{\rm)} Riemannian metric on $M$, 
 and $|G|$ is not identically $1$.
 Let 
 \begin{equation}\label{eq:weier-diff}
     \Phi := \bigl(-2G,1+G^2,i(1-G^2)\bigr)\eta
 \end{equation}
 and assume
 \begin{equation}\label{eq:period-1}
    \Re\oint_{\gamma_j} \Phi =0 \qquad (j=1,\dots,N),
 \end{equation}
 for loops $\{\gamma_j\}_{j=1}^N$ such that the set  $\{[\gamma_j]\}$ of
 homotopy classes generates the fundamental group $\pi_1(M)$ of $M$.
 Then 
 \begin{equation}\label{eq:maxface}
   f(p):=\Re\int_{o}^p\Phi=
         \Re\int_{o}^p \bigl(-2G, 1+G^2,i(1-G^2)\bigr)\eta
 \end{equation}
 is well-defined on $M$ and gives a maxface in $\R_1^3$.
 Moreover, any maxfaces are obtained in this manner.
 The induced metric $ds^2$ and the second fundamental form $\secondff$
 are expressed as
 \begin{equation}\label{eq:f-forms}
   ds^2 = \bigl(1-|G|^2\bigr)^2|\eta|^2
   \quad\text{and}\quad
   \secondff = Q + \overline Q,\qquad
   (Q=\eta\,dG),
 \end{equation}
 respectively.
 Weak completeness as in Definition~\ref{def:complete} in the
 introduction is equivalent to completeness of the metric
 \eqref{eq:normalized}.
 The point $p\in M$ is a singular point of the maxface
 \eqref{eq:maxface}
 if and only if $|G(p)|=1$.
\end{fact}
We call the pair $(G,\eta)$ the {\em Weierstrass data\/} of the maxface
$f$ in \eqref{eq:maxface}, 
$G$ the {\em Lorentzian Gauss map},
and the holomorphic $2$-differential $Q$ in \eqref{eq:f-forms} 
the {\em Hopf differential}, respectively.
\begin{fact}[\cite{UY2}]%
\label{fact:complete}
 Let 
 $M$ be a Riemann surface,
 and $f\colon{}M\to\R^3_1$ a weakly complete maxface.
 Then $f$ is complete if and only if
 there exists a compact Riemann surface $\overline M$ and a finite 
 number of points
 $p_1,\dots,p_n\in \overline M$ such that
 $M$ is conformally equivalent to 
 $\overline M\setminus\{p_1,\dots,p_n\}$,
 and the set of singular points
 \[
        \Sigma=\{p\in M\,;\,|G(p)|=1\}
 \]
 is compact.
 In this case, the Weierstrass data $(G,\eta)$ is well-defined 
 as a pair of a meromorphic function and a meromorphic one form on
 $\overline M$,
 and the compactness of the singular set $\Sigma$ is equivalent to 
 the condition $|G(p_j)|\neq 1$ for $j=1,\dots,n$.
\end{fact}

As shown in \cite{FSUY}, generic singular points of maxfaces 
(resp.\ \cmcone{} faces) are cuspidal edges, swallowtails and cuspidal
cross caps.
We recall criteria for generic singular points of maxfaces.
For terminology on \cmcone{} faces, see Section~\ref{sec:deform}.
\begin{fact}[{\cite[Theorem 3.1]{UY}}, {\cite[Theorem 2.4]{FSUY}}]%
\label{fact:maxface-sing}
 Let $U$ be a domain of the complex plane $(\C,z)$ and
 $f:M \to \R^3_1$ 
 a maxface, 
 $(G,\eta)$ its Weierstrass data.
 Set 
 \begin{equation}\label{eq:criteria}
   \alpha:=\frac{dG}{G^2\eta},\quad\text{and}\quad
   \beta:=G\frac{d\alpha}{dG}
 \end{equation}
 Then
 \begin{enumerate}
  \item A point $p\in M$ is a singular point of $f$ if and only if
	$|G(p)|=1$.
  \item $f$ is  right-left equivalent to a cuspidal edge at $p$ if and only if
	$\Im\alpha(p)\neq 0$,
  \item $f$ is right-left equivalent to a swallowtail at $p$ if and only if
	$\Im\alpha(p)=0$ and $\Re\beta(p)\neq 0$,
  \item and $f$ is  right-left equivalent to a cuspidal cross cap at $p$
	if and only if
	$\Re\alpha(p)=0$, $\alpha(p)\neq 0$, and $\Im\beta(p)\neq 0$.
 \end{enumerate}
\end{fact}
Here, two $C^{\infty}$-maps $f_1\colon{}(U_1,p)\to N^3$ and 
$f_2\colon{}(U_2,q)\to N^3$ of domains $U_j\subset \R^2$ ($j=1,2$)
into a $3$-manifold $N^3$ are {\em right-left equivalent}
at the points $p\in U_1$ and $q\in U_2$ if there  exists a local
diffeomorphism $\varphi$ of $\R^2$ with $\varphi(p)=q$ 
and a local diffeomorphism $\Phi$ of $N^3$ with $\Phi(f_1(p))=f_2(q)$
such that $f_2=\Phi\circ f_1 \circ \varphi^{-1}$.

\section{Maxfaces with cone-like singular points}
\label{sec:cone}
In this section, we give the proof of Theorem~\ref{thm:cone} in the 
introduction.
To do this, we recall a definition and a criterion for cone-like
singular points.
\subsection{Cone-like singular points}
\label{sub:cone}
\begin{figure}
\footnotesize
\begin{center}
 \includegraphics[width=0.4\textwidth]{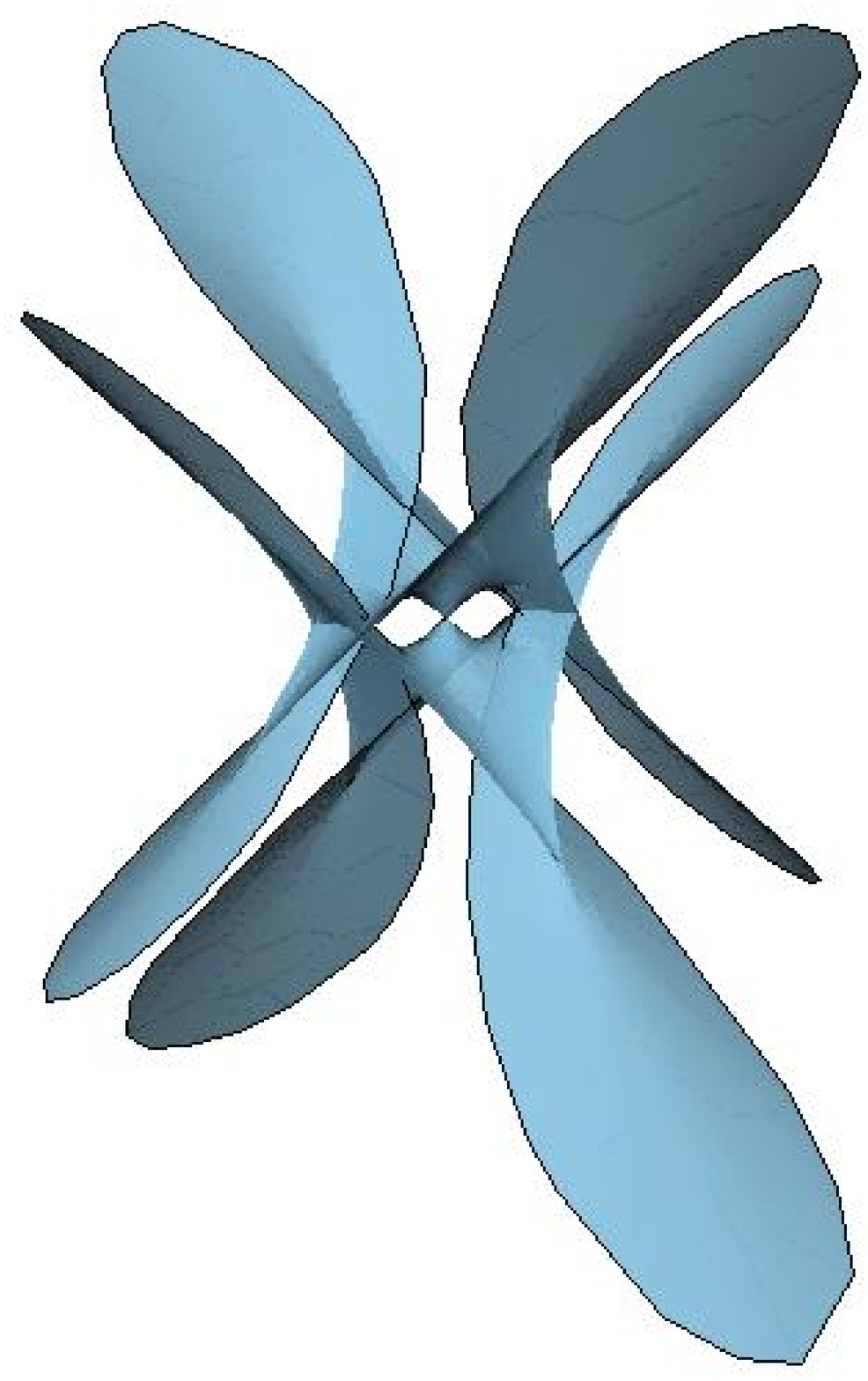} \\
  The image of the surface with \eqref{eq:cone-data} and $a=2.5$.\\
 \includegraphics[width=0.8\textwidth]{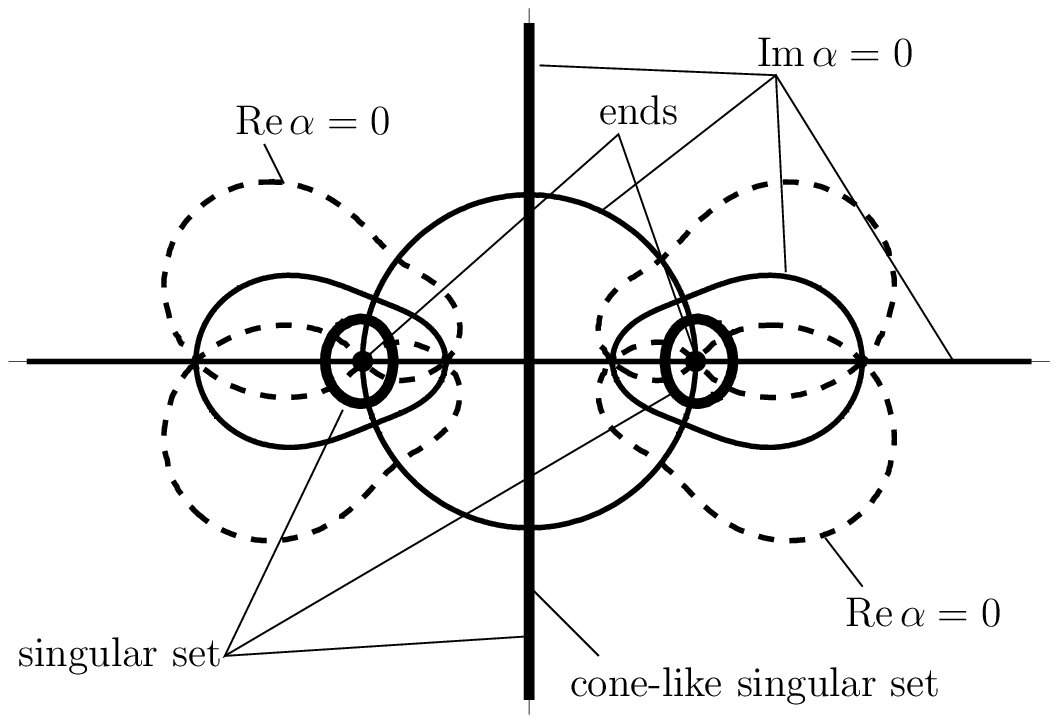}  \\
 \begin{minipage}{\textwidth}
  The singular set in the $z$-plane.
  The imaginary axis corresponds to the set of
  cone-like singular points,
  and the other connected components each consist
  of cuspidal edges, four swallowtails (the intersection points
  of the singular set with the curve $\Im\alpha=0$,
  the curves shown in thin lines),
  and four cuspidal cross caps (the intersection points
  of the singular set with the curve $\Re\alpha=0$, the dotted curves).
 \end{minipage}
\end{center}
\caption{%
 A maxface with a cone-like singularity and other singularities
}
\label{fig:conical}
\end{figure} 
\begin{definition}[Cone-like singular points]
\label{def:cone}
 Let $\Sigma_0$ be a connected component of the set of singular points
 of the maxface $f$ as in \eqref{eq:maxface} 
 which consists of non-degenerate singular points
 in the sense of \cite[Section 3]{UY} and \cite{FSUY}.
 Then each point of
 $\Sigma_0$ is called a {\em generalized cone-like singular point\/}
 if $\Sigma_0$ is compact and the image $f(\Sigma_0)$ is one point.
 Moreover, if there is a neighborhood $U$ of $\Sigma_0$ and
 $f(U\setminus \Sigma_0)$ is embedded, 
 each point of
 $\Sigma_0$ is called a {\em cone-like singular point.}
\end{definition}

\begin{remark}\label{rem:conical}
 In \cite{K2} and \cite{FLS}, the image $f(\Sigma_0)$ as a single point
 of the cone-like singular points is called a 
 {\em cone-like singular point\/}.
 However, we do not use this terminology here,
 since we treat the singular set not in $\R^3_1$ but 
 in the source manifold.
\end{remark}

\begin{lemma}\label{lem:criterion}
 A connected component $\Sigma_0$ of 
 the set of singular points of the maxface \eqref{eq:maxface}
 consists of generalized cone-like singular points 
 if and only if it is compact, and 
 \[
    \alpha \neq 0 \qquad\text{and}\qquad \Im\alpha=0
 \]
 holds, where $\alpha$ is  a function on $M$ as in \eqref{eq:criteria}.
\end{lemma}
\begin{proof}
 Take a complex coordinate $z$ around $p\in\Sigma_0$ and identify the 
 tangent plane of $M$ with $\C$.
 The point $p$ is non-degenerate if and only if $\alpha\neq 0$,
 because of \cite[Lemma 3.3]{UY}.
 The singular direction 
 (the tangential direction of the singular set) 
 and the null direction (the direction of the kernel of $df$)
 are represented as
 $i\overline{(G'/G)}$ and $i/(G\hat\eta)$,
 where ${}'={d}/{dz}$ and $\eta=\hat\eta\,dz$
 (see \cite[Proof of Theorem 3.1]{UY}).
 Here, the image $f(\Sigma_0)$ consists of one point if and only if
 these two directions are linearly dependent.
 Thus we have the conclusion.
\end{proof}
Since $dG\neq 0$ on the set $\Sigma_0$ of generalized cone-like
singular points,
$\Sigma_0$ is diffeomorphic to the circle $S^1$.

By \cite[Lemma 1]{FLS}, we have the following criterion:
\begin{lemma}\label{lem:cone-like}
 Assume a connected component $\Sigma_0$ of the singular set of a maxface
 consists of generalized cone-like singular points.
 Then it consists of cone-like singular points if and only if
 $G|_{\Sigma_0}\colon{}\Sigma_0\to S^1\subset\C$ is injective
 and $\eta$ does not vanish on $\Sigma_0$,  where $(G,\eta)$ is 
 the Weierstrass data.
\end{lemma}

\subsection{Proof of Theorem~\ref{thm:cone}}
\label{sub:cone-maxface}
Let $M:=\C\cup\{\infty\}\setminus\{-1,1\}$
and let $a$ be a real constant such that $1<a<4$ and $a\neq 2$.
Set
\begin{equation}\label{eq:cone-data}
   G = \frac{(z-1)(z^2+az +1)}{(z+1)(z^2-az+1)},\qquad
   \eta = \frac{(z^2-a z+1)^2}{(z-1)^4(z+1)^2}\,dz.
\end{equation}
Then the data $(G,\eta)$ 
gives no real period in the representation
formula \eqref{eq:maxface}, i.e., it satisfies \eqref{eq:period-1}.
Thus it defines a maxface of genus zero with
$2$-ends:
\begin{equation}\label{eq:cone-maxface}
   f\colon{}M\longrightarrow \R^3_1.
\end{equation}
Since $a\in\R$, it holds that $|G|=1$ on the imaginary axis.
When $1<a<4$ and $a\neq 2$, 
the singular set $\Sigma=\{|G|=1\}$ consists of three disjoint
(topological) circles on the Riemann sphere, including
the imaginary axis of the $z$-plane (Figure~\ref{fig:conical}, bottom).

Set $\alpha:=dG/(G^2\eta)$ as in Section~\ref{sec:prelim}.
Then the set $\{\Im \alpha=0\}$ looks like as in
Figure~\ref{fig:conical}, bottom.
In particular, $\Im \alpha=0$ on the imaginary axis, and hence
Lemma~\ref{lem:criterion} implies that the imaginary axis consists of
generalized cone-like singular points.
Here, since $G$ is of degree $3$ and the singular set 
consists of three connected components, $G$ is injective 
on each connected component of the singular set.
Since $\eta\neq 0$ on the imaginary axis, 
Lemma~\ref{lem:cone-like} implies that the imaginary axis
consists of cone-like singular points.

On the other hand, other connected components of the singular set
contain cuspidal edges, because $\Im \alpha$ is not identically
zero on the singular set.

Since the order of the Hopf differential $Q$ at $z=\pm 1$
is $-4$, both ends are Enneper ends
(cf.\ Example~5.2 in \cite{UY}), as shown in Figure~\ref{fig:conical},
top.

\section{Proof of The first part of Theorem~\ref{thm:main}}
\label{sec:main}
In this section, we construct maxfaces $f_k$ $(k=1,2,3,\dots)$
as in the statement of Theorem~\ref{thm:main}.
To do it, we proceed as follows for each $k$:
\begin{itemize}
 \item Take a Riemann surface $M_k$ of genus $k$, 
       which is a compact Riemann surface excluding $2$ points
       (which correspond to the ends),
       see Section~\ref{sub:domain}.
 \item Construct a complete maxface $\hat f_k\colon{}M_k\to\R^3_1$
       ($k=1,2,3,\dots$),
       see Sections~\ref{sub:domain}--\ref{sub:period}.
 \item When $k=2m$ is an even number, 
       show that $M_k$ is the double cover of a Riemann surface
       $M'_k$ of genus $m$, 
       and $\hat f_k$ induces a complete maxface
       $f_k\colon{}M'_k\to \R^3_1$.
       See Section~\ref{sub:double}.
 \item When $k$ is odd, we set $f_k=\hat f_k$.
\end{itemize}
\subsection{The Riemann surface $M_k$}
\label{sub:domain}
Let
\begin{equation}\label{eq:riemann}
  \overline{M}_k = 
    \left\{ 
      (z,w) \in (\C \cup \{ \infty \})^2 \, ; 
      \, w^{k+1} = z \left( z^2-1 \right)^k  \right\} , 
\end{equation} 
where $k$ is a positive integer. 
As a submanifold of $(\C \cup \{ \infty \})^2$,
$\overline{M}_k$ has singular points at $(z,w)=(\pm 1,0)$ and
$(\infty,\infty)$.
However, one can define  on $\overline{M}_k$ the structure of a Riemann
surface using complex coordinates
$\zeta_0$, $\zeta_{\infty}$, $\zeta_1$ and $\zeta_{-1}$
around $(z,w)=(0,0)$, $(\infty,\infty)$, $(1,0)$, and $(-1,0)$,
respectively,
as follows:
\begin{equation}\label{eq:coords}
 \begin{alignedat}{4}
   z&=(\zeta_0)^{k+1}\quad &&\text{at $(0,0)$},\qquad
  &z&=(\zeta_{\infty})^{-k-1}&&\text{at $(\infty,\infty)$},\\
   z&=1+(\zeta_1)^{k+1}\quad &&\text{at $(1,0)$},\qquad
  &z&=-1+(\zeta_{-1})^{k+1}\quad &&\text{at $(-1,0)$}.
 \end{alignedat}
\end{equation}
Hence the holomorphic map $z\colon{}\overline{M}_k\to \C\cup\{\infty\}$
is of degree $k+1$ with total branching number $4k$.
Then by the Riemann-Hurwitz relation, the genus of $\overline{M}_k$
is $k$.

We shall construct maxfaces $\hat f_k\colon{}M_k\to \R^3_1$
($k=1,2,3,\dots$) with two ends corresponding to $(0,0)$ and
$(\infty,\infty)$,
where 
\begin{equation}\label{eq:m-k}
  M_k = \overline{M}_k\setminus\{(0,0),(\infty,\infty)\}.
\end{equation}
Let $\widetilde{M_k}$ be the universal cover of $M_k$.  

\subsection{Symmetries and the fundamental group of $M_k$}
\label{sub:sym}
For simplicity, we set
\begin{equation}\label{eq:theta}
   \lambda := \frac{\pi}{k+1}.
\end{equation}
Define reflections (orientation-reversing conformal diffeomorphisms)
$\mu_j\colon{}M_k\to M_k$ ($j=1,2,3,4$) as 
\begin{equation}\label{eq:reflections}
\begin{alignedat}{2}
 \mu_1(z,w) &= (\bar z,\bar w),\qquad
&\mu_2(z,w) &= (\bar z, e^{2ki\lambda}\bar w),\\
 \mu_3(z,w) &= (-\bar z,e^{-i\lambda}\bar w),\qquad
&\mu_4(z,w) &= \left(\frac{1}{\bar z},e^{ik\lambda}
		 \frac{\bar w}{\bar z^2}\right).
\end{alignedat}
\end{equation}
Using these, we define the following automorphisms of $M_k$:
\begin{equation}\label{eq:symms}
\begin{alignedat}{2}
  \kappa_1&:=\mu_2\circ\mu_1\qquad
 &\bigl(\kappa_1(z,w)&=(z,e^{2ki\lambda}w)\bigr),\\
  \kappa_2&:=\mu_3\circ\mu_1\qquad
 &\bigl(\kappa_2(z,w)&=(-z,e^{-i\lambda}w)\bigr).
\end{alignedat}
\end{equation}

Choose a base point $o\in M_k$  such that
\begin{equation}\label{eq:base-point}
  o\in \{(t,w)\,;\,1<t<\infty, \arg w=0\}\subset M_k,
\end{equation}
and take  a loop $\gamma$ on $M_k$ starting at $o$ 
as in Figure~\ref{fg:loop}.
\begin{figure}
\begin{center}
\includegraphics[width=0.95\textwidth]{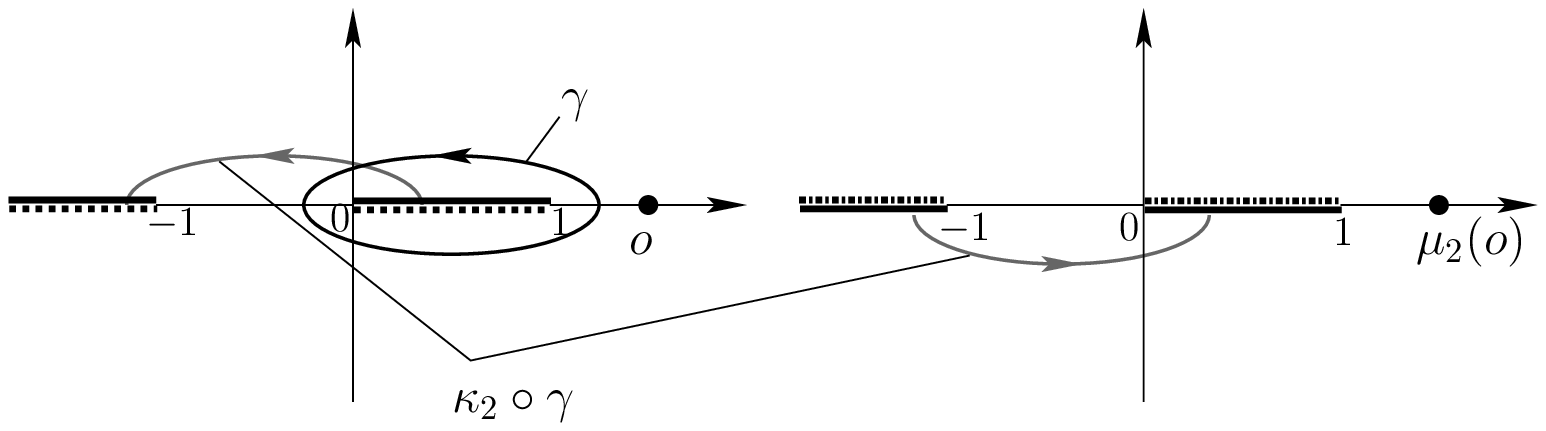}
\end{center}
\begin{quote}\footnotesize
 The projection of the loops $\gamma$ and $\kappa_2\circ\gamma$
 to the $z$-plane are shown.
 The left-hand (resp.~right-hand)
 figure shows the $z$-plane such that $\arg w=0$ 
 ($\arg w=2k\lambda$) when $z>1$.
\end{quote}
\caption{%
 The loops $\gamma$ and $\kappa_2\circ\gamma$.
}
\label{fg:loop}
\end{figure} 
Then we have the following:
\begin{lemma}\label{lem:fundamental-group}
 The fundamental group $\pi_1(M_k)$ of $M_k$ is
 generated by 
 \[
     [(\kappa_1)^j\circ\gamma]
      \quad
      \text{and}
      \quad
     [(\kappa_1)^j\circ\kappa_2\circ\gamma]
         \qquad (j=0,\dots,k).
 \]
\end{lemma}
\subsection{The Weierstrass data}
\label{sub:weier-fk}
We now take the Weierstrass data 
\begin{equation}\label{eq:weier-fk}
   G = c \frac{w}{z} , \qquad 
   \eta = \frac{dz}{w}  \qquad (c\in\R_+)
\end{equation}
on $M_k$, where $c$ is a positive real constant to be determined
in \eqref{eq:c-k}.
Define a holomorphic $2$-differential $Q$
\begin{equation}\label{eq:hopf-fk}
  Q := \eta\,dG = 
    \left(\frac{ck}{k+1}\right)
    \frac{z^2+1}{z^2(z^2-1)}\,dz^2.
\end{equation}
We call $Q$ the {\em Hopf differential\/} because
it will be the Hopf differential of the maxface when the construction
is completed.
The orders of $G$, $\eta$, $G\eta$, $G^2\eta$ and $Q$
are listed as in Table~\ref{tb:g-eta-q},
where $\Ord \omega=m$ (resp.\ $-m$) for a positive integer $m$
if $\omega$ has a zero (resp.\ a pole) of order $m$.
Then one can conclude that
$\deg G=2k$ for any $k\geq 1$.
\begin{table}
\[
     \begin{array}{|c||c|c|c|c|c|}
     \hline
       (z,w) & (0,0) & (\infty,\infty) & (1,0) & (-1,0) & (\pm i,*) \\
     \hline
     \hline      
       \Ord G  &    -k &  -k &    k &  k &    0 \\\hline
       \Ord\eta &  k-1 &  k-1  &    0   &   0  &    0 \\\hline
       \Ord G\eta & -1 &  -1 &    k &  k &    0 \\\hline
       \Ord G^2\eta &-(k+1) & -(k+1) &2k & 2k &  0 \\\hline
       \Ord	Q     & -2 & -2 &   k-1 & k-1 &  1 \\\hline
    \end{array}
\]
\caption{Orders of $G$, $\eta$, $G\eta$, $G^2\eta$ and $Q=\eta\,dG$.}
\label{tb:g-eta-q}
\end{table}

Let 
\begin{equation}\label{eq:f-k-c}
   \hat{f}_{k,c}:=\Re\int \Phi\colon{}\widetilde M_k\longrightarrow\R^3_1,
\end{equation}
where $\Phi$ is the $\C^3$-valued $1$-form as in \eqref{eq:weier-diff}
obtained by $(G,\eta)$ in \eqref{eq:weier-fk}.
\begin{lemma}\label{lem:complete}
 Suppose $\hat{f}_{k,c}$ as in \eqref{eq:f-k-c} is well-defined on $M_k$.
 Then it is a complete maxface of genus $k$ with $2$ ends.
 Moreover, each end is asymptotic to the $k$-fold cover of  
 the Lorentzian catenoid.
\end{lemma}
\begin{proof}
 Observe that $\eta$ and $G^2\eta$ are holomorphic on $M_k$ and have no common
 zeros, and $G^2\eta$ has poles of order $k+1$ at $(z,w)=(0,0)$ and
 $(\infty,\infty)$.
 Thus the metric as in \eqref{eq:normalized} gives a complete Riemannian
 metric on $M_k$.
 Hence  $\hat{f}_{k,c}$ is a weakly complete maxface.
 Moreover, since $G$ has poles at $(z,w)=(0,0)$ and $(\infty,\infty)$,
 $|G|\neq 1$ at the ends.
 Then 
 $\hat f_{k,c}$ is complete because of Fact~\ref{fact:complete}.
 At each end, $Q$ has a pole of order $2$ and the ramification order 
 of $G$ is $k$.
 Then the Weierstrass representation \eqref{eq:maxface} yields that
 the end is asymptotic to the $k$-fold cover of an end of 
 the Lorentzian catenoid.
\end{proof}
\begin{remark}\label{eq:re-Q}
 The Hopf differential $Q(z)$ is real if $z\in\R$ or $z\in i\R$.  
 Also, $Q(z)$ is pure imaginary if $|z|=1$.  
 Then the image of the real and imaginary axes on the $z$-plane 
 are planar geodesics with respect to the first
 fundamental form of the surface,
 and the image of the unit circle consists of line segments joining
 singular points.
\end{remark}

\subsection{The period problem}
\label{sub:period}
In this section, we shall solve the period problem \eqref{eq:period-1}.
The following lemma can be obtained by straightforward
calculations:
\begin{lemma}\label{lem:symm-weier}
 Let $\Phi$ be as in \eqref{eq:weier-diff} for the data 
 \eqref{eq:weier-fk}.
 Then for automorphisms $\kappa_j$ in \eqref{eq:symms} $(j=1,2)$,
 it holds that
 \[
    \kappa_1^*\transpose{\Phi} =
    \begin{pmatrix}
       1 & 0 & 0\\
       0 & \cos 2k\lambda & -\sin 2k\lambda \\
       0 & \sin 2k\lambda & \hphantom{-}\cos 2k\lambda 
    \end{pmatrix}\transpose{\Phi},\quad
    \kappa_2^*\transpose{\Phi} =
    \begin{pmatrix}
       1 & 0 & 0\\
       0 &\hphantom{-} \cos k\lambda & \sin k\lambda \\
       0 &-\sin k\lambda & \cos k\lambda 
    \end{pmatrix}\transpose{\Phi},
 \]
 where $\transpose{~}$ stands for transposition.
\end{lemma}
Because of the matrices in Lemma~\ref{lem:symm-weier}, 
Lemma \ref{lem:fundamental-group} implies that
\begin{lemma}\label{lem:period}
 The map $\hat f_{k,c}$ defined in \eqref{eq:f-k-c} is single-valued on
 $M_k$ 
 if and only if the period condition \eqref{eq:period-1} holds for
 the single loop $\gamma$ in Figure~\ref{fg:loop}.
\end{lemma}
Now, we determine the value of $c\in\R_+$ 
in \eqref{eq:weier-fk} such that the condition in Lemma~\ref{lem:period}
holds.

Since $G\eta=(c/z)dz$, it holds that
\[
   \Re\oint_{\gamma} G\eta = \Re\oint_{\gamma}
                  c\,d\log z =
		  \Re\left(2\pi i c\right)=0.
\]
So the condition \eqref{eq:period-1} for $\gamma$ is equivalent to
\begin{equation}\label{eq:period}
 \oint_{\gamma}\eta + \overline{\oint_{\gamma}G^2\eta} = 0 .
\end{equation}
To calculate the integrals of the left-hand side, 
we take two paths on $M_k$ as 
\begin{align*}
 \gamma_1 &:= \left\{\left. (z,w)=
                  \left(t,e^{ki\lambda}\,
                   \sqrt[k+1]{t(1-t^2)^k}\right)
                  \,\right|\,t:1\mapsto 0\right\},\\
 \gamma_2 &:=
          \left\{\left. (z,w)=
          \left(t,e^{-ki\lambda}\,\sqrt[k+1]{t(1-t^2)^k}\right)
          \,\right|\,t:0\mapsto 1\right\},
\end{align*}
where we consider $\sqrt[k+1]{t(1-t^2)^k}$ as a positive real number.
Roughly speaking, a closed loop $\gamma_1*\gamma_2$
(the definition of $\gamma_1*\gamma_2$ is given in the appendix)
is homotopic to $\gamma$ on $\overline{M}_k$.
Here, $\gamma_1$ and $\gamma_2$ are parametrized regular curves on 
$\overline{M}_k$ from $t=1$ to $t=0$ and $t=0$ to $t=1$,
respectively.
Adding an exact form to $G^2\eta$, we have 
\[
   G^2\eta + c^2\frac{k+1}{k}d\left(\frac{w}{z}\right)
      =-2c^2\frac{w}{1-z^2}dz.
\]
Since the right-hand side does not have poles
at $(z,w)=(0,0)$ and $(1,0)$, 
we have
\begin{align*}
  \oint_{\gamma}G^2\eta &=
  -2c^2\oint_{\gamma}\frac{w}{1-z^2}\,dz
  =
    -2c^2\left(
         \int_{\gamma_1}\frac{w}{1-z^2}\,dz +
         \int_{\gamma_2}\frac{w}{1-z^2}\,dz 
         \right)\\
  &= -c^2 \cdot 4i\sin k\lambda\cdot A_k\qquad
       \left(A_k:=\int_0^1\sqrt[k+1]{\frac{t}{1-t^2}}\,dt\right).
\end{align*}
On the other hand,
since $\eta$ does not have poles at $(z,w)=(0,0)$ and $(1,0)$,
we have
\[
   \oint_{\gamma}\eta =
   \int_{\gamma_1}\eta +
   \int_{\gamma_2}\eta
  =2i\sin k\lambda\cdot B_k
 \qquad \left(B_k:=\int_0^1\frac{dt}{\sqrt[k+1]{t(1-t^2)^k}}\right).
\]
Therefore, if we set 
\begin{equation}\label{eq:c-k}
     c=c_k:= \sqrt{\frac{B_k}{2A_k}}>0, 
\end{equation}
then \eqref{eq:period} holds for $\gamma$ and hence 
\begin{equation}\label{eq:fk}
   \hat f_k := \hat f_{k,c_k}
\end{equation}
is single-valued on $M_k$.
By Lemma~\ref{lem:complete}, 
$\hat{f}_k$ is a complete maxface of genus $k$ with $2$ ends.
See Figures \ref{fg:k=1}, \ref{fg:k=2} and \ref{fg:k=3}.

Now, we prove that $\hat f_1$ is embedded (in the wider sense,
as in Definition~\ref{def:embedded}).
When $k=1$, each end is asymptotic to a Lorentzian catenoid,
that is, each end has no self-intersection.
Moreover, the $x_0$-component of $\hat f_1$ is calculated as 
\[
    x_0:=-2\Re\int_o^p G\eta = -2 c_1\log |z|+\mbox{constant},
\]
where $p=(z,w)\in M_1$.
Hence $x_0\to +\infty$ (resp.\ $-\infty$) when $(z,w)\to (0,0)$
(resp.\ $(\infty,\infty)$).
This means that the two ends have no self-intersection outside 
a compact set in $M_1$.
Hence $\hat{f}_1$ is embedded in the wider sense.

\begin{figure}
\begin{center}
\begin{tabular}{cc}
 \includegraphics[width=.40\linewidth]{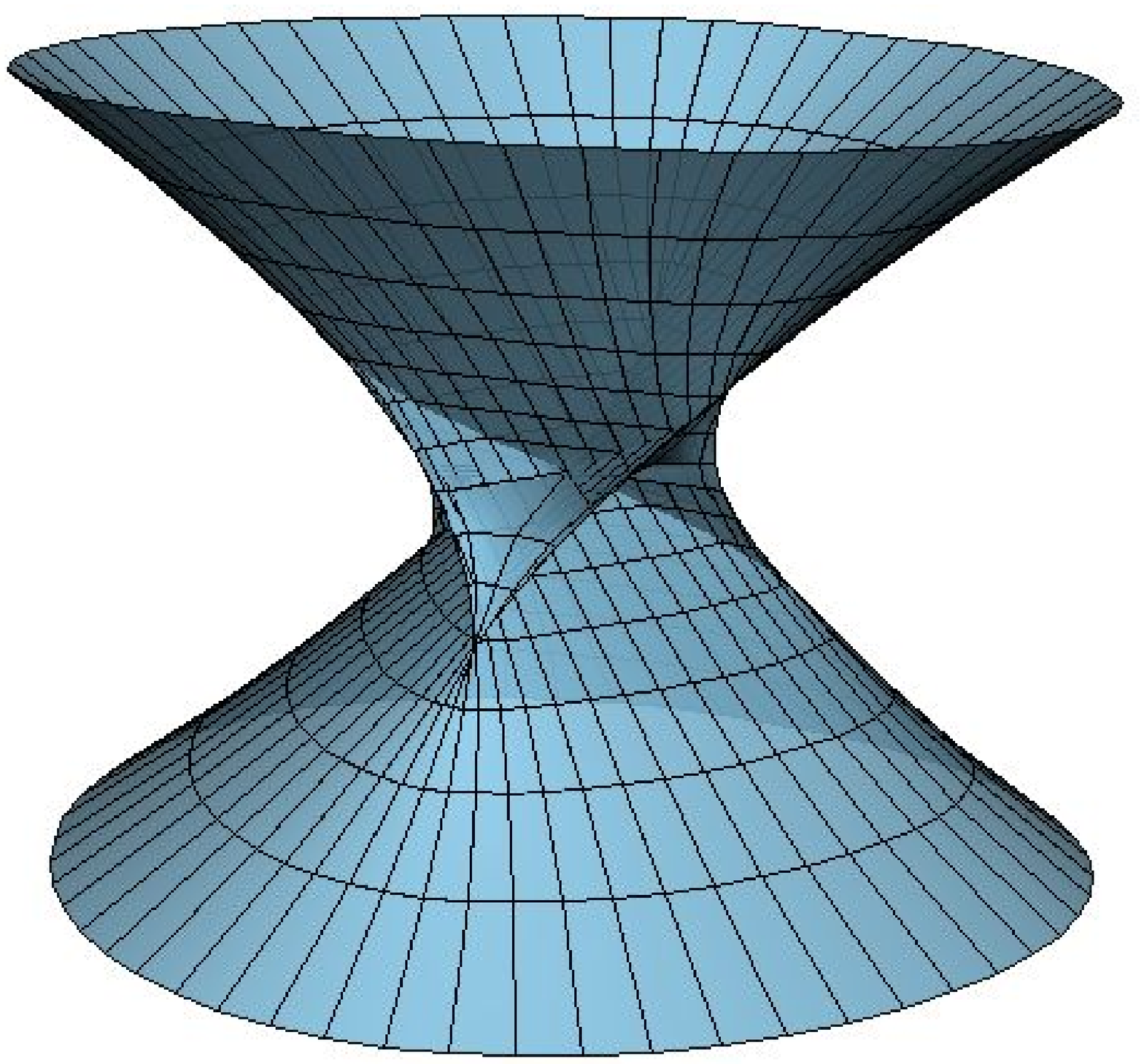} &
 \includegraphics[width=.40\linewidth]{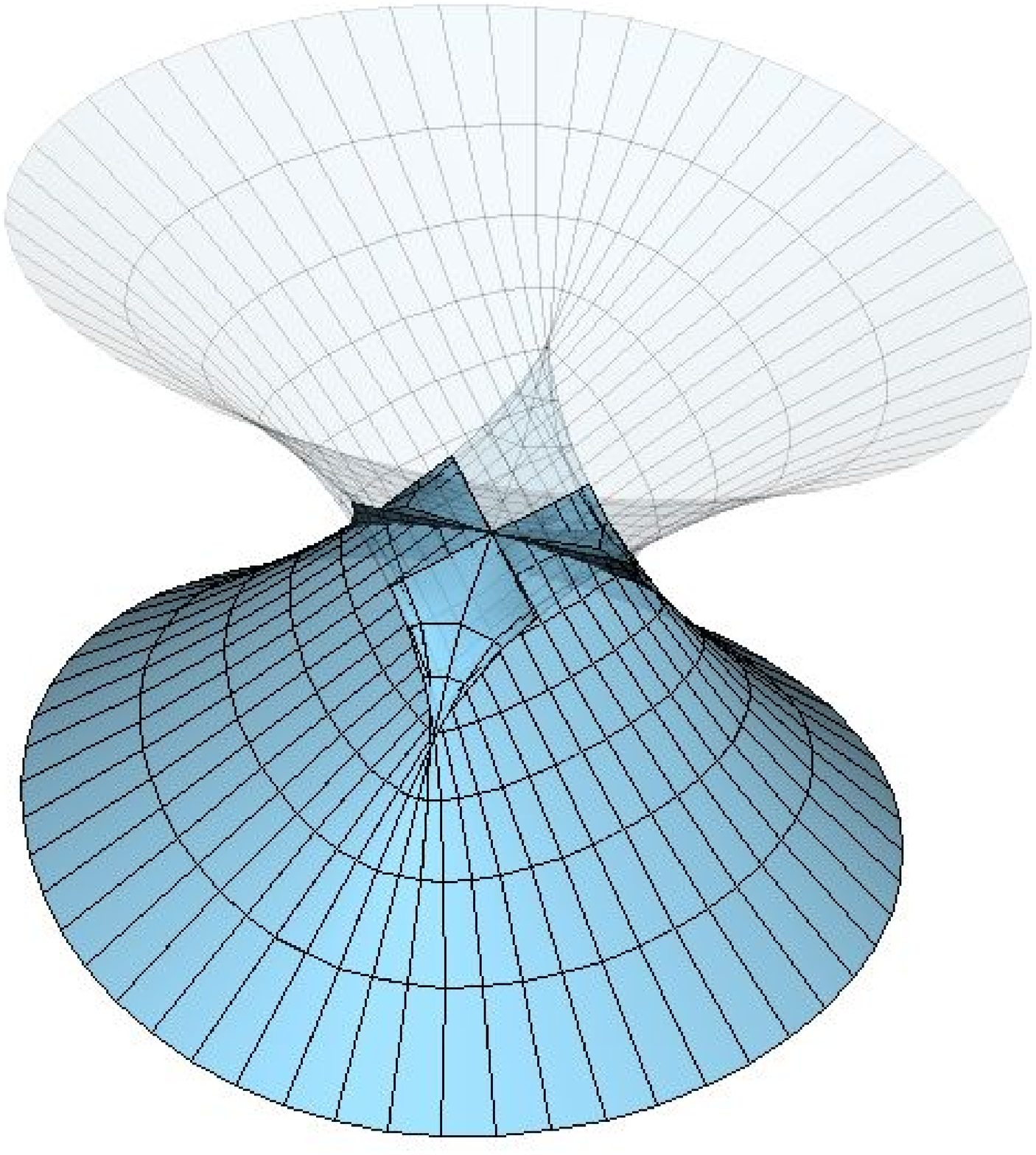} 
\end{tabular}
\caption{%
 The example for $k=2$  and half of it.}
\label{fg:k=2}
\end{center}
\end{figure}
\begin{figure}
\begin{center}
\begin{tabular}{cc}
 \includegraphics[width=.40\linewidth]{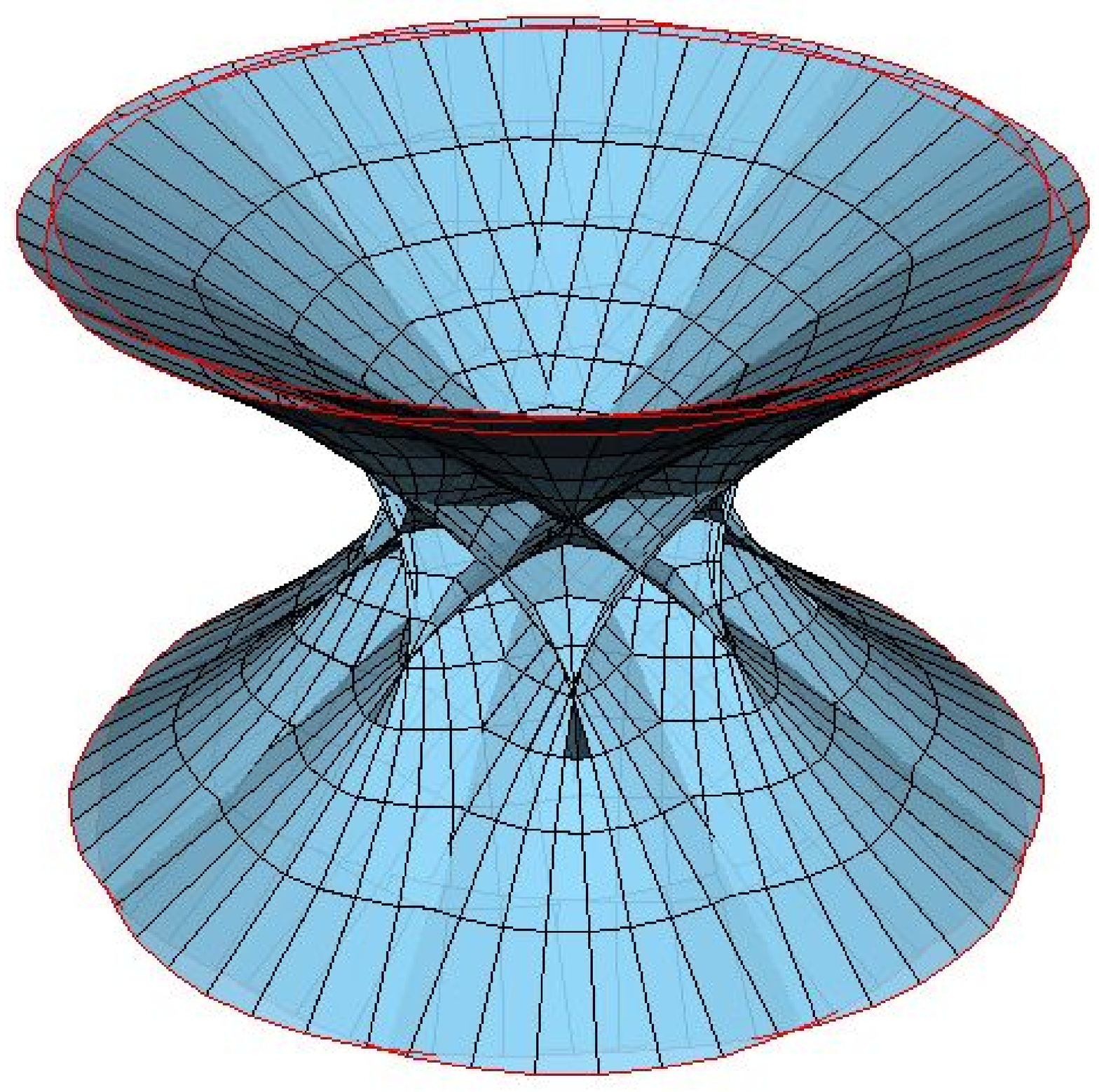} &
 \includegraphics[width=.40\linewidth]{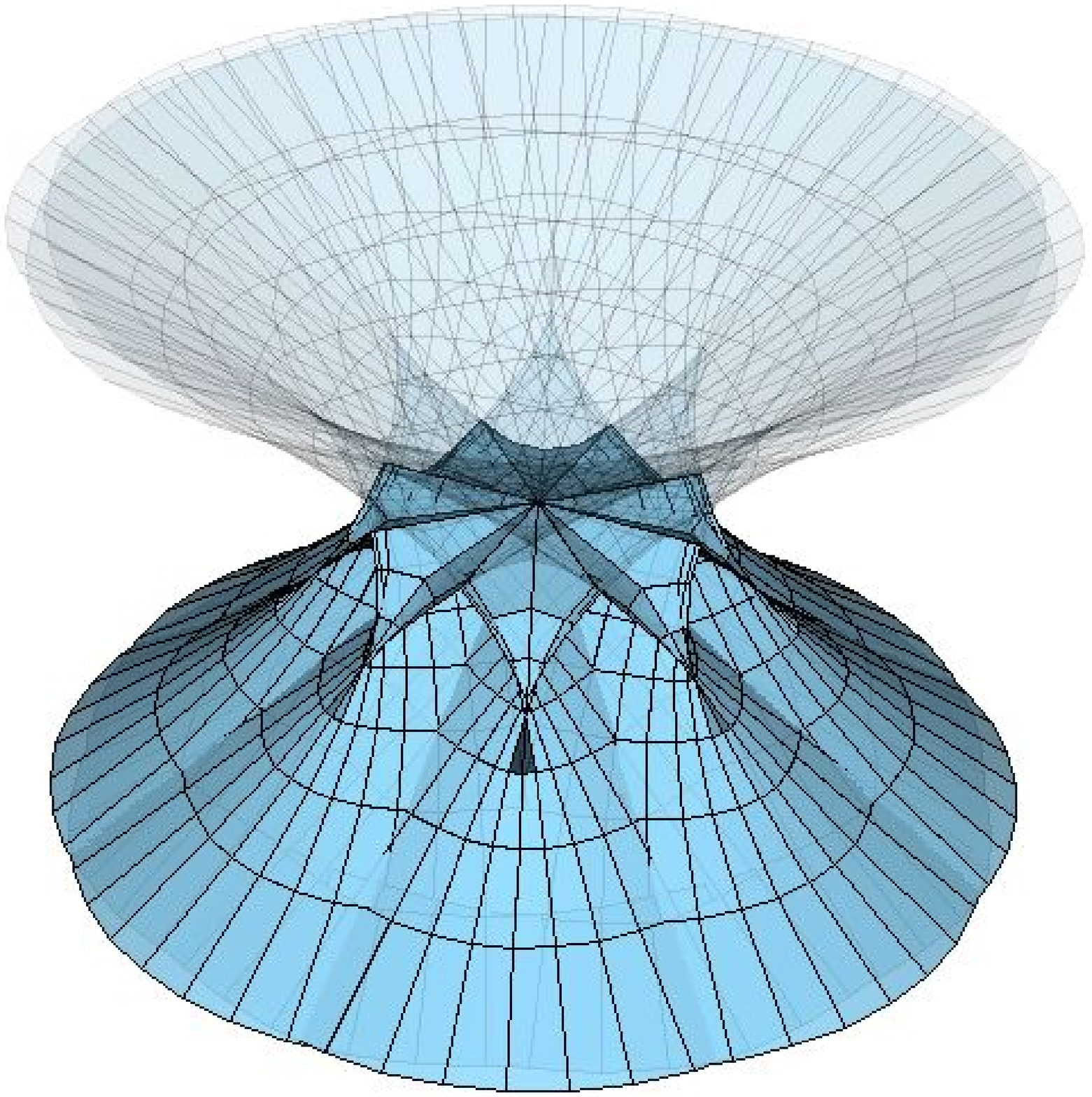} 
\end{tabular}
\caption{%
 The example for $k=3$ and half of it.}
\label{fg:k=3}
\end{center}
\end{figure}

\subsection{Reduction for the even genus case}
\label{sub:double}
When $k$ is even, the Riemann surface $\overline{M}_k$ of genus $k=2m$ is 
reduced to the Riemann surface
\[ 
  \overline{M}'_k = 
  \left\{ 
  (Z,W) \in (\C \cup \{ \infty \})^2 \, ;  
  \, W^{2m+1} = Z^{m+1} \left( Z-1 \right)^{2m}  \right\}  
\] 
of genus $m$,
where $m$ is a positive integer.  
As a submanifold of $(\C \cup \{ \infty \})^2$,
$\overline{M}'_k$ has singular points at $z=0$, $1$ and $z=\infty$.
However, in a similar way to the case of $\overline{M}_k$ (see
\eqref{eq:coords}), the structure of a Riemann surface can be introduced
to $\overline{M}'_k$.
The map
$\pi:\overline{M}'_k\ni (Z,W)\longmapsto Z\in \C\cup \{\infty\}$ 
is a meromorphic function of degree $2m+1$ with
total branching number $6m$.
Then, by the Riemann-Hurwitz relation, the genus of
$\overline{M}'_k$ is equal to $m$.
Let
\begin{equation}\label{eq:reduced-riemann-surface}
   M'_k = \overline{M}'_k \setminus \{ (0,0),(\infty,\infty) \} . 
\end{equation} 
Define a map $\varpi_k$ from $M_k$ of genus $k=2m$ 
into $M'_k$ of genus $m$ as
\begin{equation}\label{eq:double-cover}
  \varpi_k:M_k\ni (z,w)\longmapsto (Z,W)=(z^2,zw)\in M'_k.
\end{equation}
Then $\varpi_k (z,w)=\varpi_k(-z,-w)$ for any $(z,w)\in\overline{M}_k$ 
and hence $\varpi_k$ is a double cover. 

Let $(G_1,\eta_1$) be the Weierstrass data on $M'_k$ given by
\[ 
   G_1 = c \frac{W}{Z} , \qquad 
   \eta_1 = \frac{dZ}{2W} ,
\]
which satisfy
\[
    G = \varpi_k^*G_1=G_1\circ\varpi_k,\qquad 
    \eta=\varpi_k^*\eta_1.
\]
The data $(G_1,\eta_1)$ for $c=c_k$  as in \eqref{eq:c-k}
gives a maxface $f_k\colon{}M'_k\to \R^3_1$ such that
$\hat f_k=f_k\circ\varpi_k$.
Since $\deg G_1=m$,
all ends are embedded if and only if $m=1$
(cf.~\cite[Theorem 4.11]{UY}).  
Moreover, if this is the case, embeddedness of $f_2$ can be shown
in a similar way to the case of $\hat f_1$.
Compare Figures \ref{fg:k=1} (for $k=1$)
and \ref{fg:k=2} for the case $k=2$.

\subsection{Swallowtails and cuspidal cross caps of $f_k$}
\label{sub:sing}
In this section, we investigate the properties of the
singular points.
\begin{lemma}\label{lem:ck}
 The number $c_k$ in \eqref{eq:c-k} satisfies
 \[
     c_1 >  1 \qquad\text{and}\qquad
     c_k >  \frac{k^{\frac{1}{2(k+1)}}}{\sqrt{2}}
	      \left(%
	        \frac{k}{k-1}
	      \right)^{\frac{k-1}{2(k+1)}}\quad  (k\geq 2).
 \]
 In particular, for each $k=1,2,3,\dots$, it holds that
 \begin{equation}\label{eq:rho-val}
     0 < \rho^{}_k < 2\qquad \left(\rho^{}_k := c_k^{\frac{-2(k+1)}{k}}\right).
 \end{equation}
\end{lemma}
\begin{proof}
 First, we consider the case $k\geq 2$.
 Set
 \[
    V(s,t):=
       t(1-t^2)^k\left(\frac{1}{t(1-t^2)^k}-\frac{s^{k+1}t}{1-t^2}\right)
     = 1-s^{k+1}t^2(1-t^2)^{k-1}
 \]
 for $s\in\R$ and $t\in[0,1]$.
 Then for each fixed value $s$, $V(s,t)$
 attains a minimum at $t=1/\sqrt{k}$, and 
 \[
    V\left(s,\frac{1}{\sqrt{k}}\right)
    = 1-\frac{s^{k+1}}{k}\left(\frac{k-1}{k}\right)^{k-1}.
 \]
 Hence if we set 
 \begin{equation}\label{eq:s-val}
     s_k := k^{\frac{1}{k+1}}\left(\frac{k}{k-1}\right)^{\frac{k-1}{k+1}},
 \end{equation}
 then
 $V(s_k,t)\geq 0$ ($t\in[0,1]$).
 Hence 
 \[
    \sqrt[k+1]{\frac{t}{1-t^2}}-s_k\frac{1}{\sqrt[k+1]{t(1-t^2)^k}}>0
 \]
 holds on $t\in[0,1]$,
 and then, we have
 \[
     c_k =\sqrt{\frac{A_k}{2B_k}}
         >\sqrt{\frac{s_k}{2}}
         = \frac{1}{\sqrt{2}}
           k^{\frac{1}{2(k+1)}}\left(\frac{k}{k-1}\right)^\frac{k-1}{2(k+1)}.
 \]
 This implies 
 \[
    0<\rho^{}_k < 2^{\frac{k+1}{k}}k^{-\frac{1}{k}}
             \left(\frac{k}{k-1}\right)^{\frac{1-k}{k}}
           = 2 \left(\frac{2}{k}\right)^{\frac{1}{k}}
             \left(\frac{k-1}{k}\right)^{\frac{k-1}{k}}<2.
 \]

 Next, consider the case $k=1$:
 \begin{align*}
  A_1-2B_1 &=
     \int_0^1\left(
              \frac{1}{\sqrt{t(1-t^2)}}-2\sqrt{\frac{t}{1-t^2}}\,dt
             \right)
     =\int_0^1
       \left(\frac{1-2t}{\sqrt{t(1-t^2)}}\right)\,dt\\
  &= \sqrt{2}\int_{-1}^2 
         \frac{u\,du}{\sqrt{(1-u^2)(3-u)}}\\
  &= \sqrt{2}\left(
         \int_0^1
         \frac{u\,du}{\sqrt{(1-u^2)(3-u)}}+
         \int_0^1
         \frac{v\,dv}{\sqrt{(1-v^2)(3+v)}}\right)\\
  &= \sqrt{2}\int_0^1
        \frac{u}{\sqrt{1-u^2}}
        \left(
          \frac{1}{3-u}-\frac{1}{3+u}
        \right)du >0,
 \end{align*}
 where we put $u=1-2t$ and $v=-u$.
 Hence $c_1 = \sqrt{A_1/(2B_1)}>1$, and then
  $\rho^{}_1<1<2$.
\end{proof}

\begin{lemma}[The singular curve]\label{lem:sing-curve}
 The set of singular points of  $\hat{f}_k$ consists of
 $2$ simple closed curves on $M_k$.
 The projection of the singular set
 onto the $z$-plane is shown in  Figure~\ref{fig:singular-set}.
\end{lemma}
\begin{proof}
 The set of singular points is represented as 
 $\Sigma:=\{p\in M_k\,;\,|G(p)|=1\}$
 (cf.\ Fact~\ref{fact:weier}).
 Here, by \eqref{eq:weier-fk},
 the condition $|G|=1$ is equivalent to 
 $\left|c_k w/z\right|=1$, and then, by \eqref{eq:riemann}, 
 it is equivalent to 
 \[
       \left|z-\frac{1}{z}\right|^2 = 
        r^2 + \frac{1}{r^2}-2\cos 2\theta=\rho^{}_k
          \qquad \left(\rho^{}_k=c_k{}^{\frac{-2(k+1)}{k}}\right),
 \]
 where $z=re^{i\theta}$.
 Since $\rho^{}_k\in (0,2)$ by Lemma~\ref{lem:ck}, one can set
 \[
     \Gamma_{k} := \arcsin \frac{\sqrt{\rho^{}_k}}{2}\in
     \left(0,\frac{\pi}{4}\right).
 \]
 Thus, the projection of the singular set onto the $z$-plane
 consists of two simple closed curves in
 the $z$-plane, contained in two angular domains:
 \[
       \Delta_+:=\{-\Gamma_{k}<\arg z < \Gamma_{k}\},\quad
       \text{and}\quad
       \Delta_-:=\{\pi-\Gamma_{k}<\arg z < \pi+\Gamma_{k}\}.
 \]
 We denote two subsets of the singular set $\Sigma$ as 
 \[
     \Sigma_+ := \{p\in\Sigma\,;\,z(p)\in \Delta_+\},\quad
     \text{and}\quad
     \Sigma_- := \{p\in\Sigma\,;\,z(p)\in \Delta_-\},
 \]
 see Figure~\ref{fig:singular-set},
 where $z\colon{}M_k\ni (z,w)\mapsto z\in\C$ is the projection.

 The loop $z(\Sigma_+)$ on the $z$-plane is surrounding the
 branch point $z=1$ of the projection $z$.
 Then, the inverse image $\Sigma_+$ of this loop  
 consists of $k+1$ copies of $z(\Sigma_+)$  which forms a single
 loop in $M_k$,
 that is, $\Sigma_+$ is a loop in $M_k$.
 Similarly so is $\Sigma_-$.
\end{proof}
\begin{figure}
\begin{center}
 \includegraphics[width=0.7\textwidth]{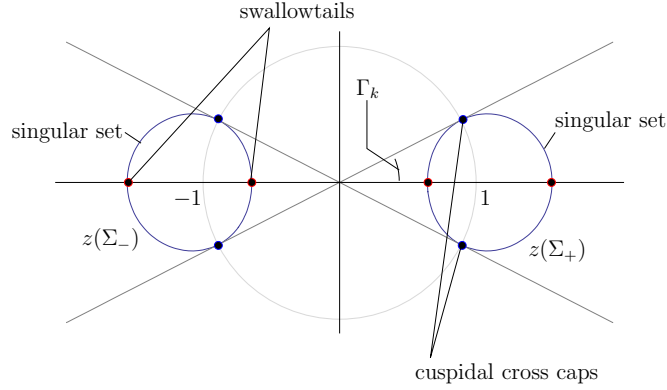}
\end{center}
\caption{The singular set of $\hat{f}_k$.}
\label{fig:singular-set}
\end{figure}
\begin{lemma}\label{lem:swallowtails}
 On each connected component of the singular set in
 Lemma~\ref{lem:sing-curve},
 there are
 $2(k+1)$ swallowtails and 
 $2(k+1)$ cuspidal cross caps.
 Singular points other than these points are cuspidal edges.
\end{lemma}
\begin{proof}
 Let $\alpha$ and $\beta$ be as in \eqref{eq:criteria}:
 \[
    \alpha:= \frac{dG}{G^2\eta}
           = \frac{Q}{(G\eta)^2} = \frac{z^2+1}{z^2-1}
           = 1-\frac{2}{z^2-1},\qquad
    \beta := \left(\frac{G}{dG}d\alpha\right)
           = z^2 -\frac{1}{z^2}.
 \]
 Then 
 \[
   2\Re\alpha = \frac{z^2+1}{z^2-1} + \frac{\bar z^2+1}{\bar z^2-1}
       = 2\frac{|z|^4-1}{|z^2-1|^2}=0 \qquad 
   \text{if and only if}\qquad |z|=1,
 \]
 and 
 $\Im\alpha = 0$ if and only if $z\in \R\cup i\R$.
 On the other hand, since
 $\Re \beta = \left(r^2-r^{-2}\right)\cos 2\theta$ $(z=re^{i\theta})$,
 \[
    \Re\beta=0 \qquad \text{if and only if}\qquad
    |z|=1\quad\text{or}\quad
    \arg z=\frac{\pi}{4}, \frac{3\pi}{4},\frac{5\pi}{4},\frac{7\pi}{4}.
 \]
 Finally, since 
 $\Im \beta =  \left(r^2+r^{-2}\right)\sin 2\theta$,
 $\Im\beta=0$ if and only if $z\in \R\cup i\R$.

 Then by Fact~\ref{fact:maxface-sing}, 
 there are two swallowtails
 and two cuspidal cross caps  on $\Sigma_+$, on the $z$-plane.
 Since $M_k$ is a $(k+1)$-fold branched cover of the $z$-plane,
 we have the conclusion.
\end{proof}
When $k=2m$ is even, 
both $\hat f_m$ and $f_k$ give maxfaces of genus $m$ with $2$ ends.
Moreover, each end is asymptotic to the $m$-fold cover of 
the Lorentzian (elliptic) catenoid.
\begin{corollary}
 When $k=2m$, $\hat f_m$ and $f_k$ are not congruent.
\end{corollary}
\begin{proof}
 The number of
 swallowtails on the image of $\hat f_m$ is
 $4m+4$. On the other hand,
 the number of
 swallowtails on the image of $f_{k}$ ($k=2m$) is
 $4m+2$, which proves the assertion.
\end{proof}
\section{Proof of the second part of Theorem~\ref{thm:main}}
\label{sec:deform}
In this section, we shall deform the maxfaces given in the previous 
section to \cmcone{} faces in de Sitter space.
The technique we use here is similar to that in \cite{RUY}.
However, we need much more technical arguments because
the maxfaces $f_k$ in the previous section do not have
{\em non-degenerate period problem\/} as in \cite[Section 5]{RUY}.
So, we accomplish the deformation by computing the derivative
of the period matrices.
\subsection{Preliminaries}
\label{sub:cmcone}
First, we recall some fundamental facts about \cmcone{} faces
in de Sitter space.
For detailed expressions, see \cite{F,F2}, or \cite{FRUYY}.
Let $\R^4_1$ be the Minkowski $4$-space
with the metric $\inner{~}{~}$ of signature $(-,+,+,+)$.
Then de Sitter $3$-space is expressed as 
\[
  S^3_1=
  \{X \in\R^4_1\, ; \, \inner{X}{X}=1\}
\]
with metric induced from $\R^4_1$, 
which is a simply-connected Lorentzian $3$-manifold with constant
sectional curvature $1$.
We identify $\R^4_1$ with the set of 
$2\times 2$ Hermitian matrices
$\Herm(2)$ by
\begin{equation}\label{eq:herm-mink}
    \R^4_1 \ni  (x_0,x_1,x_2,x_3)\leftrightarrow
    \begin{pmatrix}
      x_0+x_3           & x_1+i x_2 \\
      x_1-i x_2 & x_0-x_3
     \end{pmatrix}\in \Herm(2).
\end{equation}
Then de Sitter $3$-space is represented as 
\begin{align*}
  S^3_1&=\{X\in\Herm(2)\,;\,\det X=-1\}\\
              &=\{Fe_3F^*\,;\,F\in \SL_2\C\}=\SL_2\C/\SU_{1,1}\qquad
   \left(e_3:=\begin{pmatrix}
	       1 & \hphantom{-}0\\
	       0 & -1
	     \end{pmatrix}
	      \right).
\end{align*}
The first author \cite{F} introduced the notion of \cmcone{} faces in
$S^3_1$, which corresponds to maxfaces in $\R^3_1$.

To state the Weierstrass-type representation formula, we prepare
some notions:
\begin{definition}\label{eq:admissible}
 A pair $(G,Q)$ of a meromorphic function $G$ and a holomorphic
 $2$-differential $Q$ on $M$ is said to be 
 {\em admissible\/} if 
 \begin{equation}\label{eq:cmc-lift}
    ds^2_{\#}=\bigl(1+|G|^2\bigr)^2\left|\frac{Q}{dG}\right|^2
 \end{equation}
 is a (positive definite) Riemannian metric on $M$.
\end{definition}
\begin{proposition}\label{prop:cmc-rep}
 Let $M$ be a Riemann surface 
 and $(G,Q)$ an admissible pair on $M$.
 Let $F=(F_{ij})\colon{}\widetilde M\to \SL_{2}\C$ 
 be a holomorphic map of the universal cover $\widetilde M$ of $M$
 such that
\begin{equation}\label{eq:cmc-ode}
   dF \,F^{-1} = \Psi \qquad
    \left(
    \Psi := \begin{pmatrix}
	       G & -G^2 \\
	       1 & -G\hphantom{^2}
	     \end{pmatrix}\frac{Q}{dG}
    \right) .
\end{equation}
 Then $F$ is a null holomorphic immersion,
 that is, $F$ is a holomorphic immersion such that
 $\det (dF/dz)$ vanishes identically for each local complex coordinate  
 $z$ on $M$.
 And
\begin{equation}\label{eq:cmc-face}
    f:= Fe_3 F^*\colon{}\widetilde M\longrightarrow S^3_1
\end{equation}
 is a \cmcone{} face if 
 $|g|$ is not identically $1$, where 
 $g$ is a meromorphic function on $\widetilde M$  defined by
\begin{equation}\label{eq:second}
   g:=-\frac{dF_{12}}{dF_{11}}=-\frac{dF_{22}}{dF_{21}}. 
\end{equation}
 The induced metric $ds^2$ and the second fundamental form $\secondff$
 are expressed as
\begin{equation}\label{eq:cmc-f-forms}
    ds^2 = (1-|g|^2)^2\,\left|\frac{Q}{dg}\right|^2,\qquad
     \secondff = Q + \overline Q + ds^2,
\end{equation}
 respectively.
 Conversely, any \cmcone{} face is obtained in this manner.
\end{proposition}
\begin{remark}
 \begin{enumerate}
  \item The equation \eqref{eq:cmc-ode} should be regarded as
	an equation on the universal cover $\widetilde M$
	(see \eqref{eq:app-ode} in the appendix).
	However, for simplicity, we use the notation here.
  \item The condition that $|g|$ is not identical to $1$
	is necessary to avoid the example all of whose points are
	singular points.
	Such an example is unique up to isometry, whose image
	is a lightlike line in $S^3_1$
	(see \cite[Remark 1.3]{FRUYY}).
 \end{enumerate}
\end{remark}
\begin{proof}[Proof of Proposition~\ref{prop:cmc-rep}]
 Though the statement of this proposition is mentioned in 
 \cite[Theorem 1.2]{F}, the proof is not given there.
 So we give a proof here.
 By \eqref{eq:cmc-ode}, $F$ is a null holomorphic map 
 from $\widetilde M$ into $\SL_2\C$,
 and admissibility implies that $F$ is an immersion.
 Moreover, if $|g|$ is not identically $1$,
 $\det[Fe_3F^*]$ does not vanish identically.
 Hence by Proposition 4.7 in \cite{F}, $f=Fe_eF^*$ 
 is a \cmcone{} face.
\end{proof}
The meromorphic function $G$, the holomorphic $2$-differential $Q$
and the (multi-valued) meromorphic function $g$ in
Proposition~\ref{prop:cmc-rep} are
called the {\em hyperbolic Gauss map}, the {\em Hopf differential},
and the {\em secondary Gauss map}, respectively.
We call $F$ the {\em holomorphic null lift\/} of the \cmcone{}
face $f$.
These holomorphic data are related by
\begin{equation}\label{eq:schwarz}
  S(g) - S(G) = 2 Q,\qquad 
 \left(
  S(h):=\left[
         \left(\frac{h''}{h'}\right)'-\frac{1}{2}
         \left(\frac{h''}{h'}\right)^2
        \right]\,dz^2,~
    ~' = \frac{d}{dz}
 \right),
\end{equation}
where $z$ is a local complex coordinate on $M$ and $S(\cdot)$ is the
Schwarzian derivative.

For an admissible pair $(G,Q)$ on $M$, there exists a representation
$\rho^{}_F\colon{}\pi_1(M)$ $\to\SL_2\C$ associated with the solution $F$ 
of \eqref{eq:cmc-ode} as in Proposition~\ref{prop:monodromy-ode}
in the appendix:
\begin{equation}\label{eq:monodromy-repr}
    F\circ\tau = F\rho^{}_F(\tau)^{-1}\qquad
    \bigl(\tau\in\pi_1(M)\bigr),
\end{equation}
where $\tau\in\pi_1(M)$ is considered as a covering transformation
of the universal cover $\widetilde M$, as in
\eqref{eq:covering-transf}.

To deform maxfaces to \cmcone{} faces,
the following facts, which are proved in \cite{RUY} for \cmcone{}
surfaces in $H^3$, play important roles:
\begin{lemma}[{cf.\ \cite[Lemma 4.8]{RUY}}]
\label{lem:infinitesimal-representation}
 Take an admissible pair $(G,Q)$ on a Riemann surface $M$.
 Then $(G,Q_t=tQ)$ is also an admissible pair for all
 $t\in\R\setminus\{0\}$.
 Let 
 $F:=F_t\colon{}\widetilde M\to\SL_2\C$ be a solution 
 of 
\begin{equation}\label{eq:ode-new}
    dF F^{-1}=t \Psi_0, 
    \qquad
    \Psi_0=\begin{pmatrix}
	     G & -G^2 \\
	     1 & -G\hphantom{^2}
	   \end{pmatrix}\frac{Q}{dG},
\end{equation}
 with the initial condition $F_t(o)=\iota(t)$, 
 where $\iota(t)$ is a smooth $\SL_2\C$-valued function in $t$ with
 $\iota(0)=\id$, 
 where 
 \begin{equation}\label{eq:identity}
       \id:=\begin{pmatrix} 1 & 0 \\ 0 & 1 \end{pmatrix}
 \end{equation}
is the identity matrix 
 and $o\in \widetilde M$ is a base point.
 Let  $\rho^{}_t=\rho^{}_{F_t}(\tau)$
 for $\tau\in\pi_1(M)$, where $\rho^{}_{F_t}$ is a representation
 as in \eqref{eq:monodromy-repr}.
 Then it holds that
 \begin{equation}\label{eq:rho-deform}
   \left.\frac{d}{d t}\right|_{t=0}
    (\rho^{}_t)^{-1} = \int_{\gamma}\Psi_0,
 \end{equation}
 and $\gamma$ is a loop in $M$ which represents $\tau$.
\end{lemma}
\begin{proof}
 By \eqref{eq:ode-new}, $F_0$ is a constant map.
 Differentiating \eqref{eq:ode-new}, we have
 $d\dot F = \Psi_0\iota(0)$, where $\dot F=(\partial/\partial t)|_{t=0}F_t$.
 Hence we have
 \[
    \left.\frac{d}{dt}\right|_{t=0}
    \oint_{\gamma} dF_t =
    \oint_{\gamma}\Psi_0.
 \]
 Here, the left-hand side is computed as
 \[
    \left.\frac{d}{dt}\right|_{t=0}
    \biggl(
      F_t(o)(\rho^{}_t)^{-1}-F_t(o)
    \biggr)=
    \left.\frac{d}{dt}\right|_{t=0}
    \biggl(
      \iota(t)(\rho^{}_t)^{-1}-\iota(t)
    \biggr)
    =
   \left.\frac{d}{dt}\right|_{t=0}(\rho^{}_t)^{-1}
   ,
 \]
 because $\iota(0)=\id$.
 Hence we have the conclusion.
\end{proof}

Similar to Definition~\ref{def:complete} in the introduction, 
we define completeness and weak completeness of \cmcone{} faces:
\begin{definition}[%
 {\cite[Definitions 1.2 and 1.3]{FRUYY}}]%
\label{def:complete-cmc}
 A \cmcone{} face $f\colon{}M\to S^3_1$
 is called {\it complete\/}
 if there exists a symmetric $2$-tensor $T$ 
 which vanishes outside a compact set in $M$,
 such that
 $ds^2+T$ is a complete Riemannian metric on $M$,
 where $ds^2$ is the induced metric by $f$ as in \eqref{eq:cmc-f-forms}.
 On the other hand, $f$ is called {\it weakly complete\/}
 if the metric $ds^2_{\#}$ in \eqref{eq:cmc-lift} is complete.
\end{definition}
Like as in the case of maxfaces, we have
\begin{fact}[\cite{UY2}] \label{fact:complete2}
 Let 
 $M$ be a Riemann surface,
 and $f\colon{}M\to S^3_1$ a weakly complete \cmcone{} face.
 Then $f$ is complete if and only if
 there exists a compact Riemann surface $\overline M$ and 
 a finite number of points
 $p_1,\dots,p_n\in \overline M$ such that
 $M$ is conformally equivalent to 
 $\overline M\setminus\{p_1,\dots,p_n\}$,
 and the set of singular points
 $\Sigma=\{p\in M\,;\,|g(p)|=1\}$ is compact.
\end{fact}

\subsection{The holomorphic data}
\label{sub:cmc-data}
 Let $k$ be a positive integer, and take the Riemann surface $M_k$ as in 
 \eqref{eq:m-k}.
 Take a meromorphic function $G$, a holomorphic $1$-form $\eta$
 and a holomorphic $2$-differential $Q_t$ ($t\in\R$) on $M_k$ as 
\begin{equation}\label{eq:cmc-data}
   G := c_k\frac{w}{z},\qquad 
   Q_t:=\frac{t}{c_k}Q
   \left(= \frac{tk}{k+1}\frac{z^2+1}{z^2(z^2-1)}dz^2\right),
\end{equation}
 where $(G,\eta)$ are the Weierstrass data as in \eqref{eq:weier-fk},
 $Q=\eta dG$ 
is as in \eqref{eq:hopf-fk}, and $c=c_k$ is as in \eqref{eq:c-k}.
 We set
\begin{equation}\label{eq:cmc-data-alpha}
 \Psi:=t\Psi_0,\qquad
        \Psi_0:=
	  \frac{1}{c_k}
	  \begin{pmatrix}
	    G & - G^2 \\
	    1 & - G\hphantom{^2}
	  \end{pmatrix}\frac{Q}{dG}.
\end{equation}
 Then one can easily show that
 $(G,Q)$ (and then $(G,Q_t)$ for $t\neq 0$) 
 is an admissible pair on $M_k$, and 
 the metric \eqref{eq:cmc-lift} is complete.

 Then the second part of Theorem~\ref{thm:main} is a conclusion of the
 following 
\begin{proposition}\label{prop:deform-2}
 For each positive integer $k=1,2,3\dots$, there exists a positive
 number $\varepsilon=\varepsilon(k)$
 such that for each real number $t$ with $0<|t|<\varepsilon$,
 there exists a complete \cmcone{} face
 $\hat f_{k,t}\colon{}M_k\to S^3_1$ whose
 hyperbolic Gauss map $G$ and Hopf differential $Q_t$ are given 
 as in \eqref{eq:cmc-data}.
\end{proposition}
\begin{remark}
 Consider $\hat f_{k,t}$ as a map into the 
Minkowski space $\R^4_1$, 
 and let
 \[
    \tilde f_{k,t}= \frac{1}{t}\hat f_{k,t}
     \colon{}M_k \longrightarrow S^3_1(t^2)
        =\left\{X\in \R^4_1\,;\,\inner{X}{X}=\frac{1}{t^{2}}
\right\}\subset \R^4_1,
 \]
 where $S^3_1(t^2)$ is de Sitter $3$-space of constant sectional
 curvature $t^2$.
 Then $\tilde f_{k,t}$ is a surface of mean curvature $t$ 
 in $S^3_1(t^2)$.
 Then taking the limit as $t\to 0$ in a similar way as in
 \cite{UY-deform}, $\tilde f_{k,t}$ converges 
 to the maxface $\hat f_{k}$ as in the previous section.
 In this sense, the method provided here is considered as 
 a ``deformation''.
\end{remark}
\subsection{Representation of reflections}
\label{sub:cmc-domain}
 The proof is done by a (refined version of) the reflection method,
 which was introduced in \cite{RUY} for \cmcone{} surfaces in
 the hyperbolic space.
 To do this, we first take the reflections on the universal cover
 $\widetilde M_k$ of $M_k$.

 Let $\mu_j$ ($j=1,2,3$) be the reflections on $M_k$
 as in \eqref{eq:reflections}.
 For a matrix $a=(a_{ij})_{i,j=1,2}\in \SL_2\C$ and 
 a holomorphic function $h$ on a Riemann surface $M$, we set
 \begin{equation}\label{eq:star}
     a \star h:=\frac{a_{11}h+a_{12}}{a_{21}h+a_{22}}.
 \end{equation}
\begin{figure}
\begin{center}
 \includegraphics[width=0.95\textwidth]{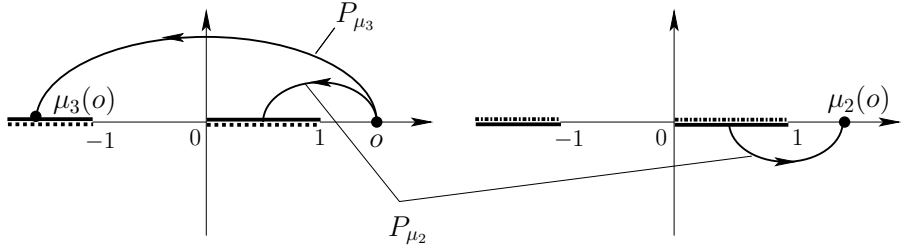}
\end{center}
\caption{%
      Paths $P_{\mu_1}$, $P_{\mu_2}$ and $P_{\mu_3}$ 
    ($P_{\mu_1}$ is the constant path at $o$).}
\label{fig:domain}
\end{figure}
The following lemma is a direct conclusion of \eqref{eq:reflections}:
\begin{lemma}\label{lem:symmetricity}
 The pair $(G,Q_t)$ as in \eqref{eq:cmc-data}
 satisfies 
 \[
   \overline{G\circ{\mu_j}} = \sigma_{j}\star G,\qquad
   \overline{Q_t\circ{\mu_j}} = Q_t \qquad (j=1,2,3),
 \]
 where
 \begin{equation}\label{eq:sigma}
     \sigma_1 =\begin{pmatrix}
		     1 & 0 \\ 0 & 1 
		    \end{pmatrix}(= \id), \quad
     \sigma_2 = \begin{pmatrix}
		   \psi^{-2} & 0 \\ 0 & \psi^{2}
                \end{pmatrix},\quad
     \sigma_3 = \begin{pmatrix}
		   \psi^{-1} & 0 \\ 0 & \psi
		\end{pmatrix},
 \end{equation}
and
 \begin{equation}\label{eq:psi}
   \psi :=e^{ki\lambda/2}=\exp\left(\frac{i\pi  k}{2(k+1)}\right)
   \qquad
   \left(\lambda=\frac{\pi}{k+1}\right).
 \end{equation}
 Namely, $(G,Q_t)$ is $\mu_j$-invariant  $(j=1,2,3)$
 in the sense of \eqref{eq:mu-inv-data} in the appendix.
\end{lemma}

We remark that we do not use $\mu_4$ here, because
$(G,Q)$ is not $\mu_4$-invariant in the sense of \eqref{eq:mu-inv-data} 
as in the appendix. 
In fact, $\overline{Q\circ\mu_4}=-Q$ holds.

We fix the base point $o\in M_k$ as in \eqref{eq:base-point}
and take paths $P_{\mu_j}$ on $M_k$ 
associated to $\mu_j$ as in the appendix
which join $o$ and $\mu_j(o)$
as in Figure~\ref{fig:domain} for each $j=1,2,3$.
(In Figure~\ref{fig:domain}, the left-hand $z$-plane is the 
sheet containing $o$ and the right-hand $z$-plane is the sheet
containing $\mu_2(o)$. These two sheets are connected along two intervals
$(0,1)$ and $(-\infty,-1)$ on each real axis.)

Here $P_{\mu_1}$ is the constant path at $o$.
Then one can take the lift $\tilde\mu_j$ ($j=1,2,3$)
(as orientation-reversing involutions on $\widetilde M_k$)
of $\mu_j:M_k\to M_k$ with
respect to the path $P_{\mu_j}$,
see \eqref{eq:lift-refl} in the appendix.

Then by Proposition~\ref{prop:covering}, we have
\begin{lemma}\label{lem:covering-transf}
 Let $\gamma$ be the loop as in Figure~\ref{fg:loop},
 and let $\kappa_j$ {\rm (}$j=1,2${\rm)} be the
 automorphisms of $M_k$ as in \eqref{eq:symms}.
 Then
 \begin{align*}
    [\gamma]&=\tilde\mu_3\circ\tilde\mu_2\circ\tilde\mu_3\circ\tilde\mu_1,\\
   [\kappa_2\circ\gamma] 
       &= \tilde\mu_3\circ\tilde\mu_1\circ\tilde\mu_3\circ\tilde\mu_2,\\
   [(\kappa_1)^j\circ\gamma] &=
         (\tilde\mu_2\circ\tilde\mu_1)^j\circ [\gamma] \circ 
         (\tilde\mu_1\circ\tilde\mu_2)^j,\\
  [(\kappa_1)^j\circ\kappa_2\circ\gamma] &=
         (\tilde\mu_2\circ\tilde\mu_1)^j\circ [\kappa_2\circ\gamma] \circ 
         (\tilde\mu_1\circ\tilde\mu_2)^j
 \end{align*}
 hold, where $[~]$ denotes the homotopy class and the meaning of the
 above equalities
 is explained in Remark \ref{rmk:covering}.
 On the other hand, let $\tau_0$ and $\tau_{\infty}$ be
 covering transformations corresponding to 
 counterclockwise simple closed loops $\gamma_0$ and
 $\gamma_\infty$ in $M_k$ around $(z,w)=(0,0)$ and
 $(\infty,\infty)$, respectively.
 {\rm (}The projections of $\gamma_0$ and $\gamma_\infty$ to
 the $z$-plane are both $(k+1)$-fold coverings of 
 simple closed loops in $\C$.{\rm)}

 Then
 \[
    \tau_0 = (\tilde\mu_3\circ\tilde\mu_2)^{2(k+1)},\qquad
    \tau_{\infty} = (\tilde\mu_1\circ\tilde\mu_3)^{2(k+1)}
 \]
 hold.
 Namely, $\{\tilde\mu_1,\tilde\mu_2,\tilde\mu_3\}$
 is a  generator of the fundamental group of $M_k$
 in the sense of Definition~\ref{def:generator} in the appendix.
\end{lemma}
For each $t\in\R$ and $b\in \SL_2\C$,
denote by 
\[
   F=F_{t,b}\colon{}\widetilde M_k\longrightarrow\SL_2\C
\]
the unique solution of the differential equation
\begin{equation}\label{eq:ode}
   dF F^{-1} = \Psi = t\Psi_0,\qquad 
    F(o)=b,
\end{equation}
where $\Psi_0$ 
is as in \eqref{eq:cmc-data-alpha}.
Then by Lemma~\ref{lem:symmetricity} and Theorem~\ref{thm:app-refl-mat}
in the appendix, 
there exists $\tilde \rho^{}_{j,t,b}\in\SL_2\C$ such that
\begin{equation}\label{eq:sym-rep}
     \overline{F_{t,b}\circ\tilde\mu_j}=
        \sigma_jF_{t,b}(\tilde \rho^{}_{j,t,b})^{-1}
        \qquad (j=1,2,3),
\end{equation}
where $\sigma_j$ ($j=1,2,3$) are the matrices 
given in \eqref{eq:sigma}. 
By Proposition \ref{prop:monodromy-refl},
the group generated by $\{\tilde \rho_{j,t,b}\}_{j=1,2,3}$
contains the subgroup $\rho_F(\pi_1(M_k))$
given in \eqref{eq:monodromy-repr}.

Then we have the following by Lemma~\ref{lem:covering-transf} and 
Theorem~\ref{thm:su11-monodromy}, which is the key to our construction:
\begin{proposition}\label{prop:main}
 If $\tilde \rho^{}_{j,t,b}\in\SU_{1,1}$ for $j=1,2,3$,
 then the \cmcone{} face 
 \[
     \hat f_{t,b}:= (F_{t,b}) e_3 (F_{t,b})^*
 \]
 corresponding to $F_{t,b}$ as in \eqref{eq:cmc-face}
 is well-defined on $M_k$.
\end{proposition}

Here, we summarize properties of the matrices $\tilde\rho^{}_{j,t,b}$:
\begin{proposition}\label{prop:rho-prop}
 Suppose $a$, $b\in\SL_2\C$.
 Then the following hold{\rm :}
 \begin{enumerate}
  \item\label{item:rho-1}
       $\tilde\rho^{}_{j,t,ba}=\bar a^{-1}(\tilde \rho^{}_{j,t,b}) a$.
  \item\label{item:rho-2}
       $\tilde\rho^{}_{j,0,b}=\bar{b}^{-1}(\sigma_j) b$.
       In particular, $\tilde \rho^{}_{j,0,\id}=\sigma_j$.
 \end{enumerate}
\end{proposition}
\begin{proof}
 Since $F_{t,ba}= (F_{t,b})a$,
 \ref{item:rho-1} holds.
 When  $t=0$, $F_{0,b}=b$.
 Then 
 \[
   \bar{b} = \overline {F_{0,b}}=
   \overline {F_{0,b}\circ\tilde\mu_j}=
   \sigma_j F_{0,b}(\rho^{}_j)^{-1}
         =\sigma_j b(\rho^{}_j)^{-1},
 \]
 where $\rho^{}_j=\tilde\rho^{}_{j,0,b}$, which implies \ref{item:rho-2}.
\end{proof}
\subsection{Existence of $\hat f_k$}
\label{sub:cmc-f-k}
The existence part of Proposition~\ref{prop:deform-2}
is a straightforward conclusion of the following proposition,
because of Proposition~\ref{prop:main}:
\begin{proposition}\label{prop:monodromy}
 For a sufficiently small positive number $\varepsilon$,
 there exists a real analytic family $\iota(t)$ of  matrices in $\SL_2\C$
 such that $\tilde\rho^{}_{j,t,\iota(t)}\in\SU_{1,1}$ $(j=1,2,3)$ holds
 if $0<|t|<\varepsilon$.
\end{proposition}
The proof is divided into three steps
(Claims~\ref{claim:1}--\ref{claim:3}):
\begin{claim}\label{claim:1}
 For $t\in \R$ and for any real matrix $b\in\SL_2\R$, 
 it holds that
 $\tilde\rho^{}_{1,t,b}=\id$.
\end{claim}
 In fact, 
 substituting $o$ into \eqref{eq:sym-rep} for $j=1$, we have 
 Claim~\ref{claim:1}.%
\begin{claim}\label{claim:2}
 For sufficiently small $\varepsilon>0$, there exists a real analytic
 family 
 $\{\iota(t)\,;\,|t|<\varepsilon\}$
 of matrices in $\SL_2\R$ such that $\iota(0)=\id$ and
 $\tilde\rho^{}_{2,t,\iota(t)}=\sigma_2$.
\end{claim}
\begin{proof}
 Since $(\mu_2\circ\mu_1)^{k+1}=\identity_{M_k}$ holds, where
 $\identity_{M_k}$ is the identity map on $M_k$,
 then by Proposition~\ref{prop:covering} in the appendix,
 $(\tilde\mu_2\circ\tilde\mu_1)^{k+1}$ is a covering transformation
 on $\widetilde M_k$.
 In fact, such a covering transformation corresponds to the 
 counterclockwise simple closed loop on $M_k$
 surrounding $(z,w)=(1,0)$, i.e., 
 $(\tilde\mu_2\circ\tilde\mu_1)^{k+1}=\identity_{\widetilde M_k}$.
 Then by Proposition~\ref{prop:monodromy-refl}, 
 \[
    e_0=\rho^{}_F(\identity_{\widetilde M_k})
      = \bigl(\overline{\sigma_2}\sigma_1\bigr)^{k+1}
         \bigl(\overline{\rho^{}_2}\rho^{}_1\bigr)^{k+1}
      = (-1)^k \bigl(\overline{\rho^{}_2}\rho^{}_1\bigr)^{k+1},
 \]
 where $\rho^{}_j = \tilde\rho^{}_{j,t,\id}$.
 Hence we have
 \begin{equation}\label{eq:rho-2-eigen}
 (\rho^{}_2)^{k+1} = (-1)^k\id.
 \end{equation}
 Since  
 $\tilde\rho^{}_{2,t,\id}$ tends to
 $\tilde\rho^{}_{2,0,\id}=\sigma_2$
 as $t\to 0$ (see \ref{item:rho-2} of Proposition \ref{prop:rho-prop}),
 the equality 
 \eqref{eq:rho-2-eigen} implies that
 the eigenvalues 
 of $\tilde\rho^{}_{2,t,\id}$ are $\{\psi^{\pm 2}\}$,
 where $\psi$ is given in \eqref{eq:psi}.
 Hence
 \[
     \trace \tilde\rho^{}_{2,t,\id} = 2 \cos (k\lambda)
    \qquad \left(\lambda=\frac{\pi}{k+1}\right).
 \]
 Then by \eqref{eq:rho-form}, one can write
 \begin{multline*}
     \tilde\rho^{}_{2,t,\id}=
     \begin{pmatrix}
        \cos (k\lambda) - i u(t) & i s_1(t) \\
         i s_2(t) & \cos (k\lambda)+ i u(t)
     \end{pmatrix}\\
       \bigl((\cos k\lambda)^2 + u(t) ^2 + s_1(t)s_2(t) =1)\bigr),
 \end{multline*}
 where $u=u(t)$, $s_j=s_j(t)$ ($j=1,2$) are real analytic functions
 in $t$. 
 Since $\tilde\rho_{2,0,\id}=\sigma_2$, we have that
 \begin{equation}\label{eq:ust-init}
     u(0) = \sin (k\lambda), \qquad 
     s_1(0)=s_2(0)=0.
 \end{equation}
 Let 
 \[
   \iota(t) :=
    \frac{1}{\sqrt{2\bigl((\sin k\lambda)^2 + u(t)\sin(k\lambda)\bigr)}}
    \begin{pmatrix}
       u(t) + \sin(k\lambda) & s_1(t) \\
       -s_2(t) & u(t)+\sin(k\lambda)
    \end{pmatrix}.
 \]
 By \eqref{eq:ust-init}, $(\sin k\lambda)^2+u(t)\sin(k\lambda)>0$
 for sufficient  small $t$, and hence $\iota(t)$ is a real matrix.
 It can be easily checked that $\iota(t)^{-1}\tilde \rho_{2,t,\id} \iota(t)=
 \sigma_2$. Then  \ref{item:rho-1}
 of Proposition \ref{prop:rho-prop} yields the assertion. 
\end{proof}
\begin{claim}\label{claim:3}
 For $\iota(t)$ in Claim~\ref{claim:2}, 
 it holds that
 \begin{equation}\label{eq:rho-three}
    \tilde\rho^{}_{3,t,\iota(t)}=
     \begin{pmatrix}
       q(t) & i r_1(t) \\
       i r_2(t) & \overline{q(t)}
     \end{pmatrix}
     \qquad (
      |q(t)|^2 + r_1(t)r_2(t)=1
     ),
 \end{equation}
 where $q(t)$ is a complex-valued real-analytic function
 and $r_j(t)$ $(j=1,2)$ are real-valued real-analytic
 functions in $t$ 
 such that
 \begin{equation}\label{eq:qbc-init}
  q(0) = \psi^{-1},\qquad r_1(0)=r_2(0)=0.
 \end{equation}
 Moreover, for $t\neq 0$ with sufficiently small absolute value,
 \begin{equation}\label{eq:negative}
  r_1(t)r_2(t)<0\qquad \mbox{or equivalently},\quad  |q(t)|>1
 \end{equation} 
 holds.
\end{claim}
 The inequality \eqref{eq:negative} corresponds to 
 the argument in \cite[Lemma 6.10]{RUY}.
 If \eqref{eq:negative} holds, we can
 prove the existence of the desired deformation $F_{t,\iota_1(t)}$
 by modifying $\iota(t)$ by $\iota_1(t)$,
 as we shall see later.
 If $r_1(0)r_2(0)$ had been negative, Claim \ref{claim:3} 
 would be obvious.
 However, in our case, \eqref{eq:qbc-init} implies
 $r_1(0)r_2(0)=0$,
 although all of the examples in \cite{RUY} 
 satisfy $r_1(0)r_2(0)\ne 0$.
 We show \eqref{eq:negative} by examining the derivative of the
 monodromy matrix. 
 Set
\begin{equation}\label{eq:newF}
    F_t:=F_{t,\iota(t)} ,
\end{equation}
\begin{equation}\label{eq;rho}
    \rho^{}_j=\tilde\rho^{}_{j,t,\iota(t)} \qquad (j=1,2,3).
\end{equation}
Then by Claims \ref{claim:1} and \ref{claim:2}, we have
\begin{equation}\label{eq:rho-12}
    \rho^{}_1=e_0,\qquad \rho^{}_2 = \sigma_2.
\end{equation}
Moreover, by Theorem~\ref{thm:app-refl-mat}, $\rho_3$ is written
as in \eqref{eq:rho-three},
and by \ref{item:rho-2} in Proposition~\ref{prop:rho-prop}
and the fact that $\iota(0)=\id$, we have $\rho_3\to\sigma_3$ as
$t\to 0$.
Hence $q$, $r_1$ and $r_2$ in \eqref{eq:rho-three} satisfy
\eqref{eq:qbc-init}.

 Let $\tau_0$ and $\tau_{\infty}$ be 
 the covering transformations on $\widetilde M_k$
 given in Lemma \ref{lem:covering-transf}.
 Since the initial value $\iota(t)=F_t(o)$ is real analytic in $t$,
 so are $\rho^{}_{F_t}(\tau_0)$ and $\rho^{}_{F_t}(\tau_\infty)$. 

\begin{lemma}\label{lem:end-monodromy}
  $F_t:=F_{t,\iota(t)}$ satisfies
  \begin{equation}\label{eq:rho-trace}
    \trace \rho^{}_{F_t}({\tau_0}) =
     (-1)^k 2 \cos(\pi\nu_0),\quad
    \trace \rho^{}_{F_t}({\tau_\infty}) =
     (-1)^k 2 \cos(\pi\nu_\infty), 
  \end{equation}
 where
 \begin{align}\label{eq:g-order}
   \nu_0 =\nu_0(t)&:=k\sqrt{1+4t\frac{k+1}{k}}, \\
   \nu_\infty=\nu_{\infty}(t)&:=k\sqrt{1-4t\frac{k+1}{k}}.
 \label{eq:g-order2}
 \end{align}
\end{lemma}

\begin{proof}
 Let $g$ be the secondary Gauss map of $F_t$.
 Setting $z=\zeta^{k+1}$, we can take a complex coordinate $\zeta$
 around $(0,0)$.
 Since $G$ is a meromorphic function and $Q_t$ has a pole of 
 order $2$ at $(0,0)$, 
 \eqref{eq:schwarz} implies that the Schwarzian derivative
 $S(g)$ has a pole of order $2$ at $(0,0)$.
 Hence there exist $a\in\SL_2\C$ and a constant $\nu_0$ 
 such that
 \begin{equation}\label{eq:nu}
     a\star g = \zeta^{\nu_0}\bigl(1+o(1)\bigr),
 \end{equation}
 where $o(\cdot)$ denotes a higher order term.
 Here, 
 since $G$ has a pole of order $k$ at $(0,0)$, 
 and
 \[
    Q_t=\frac{tk}{k+1}\frac{z^2+1}{z^2(z^2-1)}dz^2
           =-\frac{tk(k+1)}{\zeta^2}\bigl(1+o(1)\bigr)d\zeta^2
 \]
 \eqref{eq:schwarz} implies that (see \cite[page 233]{UY5})
 \[
    S(g)=S(G)+2Q_t 
     =\frac{1}{2\zeta^2}
      \biggl((1-k^2)-4tk(k+1)+o(1)\biggr)d\zeta^2.
 \]
 Similarly,
 \eqref{eq:nu} implies that
 \[
    S(g) = \frac{1}{2\zeta^2}(1-\nu_0^2)d\zeta^2.
 \]
 Thus, $\nu_0$ coincides with \eqref{eq:g-order}.
 Note that $\nu_0$ is a real number 
 for $t$ with sufficiently small absolute value.

 Comparing the relation $g\circ\tau_0=\rho^{}_{F_t}({\tau_0})\star g$ 
 with \eqref{eq:nu},
 we can conclude that the eigenvalues of 
 $\rho^{}_{F_t}({\tau_0})$ are equal to those of
 the monodromy matrix of the function $\zeta\mapsto \zeta^{\nu_0}$
 up to sign.
 Then  the eigenvalues of $\rho^{}_{F_t}({\tau_0})$ are
 \begin{equation}\label{eq:eigen-rho-tau-0}
   \bigl\{e^{i\pi \nu_0},e^{-i\pi \nu_0} \bigr\}
   \quad\text{or}\quad
   \bigl\{-e^{i\pi \nu_0},-e^{-i\pi \nu_0} \bigr\}.
    \end{equation}
 Similarly,  
 setting $1/z=\zeta^{k+1}$, we take
 a complex coordinate $\zeta$ around $(\infty,\infty)$.
 Then we have \eqref{eq:g-order2}, and
 the eigenvalues of $\rho^{}_{F_t}({\tau_{\infty}})$ are
 \begin{equation}\label{eq:eigen-rho-tau-infty}
   \bigl\{e^{i\pi \nu_\infty},e^{-i\pi \nu_\infty} \bigr\}
   \quad\text{or}\quad
   \bigl\{-e^{i\pi \nu_\infty},-e^{-i\pi \nu_\infty} \bigr\}.
    \end{equation}
 On the other hand, by Lemma~\ref{lem:covering-transf} and 
 Proposition~\ref{prop:monodromy-refl} in the appendix,
 we have
 \begin{align}\label{eq:rho-tau}
    \rho^{}_{F_t}(\tau_0) &=
    \left[\bigl(\overline{\sigma_3}\sigma_2\bigr)^{2(k+1)}\right]
    \bigl(\overline{\rho^{}_3}\rho^{}_2\bigr)^{2(k+1)}
    =(-1)^k  \bigl(\overline{\rho^{}_3}\rho^{}_2\bigr)^{2(k+1)},\\
  \label{eq:rho-tau2}
    \rho^{}_{F_t}(\tau_{\infty}) &= (-1)^k
   \bigl(\overline{\rho^{}_1}\rho^{}_3\bigr)^{2(k+1)}.
 \end{align}
 Here, by \ref{item:rho-2} in Proposition~\ref{prop:rho-prop} and 
 Claim~\ref{claim:2},
 $\rho^{}_j\to \sigma_j$ ($t\to 0$) holds
 for $j=1,2,3$.
Since $\sigma_j$ ($j=1,2,3$) are given explicitly in
\eqref{eq:sigma}, we have that, as $t\to 0$,
 \begin{equation}\label{eq:rho-tau-limit}
   \rho^{}_{F_t}({\tau_0})\to(-1)^k
   \bigl(\sigma_2\overline{\sigma_3}\bigr)^{2(k+1)}
   =(-1)^{2k}\id=\id,\quad\text{and}\quad
   \rho^{}_{F_t}({\tau_{\infty}})\to\id.
 \end{equation}
 Then the eigenvalues of $\rho^{}_{F_t}({\tau_0})$ 
 and $\rho^{}_{F_t}(\tau_{\infty})$
 tend to $1$ as $t\to 0$.
 Hence by real analyticity,
 the right-hand possibilities for eigenvalues
 in \eqref{eq:eigen-rho-tau-0} and \eqref{eq:eigen-rho-tau-infty}
 never occur, which implies the conclusion.
\end{proof} 

\begin{lemma}\label{lem:rho-der}
 $F_t:=F_{t,\iota(t)}$ satisfies
 \begin{align*}
    \left.\frac{d}{dt}\right|_{t=0}
      \rho^{}_{F_t}({\tau_0})^{-1} &= 
     2(k+1)\pi i
     \begin{pmatrix} 1 & \hphantom{-}0 \\
                     0 & -1 
     \end{pmatrix},\,\,\, \\
    \left.\frac{d}{dt}\right|_{t=0}
      \rho^{}_{F_t}({\tau_\infty})^{-1} &= 
     -2(k+1)\pi i
     \begin{pmatrix} 1 & \hphantom{-}0 \\
                     0 & -1 
     \end{pmatrix}.
 \end{align*}
\end{lemma}
\begin{proof}
 As in the proof of the previous lemma, 
 we can take
 the complex coordinate $\zeta$ around $(0,0)$
 such that $z=\zeta^{k+1}$.
 Then $G$ and $Q$ are expressed in terms of $\zeta$,
 and by \eqref{eq:rho-deform} and \eqref{eq:cmc-data-alpha}, we have
 \begin{align*}
    \left.\frac{d}{dt}\right|_{t=0}\rho^{}_{F_t}({\tau_0})^{-1}
      &=\oint_{\gamma_0} \frac{1}{c_k}
            \begin{pmatrix}
	       G & - G^2 \\
	       1 & -G\hphantom{^2}
	    \end{pmatrix}\frac{Q}{dG} \\
    &=2\pi i \Res_{\zeta=0} \frac{1}{c_k}
            \begin{pmatrix}
	       G & - G^2 \\
	       1 & -G\hphantom{^2}
	    \end{pmatrix}\frac{Q}{dG}
        =2\pi i \begin{pmatrix} k+1 & \hphantom{-}0\\
	                     0 & -(k+1) \end{pmatrix},
 \end{align*}
 where $\gamma_0$ is the loop surrounding $(0,0)$ given in Lemma 
\ref{lem:covering-transf}.
 Hence we have the conclusion for $\rho^{}_{F_t}({\tau_0})$.
 The derivative of $\rho^{}_{F_t}(\tau_\infty)$ is obtained in a similar way.
\end{proof}
Now, we shall prove Claim~\ref{claim:3}:
\begin{proof}[Proof of Claim~\ref{claim:3}]
 We have already shown \eqref{eq:rho-three} and \eqref{eq:qbc-init}.
 Then it is sufficient to show that $|q(t)|>1$ for $t\neq 0$
 with sufficiently small $|t|$.
 We set $a_0(t)=\overline{\rho^{}_3}\rho^{}_2$, then 
 \eqref{eq:rho-tau} can be rewritten as
 \begin{equation}\label{eq:a0}
    \rho^{}_{F_t}(\tau_0)=(-1)^k(a_0(t))^{2(k+1)}.
 \end{equation}
 Set 
 \[
   A_0(t):=\frac{1}{2}\trace(a_0(t)).
 \]
 By Claim \ref{claim:2} and \eqref{eq:rho-three}, we have that
 \begin{equation}\label{eq:A0}
    A_0(t)= 
      \frac{1}{2}(\psi^{-2} \overline{ q(t)} + \psi^2  q(t)).
 \end{equation}
 Letting $t\to 0$,  we have
 $A_0(0)=\cos\{k\pi/(2(k+1))\}$ by \eqref{eq:qbc-init}.
 Then there exists a real analytic function $\theta_0=\theta_0(t)$ such
 that
 \begin{equation}\label{eq:theta-0}
   \cos\theta_0(t) = A_0(t)=\frac{1}{2}\trace(a_0(t)),\qquad
     \theta_0(0) =\frac{k\pi}{2(k+1)}.
 \end{equation}
 Then the Cayley-Hamilton identity yields that
 $a_0(t)^2=2 a_0(t)\cos\theta_0(t)-\id$.
 Then by induction, one can prove 
 the identity (purely algebraically)
 \[
    (a_0)^m
      =\frac{\sin(m\theta_0)}{\sin \theta_0}a_0
        -\frac{\sin((m-1)\theta_0)}{\sin \theta_0}\id\qquad
       (m=1,2,3,\dots).
 \]
 By \eqref{eq:a0}, we have
 \[
     \rho^{}_{F_t}({\tau_0}) = (-1)^k
     \left(\frac{\sin
     (2k+2)\theta_0}{\sin\theta_0}a_0
     -\frac{\sin (2k+1)\theta_0}{\sin\theta_0}\id\right).
 \]
 Taking the trace of this, Lemma~\ref{lem:end-monodromy} yields
 \[
     \cos \pi\nu_0= 
      \frac{\sin (2k+2)\theta_0}{\sin\theta_0}\cos\theta_0
                  -\frac{\sin (2k+1)\theta_0}{\sin\theta_0}\\
            = \cos (2k+2)\theta_0.
 \]
 Hence 
 \[
   \pi \nu_0(t)   \equiv
   \pm 2(k+1)\theta_0(t)\qquad\pmod{ 2\pi}.
 \]
 Comparing both sides of this equation at $t=0$,  \eqref{eq:theta-0}
 implies
 \begin{equation}\label{eq:theta-trace}
    \pi \nu_0(t) 
       = 2(k+1) \theta_0(t)\qquad \text{or} \qquad 
       \pi \nu_0(t)=2(k+1)(\pi-\theta_0(t)),
 \end{equation}
 by real analyticity.

 On the other hand, 
 by Lemma~\ref{lem:rho-der}, it holds that
 \allowdisplaybreaks{
 \begin{align*}
    -2&(k+1)\pi i\begin{pmatrix}
	   1 & \hphantom{-}0 \\
	   0 & -1
	  \end{pmatrix}
   = -\left.\frac{d \rho^{}_{F_t}({\tau_0})^{-1}}{dt}\right|_{t=0}
   = \left.\frac{d\rho^{}_{F_t}({\tau_0})}{dt}\right|_{t=0}
   \\
   &=(-1)^k
     \left.\frac{d}{dt}\right|_{t=0}
     \left(\frac{\sin (2k+2)\theta_0}{\sin\theta_0}a_0-
           \frac{\sin (2k+1)\theta_0}{\sin\theta_0}\id\right),
 \end{align*}
 where we used the fact that $\rho^{}_{F_0}(\tau_0)=\id$.
 Since 
 \[
   \theta_0(0)=\frac{k\pi}{2k+2},\quad \sin((2k+2)\theta_0(0))=0,\quad
       \cos((2k+2)\theta_0(0))=(-1)^k,
 \]
 by setting $\dot \theta_0(0)=d\theta_0/dt|_{t=0}$, we have that 
{\small
 \begin{align*}
    -2&(k+1)\pi i\begin{pmatrix}
	   1 & \hphantom{-}0 \\
	   0 & -1
	  \end{pmatrix} \\
  &=(-1)^k
     \left.\frac{d}{dt}\right|_{t=0}
     \left(\frac{\sin (2k+2)\theta_0}{\sin\theta_0}a_0-
           \frac{\sin \bigl((2k+2)\theta_0-\theta_0\bigr)
                 }{\sin\theta_0}\id\right)\\
  &=(-1)^k
     \left.\frac{d}{dt}\right|_{t=0}
     \left(\frac{\sin (2k+2)\theta_0}{\sin\theta_0}a_0-
           \left(\frac{\sin (2k+2)\theta_0}{\sin\theta_0}\cos\theta_0
                 -\cos (2k+2)\theta_0\right)\id\right)\\
   & = \frac{2k+2}{\sin \theta_0(0)}
          \left(a_0(0)-\cos\frac{k\pi}{2k+2}\id\right)
        \dot \theta_0(0)\\
   &= \frac{2k+2}{\sin \frac{k\pi}{2k+2}}
      \left(\overline{\sigma_3}\sigma_2-\cos\frac{k\pi}{2k+2}\id
      \right)\dot \theta_0(0)
   =  -2(k+1)i
        \begin{pmatrix} 1 & \hphantom{-}0 \\ 0 & -1 \end{pmatrix}
        \dot \theta_0(0).
 \end{align*}}
 }
 Thus, we have 
 \begin{equation}\label{eq:theta-der}
    \dot \theta_0(0) = \id.
 \end{equation}
 Then by \eqref{eq:theta-trace}, we have
 \begin{equation}\label{eq:theta-val}
    \theta_0(t) = \frac{\pi\nu_0(t)}{2k+2}. 
 \end{equation}

 Similarly, if we set $a_\infty=\overline{\rho^{}_1}\rho^{}_3$,
 then \eqref{eq:rho-tau2} can be rewritten as
 \begin{equation}\label{eq:a-infty}
    \rho^{}_{F_t}(\tau_\infty)=(-1)^k(a_\infty(t))^{2(k+1)}.
 \end{equation}
 Set 
 \[
    A_\infty(t):=\frac{1}{2}\trace(a_\infty(t)).
 \]
 By Claim \ref{claim:2} and \eqref{eq:rho-three}, we have that
 \begin{equation}\label{eq:Ainfty}
    A_\infty(t)= 
        \frac{1}{2}(\overline{q(t)} + q(t)).
 \end{equation}
 Letting $t\to 0$,  we have
 $A_\infty(0)=\cos\{k\pi/(2k+2)\}$ by \eqref{eq:qbc-init}.
 Then there exists a real analytic 
 function $\theta_\infty=\theta_\infty(t)$ such that
 \begin{equation}\label{eq:theta-infty}
     \cos\theta_\infty(t) = A_\infty(t)=\frac{1}{2}\trace(a_\infty(t)),\qquad
     \theta_\infty(0) =\frac{k\pi}{2k+2}.
 \end{equation}
 Like as in the computation of $\theta_0(t)$, 
 we have
 \begin{equation}\label{eq:varphi}
  \dot\theta_\infty(0)=-\pi,\qquad
   \theta_\infty(t) = \frac{\pi\nu_\infty(t)}{2k+2}.
 \end{equation}
 By \eqref{eq:A0}, \eqref{eq:theta-0}, \eqref{eq:Ainfty}
 and  \eqref{eq:theta-infty}, we have that
 \[
     \psi^{2}q+\psi^{-2}\bar q = 2 \cos\theta_0,\qquad
        q +\bar q =2\cos\theta_\infty.
 \]
 Hence 
 \[
     q(t)=\frac{-i}{\sin\frac{k\pi}{k+1}}\biggl(
       \cos\theta_0(t)-\psi^{-2}\cos \theta_\infty(t)
       \biggr)
 \]
 and
 \[
     q \bar q =
       \frac{1}{\sin^2\frac{k\pi}{k+1}}
        \left(
           (\cos\theta_0)^2 + (\cos \theta_\infty)^2 - 2
             \left(\cos\frac{k\pi}{k+1}\right)
                    \cos\theta_0\cos\theta_\infty
             \right).       
 \]
 Differentiating this twice
 using \eqref{eq:theta-der}, \eqref{eq:theta-val} and
 \eqref{eq:varphi},
 we have
 \[
   \left.q\bar q\right|_{t=0}=1,\quad
   \left.\frac{d}{dt}\right|_{t= 0}
       q\bar q = 0 ,\quad
   \left.\frac{d^2}{dt^2}\right|_{t= 0}
      q\bar q = \frac{4(k+1)\pi}{k}\tan\frac{\pi k}{2k+2}>0.
 \]
 Thus, $(|q|^2=)q\bar q>1$ 
 for $t$ with sufficiently small absolute value,
 and by \eqref{eq:rho-three}, 
 $r_1r_2<0$ holds.
\end{proof}
\begin{proof}[Proof of Proposition~\ref{prop:monodromy}]
 Take $\iota(t)$ as in Claim~\ref{claim:2}.  
 We set
 \[
    \iota_1(t) = \iota(t) \begin{pmatrix}
	     s(t) & 0\\
	     0 & 1/s(t)
	    \end{pmatrix},
         \qquad s(t) = \sqrt[4]{-r_1(t)/r_2(t)}.
 \]
 Then by \ref{item:rho-1} of Proposition~\ref{prop:rho-prop} and 
 Claims~\ref{claim:1} and \ref{claim:2}, 
 we have
 $\tilde\rho^{}_{1,t,\iota_1(t)}=\id\in\SU_{1,1}$,
 $\tilde\rho^{}_{2,t,\iota_1(t)}=\sigma_2\in\SU_{1,1}$ and
 \[
    \tilde\rho^{}_{3,t,\iota_1(t)}=
    \begin{pmatrix}
     q(t) & \varepsilon i\sqrt{-r_1(t)r_2(t)} \\[6pt]
     -\varepsilon i\sqrt{-r_1(t)r_2(t)} & \overline{q(t)}
    \end{pmatrix}\in\SU_{1,1},
 \]
 where $\varepsilon=1$ (resp.\ $-1$) if $r_1(t)>0$ (resp.\ $r_1(t)<0$).
 By replacing $\iota(t)$ by $\iota_1(t)$,
 we have the conclusion. 
\end{proof}

Thus, we obtain a one parameter family of \cmcone{} faces
$\{\hat f_{k,t}\}$ defined on $M_k$.
When $k$ is even, it induces $f_k\colon{}M'_k\to S^3_1$, where
$M'_k$ is the Riemann surface of genus $k/2$ as in
Section~\ref{sub:double}, because
$\{\tilde\mu_1,\tilde\mu_2,\tilde\mu_3\}$
generates the fundamental group of $M'_k$.

\subsection{Completeness and embeddedness}
Now, we have weakly complete \cmcone{} faces $f_{k,t}$ 
for positive integers  $k>0$ and 
for $t\neq 0$ with sufficiently small absolute value.
In this subsection, we shall prove completeness of $f_{k,t}$
and embeddedness of $f_{1,t}$ and $f_{2,t}$ for sufficiently
small $t$, which shows the second part of Theorem~\ref{thm:main}.
\begin{proposition}[Completeness]\label{prop:completeness}
 For each positive integer $k$ and 
 for $t\neq 0$ with sufficiently small absolute value,
 the \cmcone{} face $f_{k,t}\colon{}M_k\to S^3_1$ is complete, 
 and each end is a regular elliptic end 
 in the sense of \cite{FRUYY}.
\end{proposition}
\begin{proof}
 Since $G$ is meromorphic at the ends, they are regular ends.
 Moreover, by \eqref{eq:nu} and \eqref{eq:g-order},
 the end $(0,0)$ is $g$-regular non-integral elliptic
 in the sense of \cite[Definition 3.3]{FRUYY}
 because $\nu_0\not\in\pi\Z$.
 Then by Lemma E1 in \cite{FRUYY},
 the singular set does not
 accumulate at the end, and hence the end $(0,0)$  is complete.
 Similarly, the end $(\infty,\infty)$ is also complete.
\end{proof}
To show embeddedness, we shall look at the asymptotic behavior 
of the two ends of $f_{k,t}$.
The ideal boundary of $S^3_1$ consists of two connected components:
\[
   \partial_+S^3_1 = LC_+/\R_+\qquad \text{and}\qquad
   \partial_-S^3_1 = LC_-/\R_+,
\]
where $LC_+$ (resp.\ $LC_-$) denotes the positive (resp.\ negative)
light cone in $\R^4_1$:
\[
    LC_{\pm} = \{(v_0,v_1,v_2,v_3)\in\R^4_1\,; \,\pm v_0>0\},
\]
see \cite[Section 4]{FRUYY}.
\begin{lemma}\label{lem:asymp}
 For $t\neq 0$ with sufficiently small absolute value,
 the two ends of $f_{k,t}$ are asymptotic to the same point 
 of the same connected component of the ideal boundary.
\end{lemma}
\begin{proof}
 Noticing $Q_t\to 0$ as $t\to 0$, \eqref{eq:schwarz}
 for $(G,Q_t)$ implies that $S(g)\to S(G)$ as $t\to 0$.
 Since $G(0,0)=G(\infty,\infty)=\infty$
 (see Table~\ref{tb:g-eta-q} in Section~\ref{sec:main}),
 this implies that  $|g|-1$ has the same sign at 
 $(0,0)$ and $(\infty,\infty)$.
 In particular, one can choose $g$ such that
 $|g|>1$ on neighborhoods of $(0,0)$ and $(\infty,\infty)$.
 Then by Proposition 4.2 in \cite{FRUYY},
 $f_{k,t}$ converges to $\partial_-S^3_1$ 
 at the two ends.
 Moreover, since $G(0,0)=G(\infty,\infty)$,
 the proof of \cite[Proposition 4.2]{FRUYY} implies that
 both of the  ends converge to the same point of $\partial_-S^3_1$.
\end{proof}
\begin{proposition}
 For $t\neq 0$ with sufficiently small absolute value,
 the \cmcone{} face $f_{k,t}$ is embedded if and only if $k=1$
 or $2$.
\end{proposition}
\begin{proof}
 By Table~\ref{tb:g-eta-q} in Section~\ref{sec:main},
 $G$ has poles of order $k$ at $(0,0)$ and  $(\infty,\infty)$ in
 $\overline{M}_k$.
 Then there exists a complex coordinate $\zeta$ of $\overline{M}_k$
 around $(0,0)$ such that $G=\zeta^{-k}$.
 On the other hand,
 by \eqref{eq:nu}, $g$ is represented as
 \[
    g = \zeta^{-\nu_0}\bigl(\alpha + o(1)\bigr)\qquad 
    \alpha\in\R\setminus\{0\},
 \]
 where $\nu_0>0$ is as in \eqref{eq:g-order},
 since we can replace the secondary Gauss map $g$ by $1/g$.
 Then by Small's formula \cite[Equation (1.10)]{FRUYY} 
 (see also \cite{KUY}),
 the holomorphic lift $F$ is expressed as
 \[
     F =
         \frac{1}{2\sqrt{k\nu_0}}
         \begin{pmatrix}
	  \zeta^{\frac{-k+\nu_0}{2}}(k+\nu_0) 
	  \left(\dfrac{-1}{\sqrt{\alpha}}+o(1)\right)&
	  \zeta^{\frac{-k-\nu_0}{2}}(k-\nu_0) 
	  \left(\dfrac{1}{\sqrt{\alpha}}+o(1)\right)\\[8pt]
	  \zeta^{\frac{\hphantom{-}k+\nu_0}{2}}
	     (k-\nu_0)\left(\sqrt{\alpha}+o(1)\right)&
	  \zeta^{\frac{\hphantom{-}k-\nu_0}{2}}
	     (k+\nu_0)\left(-\sqrt{\alpha}+o(1)\right)
	 \end{pmatrix},
 \]
 where $o(\cdot)$ denotes a higher order term.
 Hence the coordinate functions of 
 $\hat f_{k,t} = Fe_3 F^*$
 as in \eqref{eq:herm-mink} are expressed as
 \begin{align*}
    x_0 &= -r^{-(k+\nu_0)}\bigl(\alpha_0+o(1)\bigr), \\
    x_3 &= -r^{-(k+\nu_0)}\bigl(\alpha_0+o(1)\bigr), \quad
    x_1 + i x_2 = r^{-\nu_0}e^{ik\theta}\alpha_1 \bigl(1+o(1)\bigr),
 \end{align*}
 where $\alpha_0$, $\alpha_1$ 
 are positive real numbers and $\zeta=re^{i\theta}$.
 Hence $x_0\to -\infty$ as $r\to 0$ and 
 \[
    x_1 + i x_2 = C_1 e^{ik\theta}x_0^{\frac{\nu_0}{k+\nu_0}}
                  \bigl(1+o(1)\bigr),\qquad
    x_3 = x_0 \bigl(1+o(1)\bigr),
 \]
  where $C_1$ is a non-zero constant.
 Similarly, the end $(\infty,\infty)$ is expressed as 
 \[
    x_1 + i x_2 = C_2 e^{ik\theta}x_0^{\frac{\nu_{\infty}}{k+\nu_{\infty}}}
                  \bigl(1+o(1)\bigr),\qquad
    x_3 = x_0 \bigl(1+o(1)\bigr),
 \]
 where $\nu_{\infty}$ is as in \eqref {eq:g-order2} and 
 $C_2$ is a non-zero constant.
 Since $\nu_0\neq \nu_{\infty}$, 
 these two ends do not intersect each other
 in sufficiently small neighborhood of the ends.
 Moreover, each end has no self intersection if and only if $k=1$, or 
 $k=2$ (notice that $M_2$ is a double cover of $M'_2$).
\end{proof}

\appendix
\section{Reflections and fundamental groups}
\label{sec:ref}
In this appendix, we review the properties of the fundamental
group and reflections on Riemann surfaces,
which were used in Section~\ref{sec:deform}.
\subsection{Properties of reflections}
\label{sub:reflections}
Let $M$ be a connected Riemann surface and fix a base point $o\in M$.
We denote  the set of continuous paths starting at $o$ by
\[
    C_o(M):=\{\gamma\colon{}[0,1]\to M;\, 
          \text{$\gamma$ is continuous and $\gamma(0)=o$}\}.
\]
We denote by $[\gamma]$ the homotopy class containing $\gamma\in C_o(M)$.
Then the universal covering space $\widetilde M$ can be canonically
identified with the quotient space $\{[\gamma]\,;\, \gamma\in C_o(M)\}$,
and the covering projection is given by
\[
     \pi:\widetilde M\ni [\gamma]\longmapsto \gamma(1)\in M.
\]

When two paths $\gamma_j\colon{}[0,1]\to M$ $(j=1,2)$ satisfy
$\gamma_1(1)=\gamma_2(0)$,
we denote by $\gamma_1*\gamma_2$ the path obtained by joining 
$\gamma_1$ and $\gamma_2$ as follows:
\begin{equation}\label{eq:join}
   \gamma_1*\gamma_2(u):=
    \begin{cases} 
     \gamma_1(2u)   & \mbox{if $u\in [0, 1/2]$}, \\
     \gamma_2(2u-1) & \mbox{if $u\in [1/2, 1]$}. 
   \end{cases}
\end{equation}
We denote the set of loops at $o$ by 
$L_o(M):=\{\gamma\in C_o(M)\,;\,\gamma(1)=o\}$.
The {\em fundamental group\/} $\pi_1(M)$ is the set of homotopy classes
of $L_o(M)$ with group multiplication induced from \eqref{eq:join},
which acts on $\widetilde M$ as {\em covering transformations},
as follows:
\begin{equation}\label{eq:covering-transf}
     \tau\colon{}\widetilde M\ni [\gamma]\longrightarrow
             [\gamma_1*\gamma]\in\widetilde M
	     \qquad \bigl(\tau=[\gamma_1]\in\pi_1(M)\bigr),
\end{equation}
where $\gamma_1\in L_o(M)$ and $\gamma\in C_o(M)$.

An orientation-reversing conformal involution $\mu:M\to M$ is called a
{\it reflection\/} of $M$.
In this section, 
we want to define a reflection of $\widetilde{M}$
as a lift of $\mu$.
Let $P_{\mu}\colon{}[0,1]\to M$ be a path 
starting from $o$ and ending at $\mu(o)$.
We suppose that  $P_{\mu}$ is $\mu$-invariant, that is,
\begin{equation}\label{eq:mu-inv}
    \mu\circ P_{\mu} (u) = P_{\mu}(1-u)=P_{\mu}{}^{-1}(u)\qquad
     (u\in [0,1]).
\end{equation}
Now we define a map
$\tilde \mu:\widetilde M\longrightarrow \widetilde M$ by
\begin{equation}\label{eq:lift-refl}
     \tilde \mu([\gamma]):=
            [P_{\mu}*(\mu\circ \gamma)],
\end{equation}
which is called the {\it lift\/} of $\mu$ with respect to $P_{\mu}$.
Then the following assertion holds:
\begin{lemma}\label{lem:involution}
 The lift $\tilde \mu$ is an involution of $\widetilde M$.
\end{lemma}
\begin{proof}
 By \eqref{eq:mu-inv}, 
 $P(\mu)^{-1}* P(\mu)$
 is homotopic to the constant map $[0,1]\ni u \mapsto o\in M$.
 Then 
 $\tilde \mu\circ \tilde \mu([\gamma])=
          \tilde \mu([P_{\mu}*(\mu\circ \gamma)])
          =[P_{\mu}^{-1}* P_{\mu}*(\mu\circ \mu\circ \gamma)]
         =[\gamma]$.
\end{proof}

We now prove the following:
\begin{proposition}\label{prop:covering}
 Let $\mu_1$,\dots, $\mu_{2r}$
 be a sequence of reflections on $M$, and 
 take the lift $\tilde\mu_j$ of $\mu_j$ with 
 respect to the curve $P_{\mu_j}$ for each $j=1,\dots,2r$.
 Suppose that
 $\mu_1\circ \cdots \circ \mu_{2r}$ is the identity map $\identity_M$
 on $M$.
 Then
 \[
   \tilde \mu_1\circ \cdots \circ \tilde\mu_{2r}:
                   \widetilde M\longrightarrow \widetilde M
 \]
 gives a covering transformation on the universal covering 
 $\widetilde M$ of $M$ which corresponds to the loop
 \[
     P_{\mu_1}*(\mu_1\circ P_{\mu_2})*
     (\mu_1\circ\mu_2\circ P_{\mu_3})*
     \dots *
     (\mu_1\circ\mu_2\circ\dots\circ \mu_{2r-1}\circ P_{\mu_{2r}}).
 \]
\end{proposition}
We call
$\tilde \mu_1\circ \dots \circ \tilde\mu_{2r}$
the covering transformation associated with $\mu_1$,\dots, $\mu_{2r}$.
\begin{proof}[Proof of Proposition~\ref{prop:covering}]
 For the sake of simplicity, we write
 $P_j:=P_{\mu_j}$.
 Then for each $\gamma\in C_o(M)$, it holds that
 \[
    \tilde \mu_1\circ \dots \circ \tilde\mu_r([\gamma])
    =[P_1* (\mu_1\circ P_2) * \dots *
      (\mu_1\circ\dots\circ \mu_{2r-1}\circ P_{2r})*
     (\mu_1\circ\dots\circ\mu_{2r}\circ\gamma)].
 \]
 Since  $\mu_1\circ\dots\circ\mu_{2r}=\identity_M$, 
 we have the conclusion.
\end{proof}

\begin{remark}\label{rmk:covering}
 Proposition \ref{prop:covering} gives  
 a method to explicitly write down the isomorphism between 
 the covering transformation group
 and the fundamental group $\pi_1(M)$.
\end{remark}

\subsection{A certain analytic property of reflections}
Let $G$ be a meromorphic function and
$Q$ a holomorphic 2-differential on the Riemann surface $M$.
Such a pair $(G,Q)$ is called {\em admissible\/}
if 
\begin{equation}\label{eq:app-lift-metric}
   ds^2_{\#}:=\bigl(1+|G|^2\bigr)^2\left|\frac{Q}{dG}\right|^2
\end{equation}
gives a (positive definite) Riemannian metric on $M$.

Let $\pi\colon{}\widetilde M\to M$ be the universal covering as in 
Appendix \ref{sub:reflections}, and let $\tilde o\in\widetilde M$ be
the point
corresponding to the constant path at $o\in M$ 
(then $\pi(\tilde o)=o$ holds).
For an admissible pair $(G,Q)$ on $M$, we define
\[
    \widetilde G := G\circ\pi,\qquad
    \widetilde Q := Q\circ\pi.
\]
Then $(\widetilde G,\widetilde Q)$ is an admissible pair on 
$\widetilde M$ which is invariant under the covering transformations.
Consider the following ordinary differential equation
\begin{equation}\label{eq:app-ode}
   dF F^{-1}=\Psi,\qquad 
    \Psi:=
    \begin{pmatrix}
     \widetilde G & -\widetilde G^2 \\ 1 & -\widetilde G\hphantom{^{2}}
    \end{pmatrix}\frac{\widetilde Q}{d\widetilde G}
\end{equation}
with the initial condition
\begin{equation}\label{eq:app-ode-init}
   F(\tilde o)=a \in\SL_2\C.
\end{equation}
\begin{proposition}\label{prop:monodromy-ode}
 For each $a\in\SL_2\C$, there exists a unique holomorphic null
 immersion $F\colon{}\widetilde M\to \SL_2\C$ satisfying 
 \eqref{eq:app-ode} and \eqref{eq:app-ode-init}.
 For such an immersion $F$, there exists a representation
 $\rho^{}_F\colon{}\pi_1(M)\to \SL_2\C$ such that
 \[
     F\circ\tau = F\rho^{}_F(\tau)^{-1}\qquad
     \bigl(\tau\in\pi_1(M)\bigr),
 \]
 where $\tau$ is considered as a covering transformation.
 Moreover, we set a holomorphic function $g$ on $\widetilde M$
 as $g=-dF_{12}/dF_{11}=-dF_{22}/dF_{21}$  
 {\rm(}that is, $g$ 
 is the secondary Gauss map{\rm)}
 where $F=(F_{ij})$.
 Then it satisfies $g\circ\tau = \rho^{}_F(\tau)\star g$,
 where $\star$ denotes the M\"obius transformation{\rm:}
\begin{equation}\label{eq:app-moebius}
  a\star g := \frac{a_{11}g+a_{12}}{a_{21}g+a_{22}}\qquad
  \bigl(a=(a_{ij})\in\SL_2\C\bigr).
\end{equation} 
\end{proposition}
We call $\rho^{}_F$ the {\em representation associated to $F$}.
\begin{proof}[Proof of Proposition~\ref{prop:monodromy-ode}]
 By admissibility, $\Psi$ is an $\sl_2\!\C$-valued holomorphic one-form.
 Then the existence and uniqueness of $F$ follows.
 Since $\Psi$ is $\tau$-invariant, $F\circ\tau$ is also a solution 
 of \eqref{eq:app-ode}.
 Hence the existence of $\rho^{}_F$ follows.
 The final assertion can be shown directly.
\end{proof}

Let $\mu$ be a reflection on $M$.
Then an admissible pair $(G,Q)$ is said to be
{\em $\mu$-invariant\/}
if 
\begin{equation}\label{eq:mu-inv-data}
   ds^2_{\#}\circ\mu=ds^2_{\#}\qquad
    \text{and}\qquad
    \overline{Q\circ\mu}=Q
\end{equation}
hold, where $ds^2_{\#}$ is the metric as in \eqref{eq:app-lift-metric}.
\begin{lemma}[{See \cite[Lemma 4.2]{RUY}}]\label{lem:change-G}
 Let $(G,Q)$ be an $\mu$-invariant admissible pair on $M$,
 where $\mu$ is a reflection of $M$ whose fixed point set is not empty.
 Then there exists a matrix $\sigma(\mu)$ such that
 \begin{equation}\label{eq:change-G}
   \overline{G\circ \mu}=\sigma(\mu)\star G,
 \end{equation}
 where $\star$ denotes the M\"obius transformation in
 \eqref{eq:app-moebius}.
 Moreover, such a matrix $\sigma(\mu)$
 is unique up to $\pm$-ambiguity, and 
 is written in the following form{\rm:}
 \[
   \sigma(\mu)=\begin{pmatrix}
		q_\mu & i s_\mu\\
		is_\mu & \overline{q_\mu}
		\end{pmatrix},\quad
    \left(q_\mu\in\C,~s_\mu\in\R,~|q_\mu|^2+(s_\mu)^2=1\right).
 \]
 In particular $\overline{\sigma(\mu)}\sigma(\mu)=\id$ holds,
 where $\id$ is the identity matrix.
\end{lemma}
\begin{proof}
 By \eqref{eq:mu-inv-data}, the pull-back
 $ds^2_{FS}:=4|dG|^2/(1+|G|^2)^2$ of the Fubini-Study metric 
 of $\C\cup\{\infty\}$ by $G$ is $\mu$-invariant.
 Hence $\overline{G\circ\mu}$ is an orientation-preserving 
 developing map of $ds^2_{FS}$, as well as $G$.
 Then there exists $\sigma\in\SU_2$ such that
 $\overline{G\circ\mu}=\sigma\star G$.
 Since 
 $G = G\circ\mu\circ\mu=
     \overline{\sigma}\sigma\star G$,
 $\bar\sigma\sigma=\pm\id$ holds.
 If $\bar\sigma\sigma=-\id$ holds, 
 \[
     \sigma = \pm \begin{pmatrix} 0 & -1 \\
		                  1 & \hphantom{-}0
		  \end{pmatrix}
 \]
 because $\sigma\in\SU_2$.
 In this case, for a fixed point $z$ of $\mu$, 
 \[
    \overline{G(z)}=\overline{G\circ\mu(z)}=\sigma\star G(z) =
   -\frac{1}{G(z)},\qquad\text{that is,}\quad
   |G(z)|^2=-1
 \]
 holds, which is impossible.
 Then $\overline{\sigma}\sigma=\id$, and by a direct calculation
 we have the conclusion.
\end{proof}

\begin{theorem}\label{thm:app-refl-mat}
 Let $(G,Q)$ be a $\mu$-invariant admissible pair
 satisfying \eqref{eq:change-G}, and
 take a lift $\tilde\mu\colon{}\widetilde M\to\widetilde M$
 as in \eqref{eq:lift-refl}.
 Assume  $F:\widetilde M\to \SL_{2}\C$ satisfies \eqref{eq:app-ode}.
 Then there exists $\rho^{}(\tilde\mu)\in \SL_{2}\C$ such that
 \begin{equation}\label{eq:refl-mono}
   \overline{F\circ \tilde\mu}=\sigma(\mu) F \rho^{}(\tilde\mu)^{-1},  
 \end{equation}
 where $\sigma(\mu)$ is as in Lemma~\ref{lem:change-G}.
 Moreover, $\rho^{}(\tilde\mu)$ is written as 
 \begin{equation}\label{eq:rho-form}
   \rho^{}(\tilde\mu)=
   \begin{pmatrix}
      q & is_{1}\\
       is_{2} & \overline{q}
   \end{pmatrix}\qquad
   \left(
      q\in\C,~
      s_{j}\in\R~ (j=1,2),~
      |q|^2+s_{1}s_{2}=1
   \right).
 \end{equation}
\end{theorem}
\begin{proof}
 By \eqref{eq:mu-inv-data} and Lemma~\ref{lem:change-G}, 
 $\overline{\Psi\circ\mu}=\sigma\Psi\sigma^{-1}$ holds, where
 $\sigma=\sigma(\mu)$.
 Then
 \[
   d(\sigma F)(\sigma F)^{-1}=
       \sigma\Psi\sigma^{-1} = \overline{\Psi\circ\mu}
       = d\overline{(F\circ\tilde\mu)}(\overline{F\circ\tilde\mu})^{-1},
 \]
 which implies that
 $\sigma^{-1} F$ and $\overline{F\circ\tilde\mu}$ satisfy the same 
 equation.
 Thus, there exists $\rho^{}=\rho^{}(\tilde\mu)\in \SL_2\C$ such that
 $\overline{F\circ\tilde\mu}=\sigma F \rho^{-1}$.
 Since $\tilde\mu$ is an involution and $\overline{\sigma}\sigma=\id$,
 $\overline{\rho^{}}\rho^{}=\id$ holds.
 Noticing $\rho^{}\in\SL_2\C$, we have \eqref{eq:rho-form}.
\end{proof}

Finally, we write a representation $\rho^{}_F$ as in 
Proposition~\ref{prop:monodromy-ode} of the fundamental group
in terms of reflections.
\begin{definition}\label{def:generator}
 Let $\mu_1$,\dots, $\mu_N$ be mutually distinct reflections on $M$
 and take the lift $\tilde\mu_j$ of $\mu_j$ for $j=1,\dots,N$
 as in \eqref{eq:lift-refl}.
 If each covering transformation $\tau\in \pi_1(M)$
 has an expression
\begin{equation}\label{eq:refl-repr}
 \tau=\tilde{\mu}_{i_1}\circ\cdots \circ \tilde{\mu}_{i_{2r}},
\end{equation}
 then $\{\tilde\mu_1,\dots,\tilde\mu_N\}$
 is called a {\it generator\/} of $\pi_1(M)$.
\end{definition}
We now fix a generator
$\{\tilde\mu_1,\dots,\tilde\mu_N\}$ and take an admissible pair $(G,Q)$
on $M$ which is $\mu_j$-invariant for each $j=1,\dots,N$.
Choose $\sigma(\mu_j)$ as in Lemma~\ref{lem:change-G} 
for each $j=1,\dots,N$.
\begin{lemma}\label{lem:repr-sigma}
 If a covering transformation $\tau\in\pi_1(M)$ is written
 as in \eqref{eq:refl-repr} in terms of a generator,
 \[
   \overline{\sigma(\mu_{i_1})}\sigma(\mu_{i_2})\dots
    \overline{\sigma(\mu_{i_{2r-1}})}
    \sigma(\mu_{i_{2r}})
 \]
 is equal to $\id$ or $-\id$.
\end{lemma}
\begin{proof}
 Since $\mu_{i_1}\circ\dots\circ\mu_{i_{2r}}=\identity_M$,
 we have
 \[
   G = G\circ\mu_{i_1}\circ\dots\circ\mu_{i_{2r}}
     = 
    \bigl(
    \overline{\sigma(\mu_{i_1})}
    \sigma(\mu_{i_2})\dots
    \overline{\sigma(\mu_{i_{2r-1}})}
    \sigma(\mu_{i_{2r}})
    \bigr)\star G.
 \]
 Thus, we have the conclusion.
\end{proof}
Then, by the definition of the representation $\rho^{}_F$
in Proposition~\ref{prop:monodromy-ode}, we have
\begin{proposition}\label{prop:monodromy-refl}
 Under the situations above, we have
 \[
  \rho^{}_F(\tau) = 
    \biggl(
    \overline{\sigma(\mu_{i_1})}
    \sigma(\mu_{i_2})\dots
    \overline{\sigma(\mu_{i_{2r-1}})}
    \sigma(\mu_{i_{2r}})
    \biggr)
    \biggl(
    \overline{\rho^{}(\tilde\mu_{i_1})}
    \rho^{}(\tilde\mu_{i_2})\dots
    \overline{\rho^{}(\tilde\mu_{i_{2r-1}})}
    \rho^{}(\tilde\mu_{i_{2r}})
    \biggr).
 \]
\end{proposition}
Note that the $\pm$-ambiguity of $\sigma(\mu_j)$ 
does not affect this expression, because
if one were to choose $-\sigma(\mu_j)$ instead of $\sigma(\mu_j)$,
$\rho^{}(\tilde\mu_j)$ changes to $-\rho^{}(\tilde\mu_j)$.

Hence we have
\begin{theorem}\label{thm:su11-monodromy}
 Let $M$ be a Riemann surface with reflections $\{\mu_1,\dots,\mu_N\}$,
 and assume its lift $\{\tilde\mu_1,\dots,\tilde\mu_N\}$ is a generator
 of $\pi_1(M)$.
 Take an admissible pair $(G,Q)$ on $M$ which is $\mu_j$-invariant
 for each $j=1,\dots, N$, and 
 let $F$ be a solution of \eqref{eq:app-ode}, $\rho^{}_F$ a
 representation
 as in Proposition~\ref{prop:monodromy-ode},
 and $\rho^{}(\tilde\mu_j)$ $(j=1,\dots,N)$ as in 
 \eqref{eq:refl-mono}.
 Then,  
 if $\rho^{}(\tilde\mu_j)\in\SU_{1,1}$ holds for all $j=1,\dots,N$,
 the image $\rho^{}_F(\pi_1(M))$ lies in $\SU_{1,1}$.
\end{theorem}



\begin{thebibliography}{20}
\bibitem{ACM}
 L. Al\'\i{}as, R. M. B. Chaves and P. Mira, 
 {\itshape
	Bj\"orling problem for maximal surfaces in Lorentz-Minkowski
	space},
	Math.\ Proc.\ Cambridge Philos.\ Soc., 
	{\bfseries 134} (2003), no. 2, 289--316.
\bibitem{C}
 E. Calabi, 
 {\itshape
	Examples of Bernstein problems for some nonlinear equations}, 
	Proc.\ Symp.\ Pure Math., {\bfseries 15} (1970), 223--230. 

\bibitem{CY}
 S.-Y. Cheng and S.-T. Yau, 
 {\itshape
	Maximal space-like hypersurfaces in the Lorentz-Minkowski
	spaces}, 
	Ann.\ of Math., {\bfseries 104} (1976), 407--419.
\bibitem{ER}
 F. J. M. Estudillo and A. Romero,
 {\itshape 
	Generalized maximal surfaces in Lorentz-
	Minkowski space $L^3$}, 
	Math.\ Proc.\ Camb.\ Phil.\ Soc., 
	{\bfseries 111} (1992), 515--524.

\bibitem{Fer}
 I. Fern\'andez,
 {\itshape 
	The number of conformally equivalent maximal graphs},
	preprint,
	arXiv:0903.2950.

\bibitem{FL}
 I. Fern\'andez and F. J.  L\'opez, 
 {\itshape 
	Periodic maximal surfaces in the Lorentz-Minkowski space
	${\mathbb L}^3$},
	Math.\ Z., {\bfseries 256} (2007), 573--601.

\bibitem{FLS}
 I. Fern\'andez, F. J. L\'opez, and R. Souam,
 {\itshape 
	The space of complete embedded maximal
	surfaces with isolated singularities in $3$-dimensional 
	Lorentz-Minkowski space $L^3$}, 
	Math.\ Ann., {\bfseries 332} (2005), 605--643.

\bibitem{F}
 S. Fujimori,
 {\itshape 
	Spacelike CMC $1$ surfaces with
	elliptic ends in de Sitter $3$-Space},
	Hokkaido Math.\ J., 
	{\bfseries 35} (2006), 289--320.

\bibitem{F2}
  S. Fujimori,
  {\itshape 
	Spacelike CMC $1$ surfaces of genus $1$ with 
	two ends in de Sitter $3$-space},
	Kyushu J.\ Math., {\bfseries 61} (2007), 1--20.

\bibitem{FuL}
  S. Fujimori and F. J. L\'opez, 
  {\itshape 
	Nonorientable maximal surfaces in the Lorentz-Minkowski
	$3$-space}, 
	preprint, arXiv:0905.2113.

\bibitem{FR}
  S. Fujimori and W. Rossman,
  {\itshape 
	Higher genus mean curvature $1$ catenoids in 
	hyperbolic and de Sitter $3$-spaces},
	preprint,
	arXiv:0904.3988.

\bibitem{FRUYY}
 S. Fujimori, W. Rossman, M. Umehara, K. Yamada and S.-D. Yang,
 {\itshape
	Spacelike mean curvature one surfaces 
	in de Sitter $3$-space}, 
	Comm.\ Anal.\ Geom., 
	{\bfseries 17} (2009), 383--427.

\bibitem{FRUYY3}
 S. Fujimori, W. Rossman, M. Umehara, K. Yamada and S.-D. Yang, 
 {\itshape   
	Triply periodic maximal surfaces
        in Minkowski $3$-space}, 
	in preparation.

\bibitem{FSUY}
 S. Fujimori, K. Saji, M. Umehara and K. Yamada,
 {\itshape 
	Singularities of maximal surfaces},
	Math.\ Z., {\bfseries 259} (2008), 827--848.

\bibitem{Imaizumi}
 T. Imaizumi,
 {\itshape
	Maximal surfaces with simple ends},
	Kyushu J.\ Math., {\bfseries 58} (2004), 59--70. 

\bibitem{IK}
 T. Imaizumi and S. Kato,
 {\itshape
	Flux of simple ends of maximal surfaces in $\R^{2,1}$}
	Hokkaido Math.\ J., 
	{\bfseries 37} (2008), 561--610.


\bibitem{Ka}
 S. Kato, personal communication.

\bibitem{K}
 O. Kobayashi,
 {\itshape 
	Maximal surfaces in the $3$-dimensional
        Minkowski space $\mathbb{L}^3$},
	Tokyo J.\ Math., {\bfseries 6} (1983), 297--309.

\bibitem{K2}
 O. Kobayashi,    
 {\itshape 
	Maximal surfaces with conelike singularities},
	J.\ Math.\ Soc.\ Japan, 
	{\bfseries 36} (1984), 609--617.

\bibitem{KU}
 M. Kokubu and M. Umehara, 
 {\itshape 	
	Orientability of linear Weingarten surfaces, spacelike CMC-$1$
	surfaces and maximal surfaces}.
	preprint, arXiv:0907.2284.

\bibitem{KUY}
 M. Kokubu, M. Umehara and K. Yamada,
 {\itshape
	An elementary proof of Small's formula
	for null curves in $\PSL(2,\C)$ and an analogue
	for Legendrian curves in $\PSL(2,\C)$},
	Osaka J.\ of Math., {\bfseries 40} (2003), 697--715.

\bibitem{KY}
 Y. W. Kim and S.-D. Yang,
 {\itshape
	A family of maximal surfaces in Lorentz-Minkowski three-space},
	Proc.\ Amer.\ Math.\ Soc.,
	{\bfseries 134} (2006), 3379--3390.

\bibitem{KY2}
 Y. W. Kim and S.-D. Yang,
 {\itshape
	Prescribing singularities of maximal surfaces via a singular
	Bj\"orling representation formula},
	J.\ Geom.\ Phys., 
	{\bfseries 57} (2007), no. 11, 2167--2177. 



\bibitem{LLS}
 F. J. L\'opez F.J., R. L\'opez  and R. Souam, 
 {\itshape  
	Maximal surfaces of Riemann type in Lorentz-Minkowski 
	space $L^3$}, 
	Michigan J.\ Math., 
	{\bfseries 47} (2000), 469--497.

\bibitem{MUY1}
 F. Mart\'\i{}n, M. Umehara and K. Yamada,
 {\itshape 
	Complete bounded null curves immersed in 
	$\C^3$ and $\PSL(2,\C)$}, 
	Calc.\ Var.\ Partial Differential Equations, 
	{\bfseries 36} (2009), 119--139. 

\bibitem{MUY2}
 F. Mart\'\i{}n, Umehara and K. Yamada,
 {\itshape 
	Complete bounded holomorphic curves 
	immersed in $\C^2$ with arbitrary genus},
	Proc.\ Amer.\ Math.\ Soc.,
        {\bfseries 137}  (2009),  no. 10, 3437--3450.
\bibitem{RUY}
 W.~Rossman, M.~Umehara and K.~Yamada,
 {\itshape 
	Irreducible constant mean curvature $1$ surfaces in
	hyperbolic space with positive genus},
	T\^ohoku Math.~J., 
	{\bfseries 49}  (1997), 449--484.

\bibitem{Sa}
 K.~Sato,
 {\itshape 
	Construction of higher genus minimal surfaces with one end
	and finite total curvature},
	T\^ohoku Math.~J., 
	{\bfseries 48}  (1996), 229--446.

\bibitem{S}
 R. Schoen, 
 {\itshape
	Uniqueness, symmetry and embeddedness of minimal surfaces},
	J.\ Differential Geometry, 
	{\bfseries 18} (1983), 791--809.

\bibitem{UY-deform}
 M. Umehara and K. Yamada,
 {\itshape 
	A parametrization of the Weierstrass formulae and perturbation
	of some complete minimal surfaces of $\mathbf R^3$ into the
	hyperbolic $3$-space},
	J. f\"ur reine u.\ angew.\ Math., 
	{\bfseries 432} (1992),	93--116.

\bibitem{UY5}
 M. Umehara and K. Yamada,
 {\itshape
	Duality on CMC-$1$ surfaces in hyperbolic $3$-space, 
	and hyperbolic analogue of the Osserman inequality},
	Tsukuba J.\ Math., 
	{\bfseries 21} (1997), 229--237.

 	
\bibitem{UY}
 M. Umehara and K. Yamada,
 {\itshape
	Maximal surfaces with singularities in
	Minkowski space},
	Hokkaido Math.\ J., {\bfseries 35} (2006), 13--40.


\bibitem{UY2}
 M. Umehara and K. Yamada,
 {\itshape
	Applications of a completeness lemma in minimal surface theory
	to various classes of surfaces},
	preprint, arXiv:0909.1128.

\end{thebibliography}
\end{document}